%% file: main.tex
\documentclass[letterpaper, 11pt]{article}
\usepackage[utf8]{inputenc}
\usepackage[T1]{fontenc}
\usepackage{lmodern}
\usepackage{amsmath, amsfonts, amssymb, amsthm, mathrsfs, mathtools}
\usepackage{dsfont}
\usepackage{graphicx}
\usepackage[overlap]{overpic}
\usepackage{float}
\usepackage{xcolor}
\definecolor{green1}{rgb}{0.,0.4,0.}
\definecolor{blue1}{rgb}{0.,0.,0.8}
\definecolor{red1}{rgb}{0.8,0.,0.}
\usepackage[linktocpage=true]{hyperref}
\usepackage{caption}
\usepackage{subcaption}
\usepackage{tabularx}

\usepackage{microtype}
\usepackage{comment}
\usepackage{enumitem}
\usepackage{tikz-cd}
\usetikzlibrary{arrows, patterns, shapes, arrows.meta, patterns.meta}
\usepackage{pgfplots}
\pgfplotsset{compat=1.18}
\usepackage{standalone}
\usepackage{cleveref}
\usepackage[size=tiny]{todonotes}

\newcommand{\cF}{\mbox{${\mathcal F}$}}
\newcommand{\cW}{\mbox{${\mathcal W}$}}

\newcommand{\mt}{\mbox{${\widetilde M}$}}

\newcommand{\RR}{\mbox{${\mathbb R}$}}

\newcommand{\eps}{\mbox{${\epsilon}$}}

\usepackage[margin=1.2in]{geometry}

\usepackage[backend=biber, maxnames=4, style=alphabetic, sorting=nyt]{biblatex}
\hypersetup{
  colorlinks,
  citecolor=teal,
  linkcolor=purple,
  urlcolor=teal}
  
\newtheorem{thm}{Theorem}[section]
\newtheorem*{thm*}{Theorem}
\newtheorem{cor}[thm]{Corollary}
\newtheorem{lem}[thm]{Lemma}

\newtheorem{prop}[thm]{Proposition}
\newtheorem{defn}[thm]{Definition}
\newtheorem{rem}[thm]{Remark}
\newtheorem{thmintro}{Theorem}

\newtheorem*{cor*}{Corollary}
\newtheorem{thmintrobeta}{Theorem}
\newtheorem{propintro}[thmintrobeta]{Proposition}
\newtheorem{defnintro}[thmintrobeta]{Definition}
\newtheorem{constr}[thmintrobeta]{Construction}
\newtheorem{corintro}[thmintrobeta]{Corollary}
\newtheorem{exintro}[thmintrobeta]{Example}
\newtheorem{remintro}[thmintrobeta]{Remark}

\newtheorem{questionintro}{Question}

\newcommand{\R}{\mathbb{R}}

\title{Topological invariance of Liouville structures for taut foliations and Anosov flows}

\author{ 
Jonathan Bowden\thanks{Department of Mathematics, Leibniz Universit\"at Hannover, Germany. Email address: \href{mailto:jonathan.bowden@math.uni-hannover.de}{\texttt{jonathan.bowden@math.uni-hannover.de}}. Website: \url{https://sites.google.com/view/jpbowden/}.}
\and
Thomas Massoni\thanks{Department of Mathematics, Stanford University, Stanford, USA. Email address: \href{mailto:tmassoni@stanford.edu}{\texttt{tmassoni@stanford.edu}}. Website: \url{https://sites.google.com/view/thomasmassoni/}.}
\and
\normalsize
with an appendix by Thomas Barthelm\'{e}, S\'{e}rgio R. Fenley, and Rafael Potrie}

\date{\today}

\begin{document}

\maketitle

\begin{abstract}
Building on the work of Eliashberg and Thurston, we associate to a taut foliation on a closed oriented $3$-manifold $M$ a Liouville structure on the thickening $[-1,1] \times M$, under suitable hypotheses. Our main result shows that this Liouville structure is a \emph{topological invariant} of the foliation: two such foliations which are \emph{topologically conjugate} induce Liouville structures that are exact symplectomorphic (after completion). Specializing to the case of weak foliations of Anosov flows, we obtain that under natural orientability conditions, the Liouville structures originally introduced by Mitsumatsu are invariant under \emph{orbit equivalence}. Our methods also imply that two orbit equivalent Anosov flows are \emph{deformation equivalent} through projectively Anosov flows. 
The proofs combine two main technical ingredients: (1) a careful smoothing scheme for topological conjugacies between $C^1$-foliations, and (2) a refinement of a deep result of Vogel on the uniqueness of contact structures approximating a foliation.

In an appendix, this smoothing scheme is used to construct new examples of \emph{collapsed Anosov flows}, providing a key step to complete the classification of transitive partially hyperbolic diffeomorphisms in dimension three.
\end{abstract}

\tableofcontents

\section{Introduction}

    \subsection{Context}

Anosov flows were introduced by Anosov~\cite{A69} as a generalization of geodesic flows on hyperbolic manifolds. They exhibit remarkable properties, like \emph{structural stability}---$C^1$-small perturbations yield flows which are still Anosov and conjugate via a homeomorphism, up to reparametrization. In this case, one says that the flows are \emph{topologically equivalent} or \emph{orbit equivalent}. Thus, the qualitative study of Anosov flows, or hyperbolic systems in general, seeks to classify them up to $C^0$-equivalence, and one wishes to associate invariants that behave well under such equivalences. Given the lack of smoothness of orbit equivalences, this is in general a very subtle problem.

In dimension $3$, the theory of Anosov flows reveals intricate connections between the dynamical properties of the flow and its closed orbits, and the topology of the underlying manifold. There is a well-developed structural framework initiated by Fenley and Barbot, which draws its richness from the plethora of examples arising via various surgery and gluing constructions. This analysis essentially studies the flow (up to topological equivalence) by considering the weak-stable and weak-unstable foliations as the fundamental objects. This way, one can attach invariants to the flow by considering invariants of its weak foliations. 

\paragraph{Invariants from contact topology.} One such invariant arises by considering the \emph{bicontact structure} given by a pair of transverse contact distributions that are tangent to the flow, but nowhere tangent to the stable or unstable directions, as introduced by Mitsumatsu~\cite{Mit95} and Eliashberg--Thurston~\cite{ET}. In fact, one can further consider a \emph{Liouville structure} on the thickening $[-1,1] \times M$ of the underlying $3$-manifold $M$, whose deformation class is also an invariant of the flow up to \emph{smooth} deformation equivalence~\cite{Mas25a}. Informally speaking, this Liouville structure combines the data of the aforementioned bicontact structures together with some information about how they interact---it notably ``detects'' the closed orbits of the flow as particular exact Lagrangians which are studied in~\cite{CLMM}. We note that invariants from contact and symplectic geometry have already proved useful in the study of special classes of Anosov flows; see~\cite{BM24}.

Since the invariants of interest arise from approximations of cooriented codimension-$1$ foliations, or in fact transverse pairs of such foliations, we will develop a more general approach and show a certain form of functoriality for contact approximations under homeomorphisms. Vogel~\cite{V16} showed that the contact structure approximating a $C^2$-foliation is well-defined up to isotopy, provided that some natural and necessary conditions hold; our main result will show that these smooth invariants behave well under topological transformations, even for suitable $C^1$-foliations.

\paragraph{New constructions of partially hyperbolic systems.} One key technical step (see~\Cref{thmintrobeta:approx} below) is to approximate a $C^0$-conjugation between $C^1$-foliations by diffeomorphisms with control on the `distortion' of the tangent planes of the foliations. A byproduct of this approximation scheme is a construction of new examples of \emph{partially hyperbolic systems} in dimension $3$, which in turn concludes an ongoing program to establish a topological classification of transitive partially hyperbolic diffeomorphisms in dimension $3$. The construction of these examples, featuring the first instances of anomalous partially hyperbolic diffeomorphisms isotopic to the identity (the so-called \emph{double translations}) appears in an appendix to this paper, which is written by Barthelm\'{e}, Fenley, and Potrie. 

\begin{center}
\textit{\textbf{Standing assumptions.} In this paper, $M$ denotes a smooth, closed, oriented, connected $3$-manifold. All the structures under consideration (foliations, plane fields, contact structures) will be assumed to be (co)orientable, and even (co)oriented when necessary. Furthermore, all foliations will be of class $C^1$, unless specified otherwise.}
\end{center}

    \subsection{Liouville structures arising from foliations}

To study weak foliations of Anosov flows and invariants thereof, the natural class of foliations to consider is that of taut $C^1$-foliations. Hereafter, we will consider foliations of class $C^1$, in the sense that they are defined by a $C^1$ foliated atlas; see~\cite{CC00} for the precise definition. In particular, the leaves are $C^1$-immersed and the tangent distribution is $C^0$ and \emph{uniquely} integrable. We will even consider a special class of taut foliations that are called \emph{hypertaut} in~\cite{Mas24}:

\begin{defnintro}\label{def:hypertaut}
    A cooriented $C^1$-foliation is \textbf{hypertaut} if there exists an exact $2$-form positive which evaluates positively along its leaves.
\end{defnintro}

The condition above might seem somewhat contrived; for instance, it immediately implies that such a foliation has no closed leaves by Stokes' Theorem, and is hence automatically \emph{taut} by Goodman~\cite{G75}. However, by Sullivan's results on foliation cycles~\cite{S76} (see also Candel--Conlon~\cite{CC00}), it is equivalent to the nonexistence of (nontrivial) holonomy invariant transverse measures. In particular, this notion is \emph{topological}, meaning that it is invariant under $C^0$-conjugation.

Furthermore, by a result of Bonatti--Firmo~\cite{BF94}, this condition holds for a generic taut $C^{\infty}$-foliation on atoroidal, non-fibered $3$-manifolds. In view of Gabai's work~\cite{G83}, this implies that any hyperbolic non-fibered $3$-manifold with positive first Betti number has a hypertaut foliation.

If $\mathcal{F}$ is a hypertaut foliation, then every pair of contact structures with opposite signs $(\xi_-, \xi_+)$ approximating $\mathcal{F}$ is Liouville fillable: there exists a Liouville structure on $[-1,1] \times M$ which induces $\xi_\pm$ on $\{\pm1\} \times M$; see~\cite[Proposition 4.4]{Mas24}. We now describe this in somewhat more detail.

\begin{constr} \label{construction}
    Let $\mathcal{F}$ be a hypertaut $C^1$-foliation and let $\beta$ be a smooth $1$-form such that $d\beta_{\vert T \mathcal{F}} > 0$. In particular, $\mathcal{F}$ has no closed leaves and is not the standard foliation by spheres on $S^1 \times S^2$. By Eliashberg--Thurston~\cite{ET} (or rather its generalization to foliations with lower regularity by the first author~\cite{B16} and independently in~\cite{KR17}), there exists an approximating contact pair $(\xi_-, \xi_+)$ such that $d\beta_{\vert \xi_\pm} > 0$. If $\alpha$ is a continuous $1$-form such that $\ker \alpha = T \mathcal{F}$ as cooriented plane fields, then $\alpha \wedge d\beta > 0$. We consider a smoothing $\widetilde{\alpha}$ of $\alpha$ satisfying $\widetilde{\alpha} \wedge d\beta > 0$ and an $\epsilon > 0$ to be chosen small enough, and we define a $1$-form
    $$\lambda \coloneqq \beta + \epsilon t \widetilde{\alpha} $$
    on $V \coloneqq [-1,1]_t \times M$. Then it is easy to check that $\omega \coloneqq d\lambda$ is symplectic, and for $\epsilon$ small enough, $\omega$ is positive on $\xi_\pm$ along $\{ \pm 1\} \times M$. In other words, $(V,\omega)$ is a \emph{weak symplectic filling} of $(-M,\xi_-) \sqcup (M, \xi_+)$, which is moreover exact. A result of Eliashberg~\cite{E04} (see also Lemma~\ref{lem:filling} below) implies that $\lambda$ can be modified near $\partial V$ into a \textbf{Liouville filling} of $(-M,\xi_-) \sqcup (M, \xi_+)$ in a unique way up to homotopy (see Proposition~\ref{prop:deffill}). We will refer to the resulting Liouville structure as a \textbf{Liouville thickening} of $\mathcal{F}$.
\end{constr}

While this construction seemingly depends on the contact approximations $\xi_\pm$, it can be shown to be independent of the choices of $\widetilde{\alpha}$, $\beta$, and $\epsilon$ (provided that $\epsilon$ is small enough), up to Liouville homotopy. In order to remove the dependence on the choice of contact approximations, we will use that hypertaut $C^1$-foliations have ``enough attracting holonomy'':

\begin{propintro} \label{propintro:enoughholo}
    Let $\mathcal{F}$ be a cooriented $C^1$-foliation on $M$. Then $\mathcal{F}$ is hypertaut if and only if it has no closed leaves, and every minimal set (closed set saturated by leaves) has a curve with attracting holonomy.
\end{propintro}

This proposition is a consequence of a deep result of Deroin--Kleptsyn--Navas~\cite{DKN07}; see Section~\ref{sec:hypertaut} below for a proof. It is an extension to $C^1$-foliations of a classical result of Sacksteder, which holds for $C^2$-foliations with holonomy; the additional hypertautness hypothesis compensates for the lower regularity.

In general, the contact approximations to foliations depend on various choices, and uniqueness can fail for general $C^1$-foliations, especially in the presence of closed leaves. On the other hand, for hypertaut foliations of class at least $C^2$, they are unique by~\cite{V16}, and the Liouville structure described in Construction~\ref{construction} does not depend on the choices made in the construction, up to homotopy. We will extend this result to hypertaut $C^1$-foliations in Proposition~\ref{prop:welldef} below.

\begin{exintro}
    The main examples of hypertaut foliations we will consider are the following.
    \begin{itemize}
        \item The weak foliations of a (smooth) Anosov flow on $M$, when coorientable, are hypertaut.
        \item Any taut $C^1$-foliation on a rational homology sphere is hypertaut.
        \item Any taut $C^2$-foliation without closed leaves on a non-fibered $3$-manifold is hypertaut.
    \end{itemize}
    See Lemma~\ref{lem:hypertaut} and Proposition~\ref{prop:Anosovhypertaut} below for a proof.
\end{exintro}

For a hypertaut $C^1$-foliation $\mathcal{F}$, we denote by $\lambda_\mathcal{F}$ a/the Liouville thickening of $\mathcal{F}$ on $V = [-1,1] \times M$. Our main result is:

\begin{thmintro}[$C^0$-naturality] \label{thmintro:liouv}
    Let $\mathcal{F}_0$ and $\mathcal{F}_1$ be homeomorphic hypertaut $C^1$-foliations, via a homeomorphism preserving the coorientations of the foliations. Then $\lambda_{\mathcal{F}_0}$ and $\lambda_{\mathcal{F}_1}$ are deformation equivalent. More precisely, if $h : (M, \mathcal{F}_0) \rightarrow (M, \mathcal{F}_1)$ is such a homeomorphism, then $h$ is isotopic to a smooth diffeomorphism $\widetilde{h} : M \rightarrow M$ such that $\big(\mathrm{id} \times \widetilde{h}\big)_* \lambda_{\mathcal{F}_0}$ and $\lambda_{\mathcal{F}_1}$ are homotopic Liouville structures.
\end{thmintro}

In particular, the \emph{completions} of the Liouville domains associated with these foliations are \emph{exact symplectomorphic}; see~\cite{CE12}. As a consequence, all Floer-type invariants of a hypertaut foliation defined through its Liouville thickening are invariant under topological equivalence. We remark that a special case of this result was already obtained for weak foliations of Reeb Anosov flows in~\cite{BM24}.

\bigskip

The proof of our main theorem has three key ingredients which are completely independent of each other.
\begin{itemize}
    \item The first ingredient is a careful smoothing/approximation result of the homeomorphism that proceeds via induction over a fine triangulation, jiggled into general position.
    \item The second ingredient is a refinement of the main result of~\cite{V16} on the uniqueness of contact approximations of hypertaut foliations, with $C^1$ regularity, and with some additional transverse control on the resulting contact homotopies.
    \item The last ingredient is a generalization of a classical argument of Eliashberg to deform the symplectic structure near the boundary of $V$ into a Liouville structure, together with a parametric and relative version thereof.
\end{itemize}  

Our strategy also provides a variation on the contact approximations of hypertaut foliations. Recall that a \textbf{positive contact pair} is a pair of cooriented contact structures $(\xi_-, \xi_+)$ admitting a common positively transverse vector field; see~\cite{CF11}. The Eliashberg--Thurston theorem readily provides positive contact pairs approximating foliations, and the second author showed in~\cite{Mas24} that one can construct $C^0$-foliations (but not necessarily $C^1$!) from (tight) positive contact pairs.

\begin{thmintro} \label{thmintro:poscont}
    Let $\mathcal{F}_0$ and $\mathcal{F}_1$ be two homeomorphic hypertaut $C^1$-foliations, and $(\xi^0_-, \xi^0_+)$ and $(\xi^1_-, \xi^1_+)$ be positive contact pairs sufficiently $C^0$-close to $\mathcal{F}_0$ and $\mathcal{F}_1$, respectively. Then $\big(\xi^0_-, \xi^0_+\big)$ and $\big(\xi^1_-, \xi^1_+\big)$ are deformation equivalent through positive contact pairs. More precisely, if $h : (M, \mathcal{F}_0) \rightarrow (M, \mathcal{F}_1)$ is such a homeomorphism, then $h$ is isotopic to a smooth diffeomorphism $\widetilde{h} : M \rightarrow M$ such that $\big(\widetilde{h}_*(\xi^0_-), \widetilde{h}_* (\xi^0_+)\big)$ and $\big(\xi^1_-, \xi^1_+\big)$ are homotopic through positive contact pairs.
\end{thmintro}

We now discuss the main steps of our strategy in more detail.

\subsubsection{Smoothing foliated homeomorphisms}

We fix some auxiliary Riemannian metric on $M$, which induces a natural metric on the spaces of (continuous) plane fields and line fields on $M$.

The first ingredient is a careful smoothing result for the topological conjugation $h$. Namely, we approximate $h$ by a smooth diffeomorphism while keeping some control on the plane fields tangent to the foliations:

\begin{thmintrobeta} \label{thmintrobeta:approx}
    Let $\mathcal{F}_0$ and $\mathcal{F}_1$ be two coorientable $C^1$-foliations on $M$, and $h : M \rightarrow M$ be a homeomorphism sending the leaves of $\mathcal{F}_0$ to leaves of $\mathcal{F}_1$. For every $\epsilon > 0$, there exists a smooth diffeomorphism $\widetilde{h} : M \rightarrow M$ such that
    $$d_{C^0}(h, \widetilde{h}) < \epsilon, \qquad d_{C^0}\big(T \mathcal{F}_1, T\widetilde{\mathcal{F}}_1 \big) < \epsilon,$$
    where $\widetilde{\mathcal{F}}_1 \coloneqq \widetilde{h}_*(\mathcal{F}_0)$. Moreover, $h$ and $\widetilde{h}$ are isotopic through homeomorphisms which are $\epsilon$-close to $h$.\footnote{This would follow from the fact that the homeomorphism group of $M$ is locally path-connected, but it easily holds by construction.}
\end{thmintrobeta}

We remark that as a consequence of the proof, one could also obtain a `local' uniqueness statement: any two such smoothings differ by some smooth isotopy which induces a path of foliations with tangent plane fields close to $T\mathcal{F}_1$.

Our method can be adapted to \emph{pairs} of transverse foliations. This will be relevant for approximating orbit equivalences between Anosov flows via suitable smooth diffeomorphisms.

\begin{defnintro}
    A bifoliation $(\mathcal{F}, \mathcal{G})$ on $M$ is a pair of transverse $C^1$-foliations. It is orientable if both $\mathcal{F}$ and $\mathcal{G}$ are orientable.
\end{defnintro}

A smooth Anosov flow on $M$ induces a $C^1$-bifoliation $(\mathcal{F}^{ws}, \mathcal{F}^{wu})$ obtained from the weak-stable and weak-unstable foliations of the flow. This bifoliation is not necessarily orientable, but we will assume orientability throughout; this can always be achieved after passing to a suitable finite cover.

Let $(\mathcal{F}_0, \mathcal{G}_0)$ and $(\mathcal{F}_1, \mathcal{G}_1)$ be two bifoliations on $M$. We now consider \emph{bifoliated homeomorphisms}, where a homeomorphism $h : M \rightarrow M$ is bifoliated if it sends the leaves of $\mathcal{F}_0$ to leaves of $\mathcal{F}_1$, and the leaves of $\mathcal{G}_0$ to leaves of $\mathcal{G}_1$.

\begin{thmintrobeta} \label{thmintrobeta:bifolapprox}
    Let $(\mathcal{F}_0, \mathcal{G}_0)$ and $(\mathcal{F}_1, \mathcal{G}_1)$ be orientable $C^1$-bifoliations on $M$, and let $h : M \rightarrow M$ be a bifoliated homeomorphism. For every $\epsilon > 0$, there exists a smooth diffeomorphism $\widetilde{h} : M \rightarrow M$ satisfying
        $$d_{C^0}(h, \widetilde{h}) < \epsilon, \qquad 
        d_{C^0}\big(T \mathcal{F}_1, T\widetilde{\mathcal{F}}_1 \big) < \epsilon, \qquad 
        d_{C^0}\big(T \mathcal{G}_1, T\widetilde{\mathcal{G}}_1 \big) < \epsilon,$$
    where $\widetilde{\mathcal{F}}_1 = \widetilde{h}_*(\mathcal{F}_0)$ and $\widetilde{\mathcal{G}}_1 = \widetilde{h}_*(\mathcal{G}_0)$. Moreover, $h$ and $\widetilde{h}$ are isotopic through homeomorphisms which are $\epsilon$-close to $h$.
\end{thmintrobeta}

A key property of the approximation above is that the line fields $T\mathcal{F}_1 \cap T\mathcal{G}_1$ and $T\mathcal{\widetilde{F}}_1 \cap T\mathcal{\widetilde{G}}_1 = \widetilde{h}_*\big(T\mathcal{F}_0 \cap T\mathcal{G}_0\big)$ are also $\epsilon$-close. For dynamical applications, one has the following consequence which will be used in the construction of new partially hyperbolic diffeomorphisms.

\begin{corintro}[Anosov bifoliations] \label{corintro:smoothing}
    Let $\Phi_0$ and $\Phi_1$ be two smooth Anosov flows on $M$ with orientable weak invariant bundles $E_i^{{wu}/{ws}}$, $i \in \{0,1\}$. If $h : M \rightarrow M$ is an orbit equivalence between $\Phi_0$ and $\Phi_1$, then for every $\epsilon > 0$, there exists a smooth diffeomorphism $\widetilde{h} : M \rightarrow M$ such that $d_{C^0}(h, \widetilde{h}) < \epsilon$ and
    \begin{itemize}
        \item The plane fields $\widetilde{h}_*\big(E_0^{ws}\big)$ and $E_1^{ws}$ are $\epsilon$-close,
        \item The plane fields $\widetilde{h}_*\big(E_0^{wu}\big)$ and $E_1^{wu}$ are $\epsilon$-close.
    \end{itemize}
    As a consequence, the line fields tangent to $\Phi_1$ and $\widetilde{h}_*(\Phi_0)$ are $\epsilon$-close.
\end{corintro}

In particular, if $\Phi$ is a single Anosov flow with orientable weak foliations, and $\beta$ is a self orbit equivalence of $\Phi$, then a smoothing $\widetilde{\beta}$ of $\beta$ obtained in this way, for $\epsilon$ small enough, satisfies that \emph{$\Phi$ is $\widetilde{\beta}$-transverse to itself} in the terminology of~\cite{BFP}. In Appendix~\ref{appendixB} written by Thomas Barthelm\'{e}, S\'{e}rgio Fenley, and Rafael Potrie, this result will be used to solve an important problem in the classification of \emph{partially hyperbolic diffeomorphisms} on $3$-manifolds.

\subsubsection{Uniqueness of contact approximations}

In general the contact structure approximating a foliation is not unique, as one sees by approximating a product foliation of the $3$-torus by a contact structure with (arbitrary) Giroux torsion. However, excluding this and a few other exceptional cases, Vogel was able to obtain the following uniqueness statement.

\begin{thm*}[Vogel~\cite{V16}]
    Let $\mathcal{F}$ be a coorientable $C^2$-foliation on a closed oriented $3$-manifold satisfying the following conditions: 
    \begin{enumerate}
    \item $\mathcal{F}$ has no closed leaf of genus $g \le 1$,
    \item $\mathcal{F}$ is not a foliation by planes, 
    \item $\mathcal{F}$ is not a foliation by cylinders. 
    \end{enumerate}
    Then there is a $C^0$-neighborhood $\mathcal{V}$ of $\mathcal{F}$ in the space of plane fields and a contact structure $\xi$ in $\mathcal{V}$ such that every positive contact structure in $\mathcal{V}$ is isotopic to $\xi$.
\end{thm*}

Unfortunately, the theorem \emph{does not} guarantee that the path of contact structures remains within $\mathcal{V}$, see Figure~\ref{fig:Vogelnhbd}.

\begin{figure}[!ht]
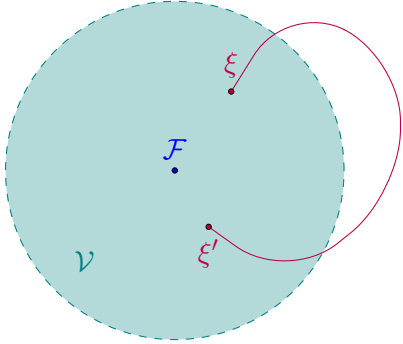

    \centering
    \includestandalone{tikz/path}
    \caption{Summary of Vogel's theorem.}
    \label{fig:Vogelnhbd}
\end{figure}

Note also that all the exceptional cases above imply that the foliation has a (nontrivial) transverse invariant measure, and hence are excluded if the foliation is hypertaut. 

We now assume that $\mathcal{F}$ is hypertaut and $C^1$, and we fix a smooth $1$-dimensional foliation $\mathcal{I}$ transverse to $\mathcal{F}$. We shall need a refinement of Vogel's result, which is stated in his paper, although several steps are not worked out in detail there. To this end, let $\mathcal{P}_\mathcal{I} \subset \mathcal{P}$ denote the space of oriented plane fields on $M$ transverse to $\mathcal{I}$. 

\begin{thmintrobeta}[Theorem~\ref{thm:uniq}] \label{thmintrobeta:uniq}
    There exists a $C^0$-neighborhood $\mathcal{V} =  \mathcal{V}_\mathcal{I} \subset \mathcal{P}_\mathcal{I}$ of $T \mathcal{F}$ such that any two positive (resp.~negative) contact structures in $\mathcal{V}$ are contact homotopic through a path of contact structures \underline{within $\mathcal{P}_\mathcal{I}$}.
\end{thmintrobeta}

Note that the neighborhood $\mathcal{V}$ depends on the choice of the transverse foliation $\mathcal{I}$ a priori, but we omit this dependence in the interest of notation economy.

\subsubsection{Deformation of \emph{exact} weak symplectic fillings}

A \textbf{pre-Liouville structure} on a compact $4$-manifold $V$ with contact boundary is a pair $(\lambda, \xi)$, where $\lambda \in \Omega^1(V)$ and $\xi$ is a contact structure on $\partial V$, such that $d \lambda$ is symplectic and dominates $\xi$ along $\partial V$, namely, $d\lambda_{\vert \xi} > 0$. Such fillings are sometimes called \emph{weakly exact} in the literature. In many situations, one is naturally led to consider pre-Liouville structures, which are somewhat more flexible than actual Liouville structures since the condition at the boundary is relaxed. However, one can always deform a pre-Liouville structure near $\partial V$ to become Liouville, without modifying the underlying contact structure. This operation yields a global primitive of a (different yet homotopic) symplectic structure which restricts to a contact form along the boundary. This procedure also extends to deformations, and we will need the following:

\begin{propintro}[Proposition~\ref{prop:deffill}] \label{propintro:deffill}
    Let $V$ be a $4$-dimensional compact manifold with boundary, and $(\lambda_t, \xi_t)_{t \in [0,1]}$ be a path of pre-Liouville structures on $V$. Assume that for $i \in \{0,1\}$, $(\lambda_i, \xi_i)$ is a Liouville structure: $\lambda_i$ is a Liouville form and $\ker {\lambda_i}_{\vert \partial V} = \xi_i$. Then $\lambda_0$ and $\lambda_1$ are Liouville homotopic.
\end{propintro}

In Section~\ref{sec:Def_weak}, we will state a more general result that shows that the natural `forgetful map' from the space of Liouville structures on $V$ to the space of pre-Liouville structures is a (weak) homotopy equivalence, which might be of independent interest.

    \subsection{Applications to Anosov flows}

We now focus on the case of weak foliations of Anosov flows.

    \subsubsection{Anosov Liouville structures}\label{subsec:Anosov_Liouville}

Following~\cite{Mit95, Hoz24}, one can associate to any Anosov flow on $M$ with oriented weak invariant bundles a Liouville pair on $V = [-1,1] \times M$, in the sense of~\cite{MNW13}. The properties of these Liouville pairs were also studied in~\cite{Mas25a, Mas24}. In particular, this endows $V$ with the structure of a Liouville domain with convex (disconnected) boundary which we call an \emph{Anosov Liouville domain}. Moreover, the contact structures on the boundary components can be identified with a supporting bicontact structure.

Let $\Phi_0$ and $\Phi_1$ be two oriented Anosov flows on $M$, and assume that $\Phi_0$ and $\Phi_1$ are orbit equivalent, via an orbit equivalence $h : M \rightarrow M$ preserving the orientations on the weak foliations of the flows. For simplicity, we will say that the flows are \emph{oriented orbit equivalent}. For $i \in \{0,1\}$, we denote by $\lambda_i$ an Anosov Liouville structure on $V$ supported by $\Phi_i$ and defining a bicontact structure $(\xi^i_-, \xi^i_+)$. As a consequence of Theorem~\ref{thmintro:liouv}, we have:

\begin{thmintro} \label{thmintro:anosov}
    The Liouville domains $(V, \lambda_0)$ and $(V, \lambda_1)$ are deformation equivalent, hence their completions are exact symplectomorphic. The symplectomorphism is isotopic to $\mathrm{id} \times h$. 
    
    In particular, $\xi^0_+$ (resp.~$\xi^0_-$) and $\xi^1_+$ (resp.~$\xi^1_-$) are contactomorphic via diffeomorphisms isotopic to $h$ (possibly through two \emph{different} diffeomorphisms).
\end{thmintro}

As noted before, our proof strategy actually shows that $(\xi^0_-, \xi^0_+)$ and $(\xi^1_-, \xi^1_+)$ are deformation equivalent as \emph{positive contact pairs}. 

    \subsubsection{Uniqueness of supporting bicontact structures}

A \textbf{bicontact structure} is a pair of contact structures $(\xi_-, \xi_+)$ with opposite signs which are transverse. In particular, bicontact structures are positive contact pairs (for \emph{any} coorientations).

As a direct consequence of our approximation results as well as the more general version of Vogel's uniqueness theorem, we obtain a variant of Theorem~\ref{thmintro:poscont} for Anosov flows: 

\begin{thmintro} \label{thmintro:anosovbicontact}
    Let $\Phi_0$ and $\Phi_1$ be two oriented Anosov flows on $M$ supported by bicontact structures $(\xi^0_-, \xi^0_+)$ and $(\xi^1_-, \xi^1_+)$, respectively. If $\Phi_0$ and $\Phi_1$ are oriented orbit equivalent, then $(\xi^0_-, \xi^0_+)$ and $(\xi^1_-, \xi^1_+)$ are deformation equivalent through bicontact structures. 
    
    More precisely, if $h : M \rightarrow M$ is an (oriented) orbit equivalence between $\Phi_0$ and $\Phi_1$, then $h$ is isotopic to a smooth diffeomorphism $\widetilde{h} : M \rightarrow M$ such that $\big(\widetilde{h}_*(\xi^0_-), \widetilde{h}_*(\xi^0_+)\big)$ and $(\xi^1_-, \xi^1_+)$ are homotopic through bicontact structures.
\end{thmintro}

By the contact characterization of \emph{projectively Anosov flows}~\cite{Mit95, ET}, we readily obtain:

\begin{corintro} \label{corintro:projanosov}
    If two oriented Anosov flows $\Phi_0$ and $\Phi_1$ are orbit equivalent, then they are deformation equivalent through projectively Anosov flows. More precisely, there exists a diffeomorphism $\widetilde{h} : M \rightarrow M$ topologically isotopic to the orbit equivalence $h$ such that $\widetilde{h}_* \Phi_0$ is homotopic to $\Phi_1$ through projectively Anosov flows.
\end{corintro}

\subsubsection{\texorpdfstring{$\R$}{R}-covered and contact Anosov flows}

Fenley~\cite{F94} and independently Barbot~\cite{B95} discovered a fundamental dichotomy among Anosov flows on $3$-manifolds, between those that are $\mathbb{R}$-covered and those that are not. Here an Anosov flow on $M$ is $\mathbb{R}$-covered if the leaf space of the lift of the weak-(un)stable foliation to the universal cover $\widetilde{M}$ is homeomorphic to $\R$. There is then a rich structure theory for such flows, essentially going back to Fenley's early work. Suspension flows of hyperbolic torus automorphisms are $\R$-covered, and in that case the global picture on $\widetilde{M}$ is that of a product; such flows are called \textbf{product $\R$-covered Anosov flows}.

The other classical example of an Anosov flow is given by the geodesic flow of a negatively curved metric on the unit tangent bundle of a closed surface, which is also $\mathbb{R}$-covered. In this case, however, there is no global product structure for the weak foliations and one obtains a ``skewed strip''; see~\cite{F94}. In particular, one refers to such flows as \textbf{skewed $\R$-covered Anosov}. Furthermore, since the manifold is oriented, one can distinguish between those flows that are positively and negatively skewed. 

In fact, the geodesic flow is in addition the Reeb flow of a suitable contact form for the canonical contact structure on the unit (co)tangent bundle of the surface. Barbot~\cite{B01} showed that any Reeb flow of a positive contact structure, which is in addition Anosov, is automatically positively skewed $\R$-covered. Such flows will be called \textbf{contact Anosov}. Very recently, Marty~\cite{M24} was able to show the converse, giving a complete characterization of $\R$-coveredness in terms of contact geometry. Hence, in what follows, one can use skewed $\R$-covered and contact Anosov interchangeably.

The following proposition is well-known to the experts but we were not able to find a proof in the literature.

\begin{propintro} \label{prop:contactanosov}
    If $\Phi$ is a contact Anosov flow for a positive (resp.~negative) contact structure $\xi$, and if $\Phi$ is tangent to a positive (resp.~negative) contact structure $\xi'$, then $\xi$ and $\xi'$ are contact homotopic.
\end{propintro}

\begin{proof}
We have $\xi = E^{ss} \oplus E^{uu}$, where $E^{ss}$ and $E^{uu}$ denote the strong-stable and unstable bundles of $\Phi$, respectively. In particular, those are $C^1$. By Hozoori~\cite[Theorem 1.8]{Hoz24}, $\xi'$ is homotopic to a contact structure which belongs to a bicontact structure supporting $\Phi$. In particular, it is transverse to $E^s$. We can apply a $C^1$-small perturbation to $\xi$ to make it transverse to $E^{ss}$ as well. We can then flow along $\Phi$ to homotope the contact structures to ones that are $C^0$-close to $E^{wu}$, and the uniqueness of contact structures approximating $\mathcal{F}^{wu}$ finishes the proof.
\end{proof}

Note that here we use Vogel's theorem for a foliation that is not quite $C^2$. However, Vogel's original proof would apply here since the foliations under consideration have no closed leaves and every minimal set has a curve with attracting holonomy; no additional transverse control is required.\footnote{One may believe that there should be a proof of the previous proposition that does not rely on Vogel's uniqueness result. For instance, one might try to find a suitable contact form $\alpha'$ for $\xi'$ whose Reeb vector field is transverse to $\xi$ with the correct orientation; that would ensure that the linear interpolation between $\alpha'$ and $\alpha$ (the contact form whose Reeb vector field generates $\Phi$) is a path of contact forms. Unfortunately, we were unable to make this strategy work.}

As a byproduct of Theorem~\ref{thmintro:anosov} and Marty's result, we obtain:

\begin{thmintro} \label{thm:Rcovered}
    Let $\Phi$ be a positive (resp.~negative) skewed $\R$-covered Anosov flow with supporting bicontact structure $(\xi_-, \xi_+)$. Then $\xi_+$ (resp.~$\xi_-$) admits a contact form whose Reeb vector field is Anosov and isotopically equivalent to $\Phi$.
\end{thmintro}

Here, two flows are isotopically equivalent if they are orbit equivalent via an orbit equivalence isotopic to the identity. This resolves part of a conjecture of Barthelm\'{e}; see~\cite[Conjecture 4.18]{B25}. Combined with Proposition~\ref{prop:contactanosov}, we readily get:

\begin{corintro}
    Let $\xi$ be a positive (resp.~negative) contact structure on $M$. Then $\xi$ admits an Anosov Reeb vector field if and only if there exists a positive (resp.~negative) skewed $\R$-covered Anosov flow tangent to $\xi$.
\end{corintro}

By the work of Barthelm\'{e}--Mann--Bowden~\cite{BM24} combined with Marty's result, we also obtain:

\begin{corintro}
    Let $\Phi_0$, $\Phi_1$ be two positive (resp.~negative) skewed $\R$-covered Anosov flows which are tangent to the same positive (resp.~negative) contact structure $\xi$. Then $\Phi_0$ and $\Phi_1$ are isotopically equivalent.
\end{corintro}

\begin{proof}
    By Theorem~\ref{thm:Rcovered}, $\Phi_0$ and $\Phi_1$ are both isotopically equivalent to Anosov Reeb flows for $\xi$. By Barthelm\'{e}--Mann--Bowden, all the Anosov Reeb flows for a given contact structure are isotopically equivalent, hence $\Phi_0$ and $\Phi_1$ are isotopically equivalent.
\end{proof}

\begin{remintro}
    The distinction between positive and negative contact structures is essential in the previous results. For instance, there exist closed $3$-manifolds admitting two \emph{positive} skewed $\R$-covered Anosov flows which are tangent to the same \emph{negative} contact structure, but are not orbit equivalent. Such examples are constructed by Barbot--Fenley in~\cite{BF24} for suitable covers of unit (co)tangent bundles of closed surfaces. Further candidate examples are constructed in~\cite{BM22}.
\end{remintro}

    \subsection{Open questions}

We conclude the introduction with some further questions.

\medskip

First, we ask if Theorem~\ref{thmintro:liouv} extends to \emph{semi-conjugacies} (continuous surjective maps sending leaves to leaves) between hypertaut foliations:

\begin{questionintro}
    If two hypertaut $C^1$-foliations are semi-conjugate, how are their Liouville thickenings related? What about their approximating positive contact pairs?
\end{questionintro}

Now considering Anosov flows, one can ask if a stronger version of Theorem~\ref{thmintro:anosov} holds:

\begin{questionintro}
If $\Phi_0$ and $\Phi_1$ are Anosov flows on $M$ generated by smooth vector fields $X_0$ and $X_1$, respectively, and if they are orbit equivalent through an orbit equivalence $h$, does there exist a smooth diffeomorphism $\widetilde{h}$ close to $h$ such that $\widetilde{h}_* X_0$ is homotopic to $X_1$ through $C^1$ Anosov vector fields?
\end{questionintro}

We remark that Theorem~\ref{thm:strongbifolapprox} would provide a $C^0$ path of $C^1$ Anosov vector fields, so one would hope to upgrade the regularity of the deformation.

In a different direction, one can ask if a converse to Theorem~\ref{thmintro:anosov} holds:

\begin{questionintro}
If two Anosov Liouville structures are exact symplectomorphic after completion, are their underlying Anosov flows orbit equivalent? 
\end{questionintro}

Recall that two Liouville structures $\lambda_0$ and $\lambda_1$ on a (finite type) Liouville manifold $\widehat{V}$ are exact symplectomorphic if there exists a diffeomorphism $\varphi : \widehat{V} \rightarrow \widehat{V}$ and a smooth map $f : \widehat{V} \rightarrow \R$ which is constant near infinity such that 
$$\lambda_1 = \varphi^* \lambda_0 + df.$$

The answer to the above question is positive for (skewed) $\R$-covered Anosov flows, by the work of Barthelm\'{e}--Mann--Bowden~\cite{BM24} and Marty~\cite{M24}. One possible way to address this question in general would be to consider the \emph{skeleta} of these two Liouville structures, since the Liouville flows restrict to (scalings of) the respective Anosov flows there; see~\cite{Mas24}. However, because of the `$df$' term in the definition of exact symplectomorphism, it is not immediately clear how to relate the two skeleta.

\medskip

One may also ask if a converse to Corollary~\ref{corintro:projanosov} holds; this was already raised by Hozoori:

\begin{questionintro}[Hozoori~\cite{Hoz24}]
    If two (oriented) Anosov flows are homotopic through projectively Anosov flows, are they orbit equivalent?
\end{questionintro}

In~\cite{B25}, Barthelm\'{e} asks: if $\Phi$ is an Anosov flow such that $\xi_+$, the positive contact structure of a supporting bicontact structure, admits an Anosov Reeb vector field, is $\Phi$ $\R$-covered? This is equivalent to:

\begin{questionintro}
    If $\xi_+$ is a positive contact structure on $M$ which supports a positive skewed $\R$-covered Anosov flow, is every Anosov flow supported by $\xi_+$ positive skewed $\R$-covered as well? If so, then all these flows are orbit equivalent.
\end{questionintro}

There is a similar statement for negative contact structures. An affirmative answer to this question would mean that the ``$\R$-coveredness nature'' of an Anosov flow can be determined from a bicontact structure supporting it.

\medskip

Finally, it would be interesting to strengthen Theorem~\ref{thmintrobeta:uniq}, in order to obtain a better control on the homotopies between approximating contact structures to a hypertaut foliation:

\begin{questionintro}
    Let $\mathcal{F}$ be a hypertaut $C^1$-foliation on $M$. Is the following statement true: for every neighborhood $\mathcal{U}$ of $T \mathcal{F}$, there exists a smaller neighborhood $\mathcal{V} \subset \mathcal{U}$ such that any two positive (resp.~negative) contact structures in $\mathcal{V}$ are homotopic through contact structures \underline{within $\mathcal{U}$}?
\end{questionintro}

We only prove this result for a neighborhood $\mathcal{U}$ corresponding to the set of plane fields transverse to a given smooth line field, which is sufficient for our purpose. To prove such a statement for weak foliations of Anosov flows, it would suffice to generalize our strategy to less regular transverse line fields, and apply the result to the strong line fields of the flow, which are continuous but not necessarily $C^1$.

        \subsection*{Acknowledgments}

J.~Bowden is supported by the Heisenberg-Program (BO 4423/4-1) of the German Science Foundation. T.~Massoni is supported by a Stanford Science Fellowship. 

We are particularly grateful to CIRM and the organizers of the conference \emph{Foliations and Diffeomorphism Groups}, as well as the hospitality of International Wissenschaftsforum in Heidelberg for hosting the Workshop \emph{Symplectic geometry and Anosov flows}, where parts of this project were born. We would also like to thank Thomas Barthelm\'{e}, Fabio Gironella, Rafael Potrie, and Jonathan Zung for their interest and valuable input.

    \section{Smoothing (bi)foliated homeomorphisms} \label{sec:smooth}

Recall that we are assuming that all the plane fields and line fields under consideration are orientable, and all the foliations are of class at least $C^1$.

In this section, we prove Theorem~\ref{thmintrobeta:approx} and Theorem~\ref{thmintrobeta:bifolapprox} from the Introduction. We first describe the strategy one dimension lower for the sake of clarity. We then move to the general case of $2$-dimensional foliations on $3$-manifolds, before explaining how to extend the strategy to bifoliations.

        \subsection{Preamble: the \texorpdfstring{$2$}{2}-dimensional case}

In order to convey the key ideas of our proof, we first describe it in the $2$-dimensional case, namely, for $1$-dimensional foliations on surfaces. We will put an emphasis on the main ideas at the expense of rigor. 

We consider the following setup. Let $\Sigma$ be a smooth, connected, oriented surface---not necessarily compact---together with two $C^1$-foliations $\mathcal{F}_0$ and $\mathcal{F}_1$ which are both cooriented. We then consider a homeomorphism $h : \Sigma \rightarrow \Sigma$ which sends the leaves of $\mathcal{F}_0$ to leaves of $\mathcal{F}_1$. For simplicity, we further assume that $h$ preserves the orientation of the surface as well as the coorientations of the foliations.

We want to approximate $h$ by a smooth diffeomorphism $\widetilde{h}$ such that the tangent line field of the foliation $\widetilde{\mathcal{F}}_1 \coloneqq \widetilde{h}_*(\mathcal{F}_0)$ is very close to that of $\mathcal{F}_1$. This problem is easy to solve \emph{locally}: near $p \in \Sigma$ and $ h(p)$, we can find $C^1$ coordinates $(x,y)$ in which $\mathcal{F}_0$ and $\mathcal{F}_1$ are horizontal (tangent to $\partial_x$), and $h$ is of the form
$$h(x,y) = \big(u(x,y), v(y)\big),$$
where $u(\, \cdot \,, y)$ and $v$ are \emph{strictly increasing}. See Figure~\ref{fig:h}.

\begin{figure}[!ht]
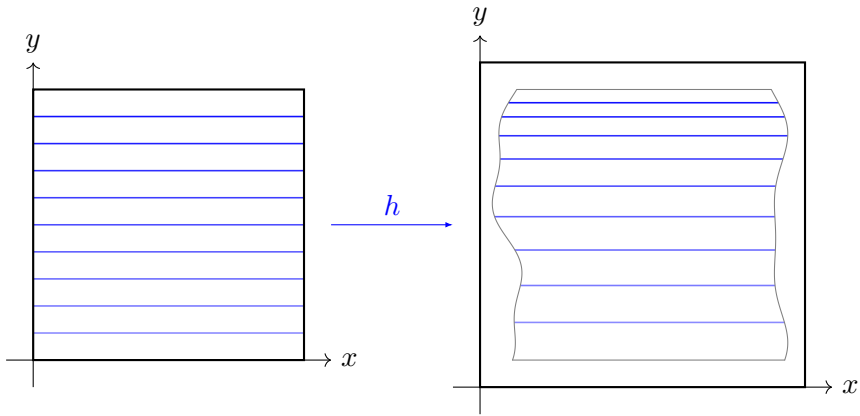

    \centering
    \includestandalone{tikz/h}
    \caption{Local depiction of $h$.}
    \label{fig:h}
\end{figure}

It is easy to approximate $u$ by a smooth map $\widetilde{u} = \widetilde{u}(x,y)$ satisfying $\partial_x \widetilde{u} > 0$, and $v$ by a smooth map $\widetilde{v} = \widetilde{v}(y)$ satisfying $\partial_y \widetilde{v} > 0$. Then, the map
$$\widetilde{h}(x,y) \coloneqq \big(\widetilde{u}(x,y), \widetilde{v}(y) \big)$$
is a smooth diffeomorphism onto its image and sends $\mathcal{F}_0$ to $\mathcal{F}_1$.

The difficulty of the proof is to carefully patch these local smoothings together in order to obtain a global \emph{diffeomorphism} (not merely a smooth map!) while keeping control on the tangent line field of $\widetilde{\mathcal{F}}_1$.

To achieve this, we consider a sufficiently fine triangulation $\mathcal{T}$ of $\Sigma$ in general position with $\mathcal{F}_0$, and so that each simplex is contained in a neighborhood in which $\mathcal{F}_0$ is standard (and so is $\mathcal{F}_1$ on the image of these neighborhoods under $h$). We will further assume that these neighborhoods come with good coordinate systems and overlap in a controlled way. More precisely, we consider for each simplex $\mathfrak{t} \in \mathcal{T}$ the following data:
\begin{itemize}
    \item A neighborhood $U_\mathfrak{t}$ containing $\mathfrak{t}$, 
    \item $C^1$ coordinates $\varphi_\mathfrak{t} : U_\mathfrak{t} \rightarrow (0,1)_{x,y}^2$ in which $\mathcal{F}_0$ becomes horizontal, i.e., spanned by $\partial_x$,
    \item An open neighborhood $\overline{h(U_\mathfrak{t})} \subset V_\mathfrak{t}$ with coordinates $\psi_\mathfrak{t} : V_\mathfrak{t} \rightarrow \R_{x,y}^2$ in which $\mathcal{F}_1$ becomes horizontal.
\end{itemize}
We further assume that the following conditions are satisfied:
\begin{itemize}
    \item For all $\mathfrak{t}, \mathfrak{t}' \in \mathcal{T}$, $U_\mathfrak{t} \cap U_{\mathfrak{t}'} \subset U_{\mathfrak{t} \cap \mathfrak{t}'}$, with the convention $U_\varnothing \coloneqq \varnothing$.
    \item If $\mathfrak{t} \in  \mathcal{T}$ is an edge and $\mathfrak{t}_0 \in \partial \mathfrak{t}$ is a vertex, we require that there is a leaf of $\mathcal{F}_0$ separating $U_\mathfrak{t} \cap \bigcup_{\mathfrak{t}_0 \in \partial \mathfrak{t}'} U_{\mathfrak{t}'}$ and $U_\mathfrak{t} \setminus U_{\mathfrak{t}_0}$. Here, the union runs over all the edges in $\mathcal{T}$ containing $\mathfrak{t}_0$.
    \item If $\mathfrak{t} \in \mathcal{T}$ is a $2$-simplex, then
    \begin{itemize}
        \item The set 
        $$\varphi_\mathfrak{t}\left(U_\mathfrak{t} \cap \bigcup_{\mathfrak{t}' \in \partial\mathfrak{t}} U_{\mathfrak{t}'}\right) \subset (0,1)^2$$
        contains some $\ell^\infty$-neighborhood $N^1_\mathfrak{t}$ of $\partial [0,1]^2$. Here, the union runs over all the edges in $\partial \mathfrak{t}$.
        \item The set 
        $$\varphi_\mathfrak{t}\left(U_\mathfrak{t} \cap \bigcup_{\mathfrak{t}' \cap \mathfrak{t} \neq \varnothing,\\ \mathfrak{t}' \neq \mathfrak{t}} U_{\mathfrak{t}'}\right) \subset (0,1)^2$$
        is contained in a $\ell^\infty$-neighborhood $N^2_\mathfrak{t} \Subset N^1_\mathfrak{t}$ of $\partial [0,1]^2$. Here, the union runs over all the $2$-simplices intersecting $\mathfrak{t}$.
    \end{itemize}
\end{itemize}
See Figure~\ref{fig:cover}. The technical conditions on the overlap of the neighborhoods of the simplices in $\mathcal{T}$ will ensure that the various smoothings of $h$ on those can be easily patched together into a global map.

\begin{figure}[!ht]
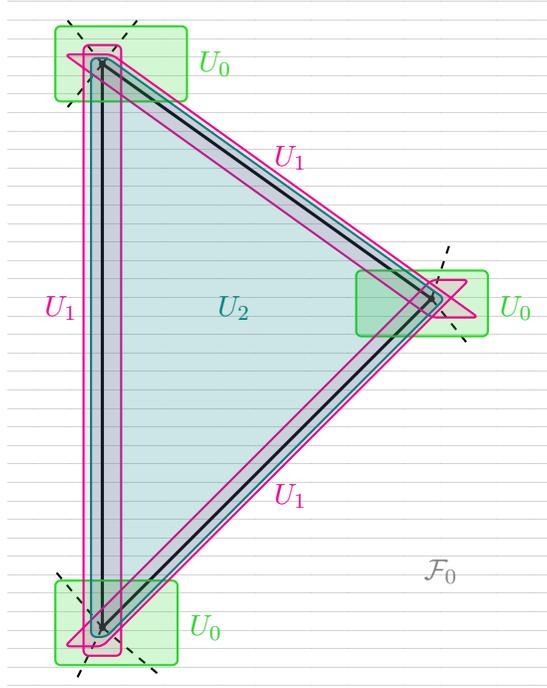

    \centering
    \includestandalone{tikz/cover}
    \caption{Neighborhoods of simplices.}
    \label{fig:cover}
\end{figure}

We will now proceed by induction on the dimension of the simplices to construct the desired smoothing of $h$. For $i \in \{0,1,2\}$, we denote by $\mathcal{T}_i$ the set of $i$-dimensional simplices in $\mathcal{T}$.

It is easy to find a smoothing $\widetilde{h}_0$ of $h$ on the union $U_0$ of the $U_\mathfrak{t}$'s over the vertices $\mathfrak{t} \in \mathcal{T}_0$. Moreover, this smoothing sends $\mathcal{F}_0$ to $\mathcal{F}_1$. The next step is to obtain a smoothing $\widetilde{h}_1$ on the union $U_1$ of the $U_\mathfrak{t}$'s over the edges $\mathfrak{t} \in \mathcal{T}_1$. For such an edge $\mathfrak{t} \in \mathcal{T}_1$, we consider the maps
\begin{align*}
    h_\mathfrak{t} &\coloneqq \psi_\mathfrak{t} \circ h \circ \varphi^{-1}_\mathfrak{t} : (0,1)^2 \rightarrow \R^2,\\
    \widetilde{h}_{\mathfrak{t}, 0} &\coloneqq \psi_\mathfrak{t} \circ \widetilde{h}_0 \circ \varphi^{-1}_\mathfrak{t} : \varphi(U_0 \cap U_\mathfrak{t}) \rightarrow \R^2,
\end{align*}
which are of the form
\begin{align*}
    h_\mathfrak{t}(x,y) &= \big(u_\mathfrak{t}(x,y), v_\mathfrak{t}(y) \big),\\
    \widetilde{h}_{\mathfrak{t},0}(x,y) &= \big(\widetilde{u}_{\mathfrak{t},0}(x,y), \widetilde{v}_{\mathfrak{t},0}(y) \big),
\end{align*}
where $u_\mathfrak{t}$ and $\widetilde{u}_{\mathfrak{t},0}$ (resp.~$v_\mathfrak{t}$ and $\widetilde{v}_{\mathfrak{t},0}$) are $C^0$-close on the set where they are both defined. See Figure~\ref{fig:h0}.

One can find smoothings $\widetilde{u}_\mathfrak{t} = \widetilde{u}_\mathfrak{t}(x,y)$ and $\widetilde{v}_\mathfrak{t} = \widetilde{v}_\mathfrak{t}(y)$ of $u_\mathfrak{t}$ and $v_\mathfrak{t}$, respectively, which satisfy:
\begin{itemize}
    \item $\partial_x \widetilde{u}_\mathfrak{t} > 0$,
    \item $\partial_y \widetilde{v}_\mathfrak{t} > 0$,
    \item $\widetilde{u}_\mathfrak{t}$ and $\widetilde{u}_{\mathfrak{t},0}$ coincide near $(0,1) \times \partial(0,1)$,
    \item $\widetilde{v}_\mathfrak{t}$ and $\widetilde{v}_{\mathfrak{t},0}$ coincide near $\partial (0,1)$.\footnote{One should be more precise about the exact neighborhoods where these maps coincide, but we remain informal for now.}
\end{itemize}

We can then patch together the maps $\psi^{-1}_\mathfrak{t} \circ \widetilde{h}_\mathfrak{t} \circ \varphi_\mathfrak{t}$, for the edges $\mathfrak{t} \in \mathcal{T}_1$, into a smooth embedding $\widetilde{h}_1 : U_1 \rightarrow \Sigma$ which is $C^0$-close to $h$, and which sends $\mathcal{F}_0$ to $\mathcal{F}_1$.

Finally, we want to find an appropriate smoothing of $h$ on the whole of $\Sigma$, using the previously constructed smoothing $\widetilde{h}_1$. At this point, we will also have to modify the target foliation.

As before, we consider for each $2$-simplex $\mathfrak{t} \in \mathcal{T}_2$
\begin{align*}
    h_\mathfrak{t} &\coloneqq \psi_\mathfrak{t} \circ h \circ \varphi^{-1}_\mathfrak{t} : (0,1)^2 \rightarrow \R^2,\\
    \widetilde{h}_{\mathfrak{t}, 1} &\coloneqq \psi_\mathfrak{t} \circ \widetilde{h}_1 \circ \varphi^{-1}_\mathfrak{t} : \varphi(U_1 \cap U_\mathfrak{t}) \rightarrow \R^2,
\end{align*}
which are of the form 
\begin{align*}
    h_\mathfrak{t}(x,y) &= \big(u_\mathfrak{t}(x,y), v_\mathfrak{t}(y) \big),\\
    \widetilde{h}_{\mathfrak{t},1}(x,y) &= \big(\widetilde{u}_{\mathfrak{t},1}(x,y), \widetilde{v}_{\mathfrak{t},1}(x,y) \big),
\end{align*}
where $u_\mathfrak{t}$ and $\widetilde{u}_{\mathfrak{t},1}$ (resp.~$v_\mathfrak{t}$ and $\widetilde{v}_{\mathfrak{t},1}$) are $C^0$-close on the set where they are both defined. See Figure~\ref{fig:h1}.

Note that $\widetilde{v}_{\mathfrak{t},1}(x,y)$ is locally constant in $x$; setting  
\begin{align*}
    \widetilde{v}^i_{\mathfrak{t},1}(y) \coloneqq \widetilde{v}_{\mathfrak{t},1}(x,y)
\end{align*}
for $x$ close to $i \in \{0,1\}$, we obtain two smooth approximations of $v_\mathfrak{t}$ satisfying $\partial_y \widetilde{v}^i_{\mathfrak{t},1} > 0$ and which coincide near $\partial (0,1)$. However, they might be \emph{different} since they come from (transversal) smoothings of $h$ near different edges in the boundary of $\mathfrak{t}$. We will have to interpolate between them in a graphical way, which will modify the image of $\mathcal{F}_0$. The key observation is that the modified line field will differ from $\mathcal{F}_1$ by a quantity that depends only on the geometry of the coverings and choices of coordinates, which are \emph{fixed}, and the quantity $\big\vert \widetilde{v}^1_{\mathfrak{t}, 1} - \widetilde{v}^0_{\mathfrak{t}, 1} \big\vert$, which can be made arbitrarily small at the previous step.

There is an extra difficulty due to the fact that the image of $h_\mathfrak{t}$, i.e., the set $\psi_\mathfrak{t}(h(U_\mathfrak{t})) \subset \R^2$, might have very ``wiggly sides'', making this graphical interpolation complicated. For simplicity, we will assume that $h_\mathfrak{t}$ is very close to the identity in the $C^0$ topology. This can be achieved by composing $h_\mathfrak{t}$ with the inverse of a smoothing of (a slight extension of) itself. Further details will be given below when we treat the $3$-dimensional case.

Then, we consider a cutoff function $\tau : [0,1] \rightarrow [0,1]$ which is nonincreasing and supported on a sufficiently large neighborhood of $0$, and we set
$$V_\mathfrak{t}(x,y) \coloneqq \tau(x) \widetilde{v}^0_{\mathfrak{t},1}(y) + (1-\tau(x))\widetilde{v}^1_{\mathfrak{t},1}(y).$$
We can also find a smoothing $\widetilde{u}_\mathfrak{t}$ of $u_\mathfrak{t}$ which satisfies $\partial_x \widetilde{u}_\mathfrak{t} > 0$ and coincides with $\widetilde{u}_{\mathfrak{t}, 1}$ near $\partial (0,1)^2$, and we define:
$$\widetilde{h}_\mathfrak{t} \coloneqq \big(\widetilde{u}_\mathfrak{t}(x,y), V_\mathfrak{t}(\widetilde{u}_\mathfrak{t}(x,y), y) \big) = \big(\widetilde{u}_\mathfrak{t}(x,y), \widetilde{v}_\mathfrak{t}(x, y) \big),$$
which coincides with $\widetilde{h}_{\mathfrak{t},1}$ near $\partial (0,1)^2$. See Figure~\ref{fig:h2}. By definition, this map is a $C^1$ embedding which is $C^0$-close to $h_\mathfrak{t}$. Moreover, we have:
$$\big\vert \big(\widetilde{h}_\mathfrak{t}\big)_*\partial_x - \partial_x \vert \leq \vert \tau'\vert \big\vert \widetilde{v}^1_{\mathfrak{t}, 1} - \widetilde{v}^0_{\mathfrak{t}, 1} \big\vert.$$

\begin{figure}[!ht]
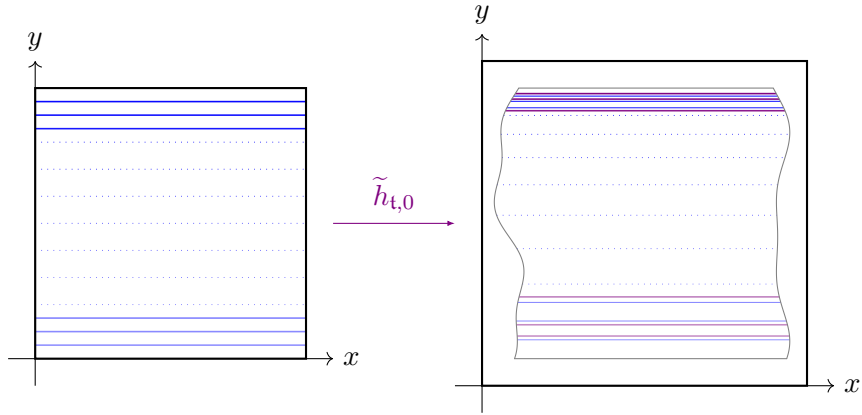
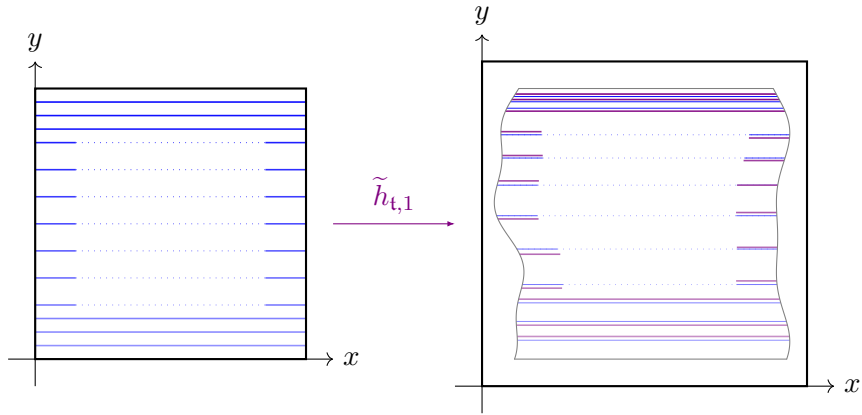
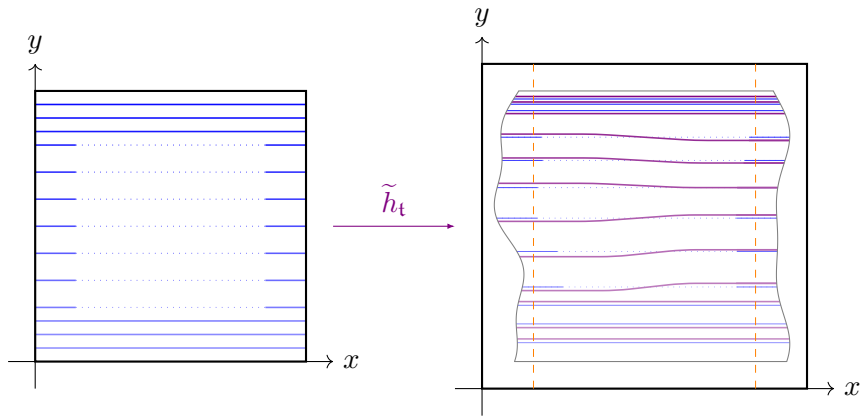

\centering
    \begin{subfigure}[b]{0.8\textwidth}
        \centering
        \includestandalone{tikz/h0}
        \caption{The smoothing $\widetilde{h}_{\mathfrak{t},0}$ for $\mathfrak{t} \in \mathcal{T}_1$.}
        \label{fig:h0}
    \end{subfigure}
    
    \vspace{1cm}

    \begin{subfigure}{0.8\textwidth}
        \centering
        \includestandalone{tikz/h1}
        \caption{The smoothing $\widetilde{h}_{\mathfrak{t},1}$ for $\mathfrak{t} \in \mathcal{T}_2$.}
        \label{fig:h1}
    \end{subfigure}

    \vspace{1cm}

    \begin{subfigure}{0.8\textwidth}
        \centering
        \includestandalone{tikz/h2}
        \caption{The smoothing $\widetilde{h}_{\mathfrak{t}}$ for $\mathfrak{t} \in \mathcal{T}_2$.}
        \label{fig:h2}
    \end{subfigure}
    \caption{Steps of the smoothing procedure.}
\end{figure}

Here, the size of $\tau'$ is essentially fixed by the setup, while the difference $\big\vert \widetilde{v}^1_{\mathfrak{t}, 1} - \widetilde{v}^0_{\mathfrak{t}, 1} \big\vert$ depends on the choice of smoothing $\widetilde{h}_1$ on $U_1$, and can be made arbitrarily small. Therefore, we can guarantee that the smooth map $\widetilde{h} = \widetilde{h}_2$ obtained by patching together the maps $\psi^{-1}_\mathfrak{t} \circ \widetilde{h}_\mathfrak{t} \circ \varphi_\mathfrak{t}$, for the $2$-simplices $\mathfrak{t} \in \mathcal{T}_2$, sends $\mathcal{F}_0$ to a foliation whose line field is arbitrarily $C^0$-close to that of $\mathcal{F}_1$.

\bigskip

Let us briefly explain how to adapt the strategy to the bifoliated case. We consider two pairs of transverse cooriented $C^1$-foliations $(\mathcal{F}_0, \mathcal{G}_0)$ and $(\mathcal{F}_1, \mathcal{G}_1)$ on $\Sigma$, as well as a homeomorphism $h : \Sigma \rightarrow \Sigma$ sending the leaves of $\mathcal{F}_0$ (resp.~$\mathcal{G}_0$) to leaves of $\mathcal{F}_1$ (resp.~$\mathcal{G}_1$). For every $p \in \Sigma$, there exist $C^1$ coordinates near $p \in \Sigma$ and near $h(p)$ in which $h$ is of the form
$$h(x,y) = (u(x), v(y)),$$
where $u$ and $v$ are both continuous and strictly increasing functions. Using the previous strategy, it is easy to produce a smoothing $\widetilde{h}_1$ of $h$ in a neighborhood of the $1$-skeleton of a sufficiently fine and generic triangulation $\mathcal{T}$ of $\Sigma$, such that $\widetilde{h}_1$ still sends $(\mathcal{F}_0, \mathcal{G}_0)$ to $(\mathcal{F}_1, \mathcal{G}_1)$. For the extension over the $2$-simplices, we can proceed similarly by extending $\widetilde{h}_1$ by graphical interpolations in both the vertical and horizontal directions. Concretely, we first define a new bifoliation $\big(\widetilde{\mathcal{F}}_1, \widetilde{\mathcal{G}}_1\big)$ which coincides with $(\mathcal{F}_1, \mathcal{G}_1)$ near $h(\mathcal{T}_1)$ by a suitable interpolation, and extend $\widetilde{h}_1$ so that it maps $(\mathcal{F}_0, \mathcal{G}_0)$ to $\big(\widetilde{\mathcal{F}}_1, \widetilde{\mathcal{G}}_1\big)$. As before, we will be able to ensure that the line field of $\widetilde{\mathcal{F}}_1$ (resp.~$\widetilde{\mathcal{G}}_1$) is very close to that of $\mathcal{F}_1$ (resp.~$\mathcal{G}_1$).

\bigskip

We will now consider the $3$-dimensional case and make some of the previous definitions and technical steps more precise.

        \subsection{Adapted coordinates and clean covers}

Let $\mathcal{F}$ be a cooriented $C^1$-foliation on $M$.

\begin{defn} \label{def:adapted}
A $C^1$ coordinate system $(x,y,z)$ near $p \in M$ is \textbf{adapted to $\mathcal{F}$} if in these coordinates, 
\begin{align} \label{eq:adapted}
T \mathcal{F} = \mathrm{span}\big\{ \partial_x, \partial_y\},
\end{align}
and $\partial_z$ is positively transverse to $\mathcal{F}$.
\end{defn}

There exists a uniform constant $\delta_0 > 0$ such that every open ball of radius less than $\delta_0$ in $M$ admits coordinates adapted to $\mathcal{F}$.

\medskip

Let $0 < \delta < \delta_0$ and $\mathcal{T}$ be a triangulation of $M$. We say that $\mathcal{T}$ is \textbf{$\delta$-fine} if each of its simplices is included in a ball of radius $\delta/2$. We now assume that $\mathcal{T}$ is $\delta$-fine and in general position with $\mathcal{F}$, which can always be achieved by considering a sufficiently fine and suitable subdivision of $\mathcal{T}$ and applying Thurston's Jiggling Lemma~\cite{T74} (see also~\cite[Section 4A2]{V16}).

For $0 \leq i \leq 3$, we write
$$\mathcal{T}_i \coloneqq \big\{ \mathfrak{t} \in \mathcal{T} \ \big\vert \ \dim(\mathfrak{t}) = i \big\}.$$

\begin{defn} \label{def:cover}
A \textbf{$\delta$-clean cover} $\big(\mathcal{U}, \boldsymbol{\varphi}\big)$ of $M$ adapted to $\mathcal{F}$ and modeled on $\mathcal{T}$ is a collection $\mathcal{U} = (U_\mathfrak{t})_{\mathfrak{t} \in \mathcal{T}}$ of open subsets of $M$ indexed by the simplices of $\mathcal{T}$, together with a collection of $C^1$ diffeomorphisms $\boldsymbol{\varphi} = (\varphi_\mathfrak{t})_{\mathfrak{t} \in \mathcal{T}}$, $\varphi_\mathfrak{t} : \overline{U}_\mathfrak{t} \rightarrow [0,1]^3$, such that for every $\mathfrak{t} \in \mathcal{T}$, the following properties hold.
\begin{enumerate}
    \item $\mathfrak{t} \subset U_\mathfrak{t}$, $\mathrm{diam}(U_\mathfrak{t}) < \delta$, and $\varphi_\mathfrak{t}$ defines coordinates on $\overline{U}_\mathfrak{t}$ adapted to $\mathcal{F}$,
    \item If $\mathfrak{t}' \in \mathcal{T}$, then $U_\mathfrak{t} \cap U_{\mathfrak{t}'} \subset U_{\mathfrak{t}\cap \mathfrak{t}'}$, where $U_\varnothing = \varnothing$ by convention,
    \item If $\dim(\mathfrak{t}) \geq 1$, there exists a subset $B_\mathfrak{t} \subset \partial[0,1]^3$ made of the union of $\dim(\mathfrak{t})$ pairs of opposite faces of $[0,1]^3$, and a width $0 < w_\mathfrak{t} < 0.1$, such that
    \begin{enumerate}
        \item The set 
        $$N_\mathfrak{t} \coloneqq \varphi_\mathfrak{t} \left( \overline{U}_\mathfrak{t} \cap \bigcup_{\mathfrak{t}' \in \partial \mathfrak{t}} U_{\mathfrak{t}'} \right) \subset [0,1]^3$$
        contains the $\ell^\infty$-neighborhood of radius $2w_\mathfrak{t}$ of $B_\mathfrak{t}$,
        \item For every $\mathfrak{t}' \in \mathcal{T}$ with $\dim(\mathfrak{t}') = \dim(\mathfrak{t})$ and $\mathfrak{t}' \neq \mathfrak{t}$, the set
        $$\varphi_\mathfrak{t} \left( \overline{U}_\mathfrak{t} \cap \overline{U}_{\mathfrak{t}'}\right) \subset N_\mathfrak{t}$$
        is contained in the $\ell^\infty$-neighborhood of radius $w_\mathfrak{t}$ of $B_\mathfrak{t}$.
        \end{enumerate}
\end{enumerate}
\end{defn}

For $0 \leq i \leq 3$, we write
    $$U_i \coloneqq \bigcup_{\mathfrak{t} \in \mathcal{T}_i} U_\mathfrak{t},$$
so that $U_i$ is a neighborhood of the $i$-skeleton of $\mathcal{T}$.

Clean covers of $M$ can easily be constructed by first considering a sufficiently fine and generic triangulation $\mathcal{T}$ of $M$, and then proceeding by induction on the skeleton of $\mathcal{T}$:

\begin{lem} \label{lem:cleancov}
For every $0 < \delta < \delta_0$, there exists a $\delta$-clean cover of $M$ adapted to $\mathcal{F}$ and modeled on some sufficiently fine triangulation of $M$.
\end{lem}

        \subsection{Foliated homeomorphisms}

Let $h : (M, \mathcal{F}_0) \rightarrow (M, \mathcal{F}_1)$ be a foliated homeomorphism, where $\mathcal{F}_0$ and $\mathcal{F}_1$ are cooriented $C^1$-foliations, and $h$ preserves the coorientations.

In coordinates adapted to $\mathcal{F}_0$ and $\mathcal{F}_1$, $h$ is locally of the form
\begin{align} \label{eq:normalformh}
    h(x,y,z) = \big(u_1(x,y,z), u_2(x,y,z), v(z)\big),
\end{align}
where $u_1$, $u_2$, and $v$ are continuous functions, and $v$ is \emph{strictly increasing}. Moreover, $$u_z : (x,y) \mapsto (u_1(x,y,z), u_2(x,y,z)) \in \R^2$$ defines a $1$-parameter family of $C^0$ embeddings.\footnote{Here, $v$ stands for `vertical' and $u$ stands for `urizontal'.}

We choose $\delta > 0$ small enough so that every open ball of radius $\delta$ admits coordinates adapted to $\mathcal{F}_0$, and the image of such a ball under $h$ is included in a ball which admits coordinates adapted to $\mathcal{F}_1$. Then, we consider a $\delta$-clean cover $(\mathcal{U}, \boldsymbol{\varphi})$ adapted to $\mathcal{F}_0$ and modeled on some triangulation $\mathcal{T}$ of $M$ which is $\delta$-fine and in general position with $\mathcal{F}_0$. For each $\mathfrak{t} \in \mathcal{T}$, we choose an open set $V_\mathfrak{t} \subset M$ containing $\overline{h(U_\mathfrak{t})}$ together with coordinates $\psi_\mathfrak{t} : V_\mathfrak{t} \hookrightarrow \R^3$ adapted to $\mathcal{F}_1$.

For $\epsilon > 0$, we will write $[\textrm{quantity}]\lesssim \epsilon$ to mean $[\textrm{quantity}] \leq C \epsilon$ for some unspecified constant $C > 0$ which does not depend on $\epsilon$.

For each $\mathfrak{t} \in \mathcal{T}$, we define
    $$h_\mathfrak{t} \coloneqq \psi_\mathfrak{t} \circ h \circ \varphi^{-1}_\mathfrak{t} : [0,1]^3 \rightarrow \R^3$$
    which is of the form
    $$h_\mathfrak{t}(x,y,z) = \big(u_\mathfrak{t}(x,y,z), v_\mathfrak{t}(z)\big) \in \R^2 \times \R$$
    for some family of $C^0$ embeddings  $u_\mathfrak{t}( \, \cdot\, , \, \cdot \, , z) : [0,1]^2 \hookrightarrow \R^2$ and a continuous, strictly increasing function $v_\mathfrak{t} : [0,1] \rightarrow \R$.

We decompose the proof of Theorem~\ref{thmintrobeta:approx} into several steps. The first two steps---smoothing $h$ near the $0$- then $1$-skeleton of $\mathcal{T}$---are relatively straightforward. The last two steps consist of smoothing near the $2$-skeleton and then extending over the $3$-cells, and will be more technical. Indeed, some care will be required in order to control the derivatives of the smoothing along the leaves of $\mathcal{F}_0$.

    \subsubsection{Smoothing near the \texorpdfstring{$0$}{0}-skeleton}

For the first step, one can simply consider a $C^1$ diffeomorphism that preserves leaves, by smoothing $h$ leafwise, and independently in the transverse direction. In order to set up notation for later steps we make this more precise.

Let $\epsilon > 0$. For $\mathfrak{t} \in \mathcal{T}_0$, we consider
\begin{itemize}
    \item A $C^1$ function $\widetilde{v}_\mathfrak{t} : [0,1] \rightarrow \R$ satisfying
    \begin{align*}
        \partial_z \widetilde{v}_\mathfrak{t} > 0, \qquad \vert \widetilde{v}_\mathfrak{t} - v_\mathfrak{t} \vert_{C^0} < \epsilon,
    \end{align*}
    as provided by the first item in Lemma~\ref{lem:smoothincrease1},
    \item A $C^1$ map $\widetilde{u}_\mathfrak{t} :  [0,1]^3 \rightarrow \R^2$, such that
    \begin{align*}
        \vert \widetilde{u}_\mathfrak{t} - u_\mathfrak{t} \vert_{C^0} < \epsilon,
    \end{align*}
    and for every $z \in [0,1]$, $\widetilde{u}_\mathfrak{t}( \, \cdot \, , z) : [0,1]^2 \rightarrow \R^2$ is a $C^1$ embedding, as provided by the first item in Lemma~\ref{lem:smoothemb1}.
\end{itemize}

We can then define $\widetilde{h}_\mathfrak{t}$ as 
$$\widetilde{h}_\mathfrak{t} \coloneqq \big(\widetilde{u}_\mathfrak{t}(x,y,z), \widetilde{v}_\mathfrak{t}(z)\big),$$
so that $\widetilde{h}_\mathfrak{t} : [0,1]^3 \rightarrow \R^3$ is a $C^1$ embedding (as a proper injective immersion) and 
$$\vert \widetilde{h}_\mathfrak{t} - h_\mathfrak{t}\vert_{C^0} \lesssim \epsilon.$$
We combine the $\widetilde{h}_\mathfrak{t}$'s, $\mathfrak{t} \in \mathcal{T}_0$, together into a map $\widetilde{h}_0 : U_0 \rightarrow M$ defined as 
\begin{align*}
    \widetilde{h}_0(p) \coloneqq \psi_\mathfrak{t}^{-1} \circ \widetilde{h}_\mathfrak{t} \circ \varphi_\mathfrak{t} (p)
\end{align*}
for $\mathfrak{t} \in \mathcal{T}_0$ and $p \in U_\mathfrak{t}$. This expression makes sense for $\epsilon > 0$ small enough, so that the image of $\widetilde{h}_\mathfrak{t}$ is contained in $\psi_\mathfrak{t}(V_\mathfrak{t})$. Also, the $U_\mathfrak{t}$'s, $\mathfrak{t} \in \mathcal{T}_0$, are pairwise disjoint by definition. By construction, $\widetilde{h}_0$ is a $C^1$ embedding sending $\mathcal{F}_0$ (restricted to $U_0$) to $\mathcal{F}_1$ and satisfies 
\begin{align*}
    d_{C^0}\big(h_{\vert U_0}, \widetilde{h}_0\big) \lesssim \epsilon.
\end{align*}
In summary, we have proved:

\begin{lem} \label{lem:skel0}
For every $\epsilon_0 > 0$, there exists a $C^1$ embedding $\widetilde{h}_0 : U_0 \rightarrow M$ satisfying
\begin{align*}
    d_{C^0}\big(h_{\vert U_0}, \widetilde{h}_0\big) < \epsilon_0, \qquad
    (\widetilde{h}_0)_* \mathcal{F}_0 = \mathcal{F}_1.
\end{align*}
\end{lem}

    \subsubsection{Smoothing near the \texorpdfstring{$1$}{1}-skeleton}

The second step is essentially the same as the first one, but relative to the boundary of the edges of the triangulation.

Let $\epsilon > 0$ and $0 < \epsilon_0 \ll \epsilon$, to be chosen sufficiently small below. We choose an embedding $\widetilde{h}_0 : U_0 \rightarrow M$ as in Lemma~\ref{lem:skel0} for $\epsilon_0$.

Let $\mathfrak{t} \in \mathcal{T}_1$. We consider the map
$$\widetilde{h}_{\mathfrak{t}, 0} : \varphi_\mathfrak{t}(\overline{U}_\mathfrak{t} \cap U_0) \subset N_\mathfrak{t} \rightarrow \R^3$$
defined by
$$\widetilde{h}_{\mathfrak{t}, 0} \coloneqq \psi_\mathfrak{t} \circ \widetilde{h}_0 \circ \varphi^{-1}_\mathfrak{t}.$$
Its restriction to $[0,1]^2 \times \big([0, 2w_\mathfrak{t}] \sqcup [1-2w_\mathfrak{t}, 1] \big)$
is of the form
$$\widetilde{h}_{\mathfrak{t}, 0}(x,y,z) = \big(\widetilde{u}_{\mathfrak{t}, 0}(x,y,z), \widetilde{v}_{\mathfrak{t}, 0}(z)\big),$$
where $\partial_z \widetilde{v}_{\mathfrak{t}, 0} > 0$, and each $\widetilde{u}_{\mathfrak{t}, 0}( \, \cdot \, ,z) : [0,1]^2\rightarrow \R^2$ is a $C^1$ embedding. Moreover, after shrinking $\epsilon_0$, we may assume that for every $z \in [0, 2w_\mathfrak{t}] \sqcup [1-2w_\mathfrak{t}, 1]$,
\begin{align*}
    \vert \widetilde{u}_{\mathfrak{t}, 0}( \, \cdot \,,z)  - u_{\mathfrak{t}}( \, \cdot \,,z)\vert_{C^0} < \epsilon, \qquad \vert \widetilde{v}_{\mathfrak{t}, 0}(z)  - v_{\mathfrak{t}}(z)\vert< \epsilon, \qquad \widetilde{v}_{\mathfrak{t}, 0}(w_\mathfrak{t}) < \widetilde{v}_{\mathfrak{t}, 0}(1-w_\mathfrak{t}).
\end{align*}

We now consider
\begin{itemize}
    \item A $C^1$ function $\widetilde{v}_\mathfrak{t} : [0,1] \rightarrow \R$ satisfying
    \begin{align*}
        \partial_z \widetilde{v}_\mathfrak{t} > 0, \qquad \vert \widetilde{v}_\mathfrak{t} - v_\mathfrak{t} \vert_{C^0} < 2\epsilon,
    \end{align*}
    and for every $z \in [0, w_\mathfrak{t} ] \sqcup [1-w_\mathfrak{t}, 1]$,
    \begin{align*}
        \widetilde{v}_\mathfrak{t}(z) = \widetilde{v}_{\mathfrak{t}, 0}(z),
    \end{align*}
    as provided by the second item in Lemma~\ref{lem:smoothincrease1},
    \item A $C^1$ map $\widetilde{u}_\mathfrak{t} : [0,1]^3 \rightarrow \R^2$, such that
    \begin{align*}
        \vert \widetilde{u}_\mathfrak{t} - u_\mathfrak{t} \vert_{C^0} < 2\epsilon,
    \end{align*}
    for every $z \in [0, w_\mathfrak{t} ] \sqcup [1-w_\mathfrak{t}, 1]$,
    \begin{align*}
        \widetilde{u}_\mathfrak{t}(\, \cdot \, ,z) = \widetilde{u}_{\mathfrak{t}, 0}(\, \cdot \,, z),
    \end{align*}
    and for every $z \in [0,1]$, $\widetilde{u}_\mathfrak{t}( \, \cdot \, , z) : [0,1]^2 \rightarrow \R^2$ is a $C^1$ embedding, as provided by the second item of Lemma~\ref{lem:smoothemb1}.
\end{itemize}

Then we define $\widetilde{h}_\mathfrak{t}$ as 
$$\widetilde{h}_\mathfrak{t}(x,y,z) \coloneqq \big(\widetilde{u}_\mathfrak{t}(x,y,z), \widetilde{v}_\mathfrak{t}(z)\big),$$
so that $\widetilde{h}_\mathfrak{t} : [0,1]^3 \rightarrow \R^3$ is a $C^1$ embedding and 
$$\vert \widetilde{h}_\mathfrak{t} - h_\mathfrak{t}\vert_{C^0} \lesssim \epsilon.$$

We combine the $\widetilde{h}_\mathfrak{t}$'s, $\mathfrak{t} \in \mathcal{T}_1$, together into a map $\widetilde{h}_1 : U_1 \rightarrow M$ defined as
\begin{align*}
 \widetilde{h}_1(p) \coloneqq \psi_\mathfrak{t}^{-1} \circ \widetilde{h}_\mathfrak{t} \circ \varphi_\mathfrak{t} (p)
\end{align*}
for $\mathfrak{t} \in \mathcal{T}_1$ and $p \in U_\mathfrak{t}$. This expression makes sense for $\epsilon > 0$ small enough, so that the image of $\widetilde{h}_\mathfrak{t}$ is contained in $\psi_\mathfrak{t}(V_\mathfrak{t})$. Importantly, if $\mathfrak{t}, \mathfrak{t}' \in \mathcal{T}_1$ are such that $\mathfrak{t} \cap \mathfrak{t}' \neq \varnothing$ and if $p \in U_\mathfrak{t} \cap U_{\mathfrak{t}'}$, then
$$\psi_\mathfrak{t}^{-1} \circ \widetilde{h}_\mathfrak{t} \circ \varphi_\mathfrak{t} (p) = \psi_{\mathfrak{t}'}^{-1} \circ \widetilde{h}_{\mathfrak{t}'} \circ \varphi_{\mathfrak{t}'} (p) = \widetilde{h}_0(p),$$
which guarantees that $\widetilde{h}_1$ is well-defined.

By construction, $\widetilde{h}_1$ is a $C^1$ embedding sending $\mathcal{F}_0$ (restricted to $U_1$) to $\mathcal{F}_1$ and satisfies 
\begin{align*}
    d_{C^0}\big(h_{\vert U_1}, \widetilde{h}_1\big) \lesssim \epsilon.
\end{align*}
In summary, we have proved:

\begin{lem} \label{lem:skel1}
For every $\epsilon_1 > 0$, there exists a $C^1$ embedding $\widetilde{h}_1 : U_1 \rightarrow M$ satisfying
\begin{align*}
    d_{C^0}\big(h_{\vert U_1}, \widetilde{h}_1\big) < \epsilon_1, \qquad
    (\widetilde{h}_1)_* \mathcal{F}_0 = \mathcal{F}_1.
\end{align*}
\end{lem}

    \subsubsection{Smoothing near the \texorpdfstring{$2$}{2}-skeleton}

This step is more involved than the previous ones, as we need to start modifying the target foliation in a very careful way. This deformation will be graphical with respect to appropriately chosen coordinates. It will be crucial to make the dependence of the various objects and quantities as explicit as possible. 

Let $\epsilon > 0$ be such that $\epsilon \ll \min \{ w_\mathfrak{t}\  \vert \  \mathfrak{t} \in \mathcal{T}_2\}$, and consider $0 < \epsilon_1 \ll \epsilon$, to be chosen small enough below. We choose an embedding $\widetilde{h}_1 : U_1 \rightarrow M$ as in Lemma~\ref{lem:skel1} for $\epsilon_1$.

Let $\mathfrak{t} \in \mathcal{T}_2$. We consider the map
$$\widetilde{h}_{\mathfrak{t}, 1} : \varphi_\mathfrak{t}(\overline{U}_\mathfrak{t} \cap U_1) \subset N_\mathfrak{t} \rightarrow \R^3$$
defined by
$$\widetilde{h}_{\mathfrak{t}, 1} \coloneqq \psi_\mathfrak{t} \circ \widetilde{h}_1 \circ \varphi^{-1}_\mathfrak{t}.$$
We may assume that the set $B_\mathfrak{t}$ from Definition~\ref{def:cover} is the union of the faces $\{x=0\}$, $\{x=1\}$, $\{z=0\}$, and $\{z=1\}$ of $\partial[0,1]^3$.

The restriction of $\widetilde{h}_{\mathfrak{t}, 1}$ to $[0,1]^2 \times \big([0, 2w_\mathfrak{t}] \sqcup [1-2w_\mathfrak{t}, 1] \big)$
is of the form
$$\widetilde{h}_{\mathfrak{t}, 1}(x,y,z) = \big(\widetilde{u}_{\mathfrak{t}, 1}(x,y,z), \widetilde{v}_{\mathfrak{t}, 1}(x,z)\big),$$
where $\partial_z \widetilde{v}_{\mathfrak{t}, 1} > 0$, and each $\widetilde{u}_{\mathfrak{t}, 1}( \, \cdot \, ,z) : [0,1]^2\rightarrow \R^2$ is a $C^1$ embedding. 

The restriction of $\widetilde{h}_{\mathfrak{t}, 1}$ to $[0, 2w_\mathfrak{t}] \times [0,1]^2$
is of the form
$$\widetilde{h}_{\mathfrak{t}, 1}(x,y,z) = \big(\widetilde{u}^0_{\mathfrak{t}, 1}(x,y,z), \widetilde{v}^0_{\mathfrak{t}, 1}(z)\big),$$
where $\partial_z \widetilde{v}^0_{\mathfrak{t}, 1} > 0$.
Similarly, the restriction of $\widetilde{h}_{\mathfrak{t}, 1}$ to $[1-2w_\mathfrak{t}, 1] \times [0,1]^2$
is of the form
$$\widetilde{h}_{\mathfrak{t}, 1}(x,y,z) = \big(\widetilde{u}^1_{\mathfrak{t}, 1}(x,y,z), \widetilde{v}^1_{\mathfrak{t}, 1}(z)\big),$$
where $\partial_z \widetilde{v}^1_{\mathfrak{t}, 1} > 0$.

By construction, $\widetilde{v}^0_{\mathfrak{t}, 1}$ and $\widetilde{v}^1_{\mathfrak{t}, 1}$ coincide on $[0, 2w_\mathfrak{t}) \cup (1-2w_\mathfrak{t}, 1]$. However, these two maps might \emph{not} be equal on the whole of $[0,1]$, since they correspond to `transversal' smoothings of $h$ along different edges bounding the $2$-simplex $\mathfrak{t}$. We will need to carefully interpolate between those below. We further note that
\begin{align} \label{ineq:vtilde1}
    \big\vert \widetilde{v}^1_{\mathfrak{t}, 1}  - \widetilde{v}^0_{\mathfrak{t}, 1} \big\vert_{C^0} \leq \big\vert \widetilde{v}^1_{\mathfrak{t}, 1}  - v_\mathfrak{t} \big\vert_{C^0} + \big\vert \widetilde{v}^0_{\mathfrak{t}, 1}  - v_\mathfrak{t} \big\vert_{C^0} \lesssim \epsilon_1.
\end{align}

We also consider $0 < \overline{\epsilon} \ll \epsilon$ and an auxiliary smoothing $\overline{h}_\mathfrak{t}$ of $h_\mathfrak{t}$, which is a $C^1$ embedding $[0,1]^3 \rightarrow \R^3$ of the form
    \begin{align*}
        \overline{h}_\mathfrak{t}(x,y,z) = (\overline{u}_\mathfrak{t}(x,y,z), \overline{v}_\mathfrak{t}(z))
    \end{align*}
    for $(x,y,z) \in [0,1]^3$,
and such that 
$$\big\vert \overline{h}_\mathfrak{t} - h_\mathfrak{t}\big\vert_{C^0} < \overline{\epsilon} < \epsilon.$$
This embedding can be constructed as in the smoothing of $h$ near the $0$-skeleton. Note that it depends on $\overline{\epsilon}$, but not on $\epsilon_1$. We may further arrange that $\overline{h}_\mathfrak{t}$ is defined on a neighborhood $W_\mathfrak{t}$ of $[0,1]^3 \subset \R^3$, so that the image of $h_\mathfrak{t}$ is contained in the image of $\overline{h}_\mathfrak{t}$. Then, after shrinking $\epsilon_1$, we may also assume that the image of $\widetilde{h}_{\mathfrak{t},1}$ is contained in the image of $\overline{h}_\mathfrak{t}$, and by also shrinking $\overline{\epsilon}$, we may achieve
$$\big\vert \overline{h}_\mathfrak{t}^{-1} \circ \widetilde{h}_{\mathfrak{t},1} - \mathrm{id} \big\vert_{C^0} < \epsilon $$
on $N_\mathfrak{t}$. The role of this auxiliary smoothing $\overline{h}_\mathfrak{t}$ is to `straighten' the image of $[0,1]^3$ under $\widetilde{h}_{\mathfrak{t},1}$; the images of the lateral sides of that cube might be extremely `wiggly', which would complicate the extension of $\widetilde{h}_{\mathfrak{t},1}$, as we would like to perform a graphical deformation (in appropriate coordinates).

\begin{center}
\emph{We now fix the value of $\overline{\epsilon}$, and we will shrink $\epsilon_1$ further. Recall that they both depend on $\epsilon$, which was introduced first.}
\end{center}

We define:
$$\mathsf{h}_{\mathfrak{t},1} \coloneqq \overline{h}_\mathfrak{t}^{-1} \circ \widetilde{h}_{\mathfrak{t},1} : N_\mathfrak{t} \rightarrow W_\mathfrak{t},$$
which is of the form
$$\mathsf{h}_{\mathfrak{t}, 1}(x,y,z) = \big(\mathsf{u}_{\mathfrak{t}, 1}(x,y,z), \mathsf{v}_{\mathfrak{t}, 1}(z)\big),$$
where $\mathsf{v}_{\mathfrak{t}, 1} = \overline{v}^{-1}_\mathfrak{t} \circ \widetilde{v}_{\mathfrak{t},1}$. 

For $i \in \{0,1\}$, we also define
$$\mathsf{v}^i_{\mathfrak{t}, 1} \coloneqq \overline{v}^{-1}_\mathfrak{t} \circ \widetilde{v}^i_{\mathfrak{t},1},$$
so that for every $z \in [0,1]$
$$\big\vert \mathsf{v}^i_{\mathfrak{t},1}(z) - z\big\vert < \epsilon,$$
while by~\eqref{ineq:vtilde1} we have
\begin{align}
    \big\vert \mathsf{v}^0_{\mathfrak{t},1} - \mathsf{v}^1_{\mathfrak{t},1}\big\vert_{C^0} \lesssim \big\vert \partial_z \overline{v}^{-1}_\mathfrak{t} \big\vert_{C^0} \, \epsilon_1 \lesssim \epsilon_1.
\end{align}
Then, Lemma~\ref{lem:smoothemb2} provides a $C^1$ map $\mathsf{u}_\mathfrak{t} : [0,1]^3 \rightarrow \R^2$, such that for every $(x,y,z) \in \big([0,w_\mathfrak{t}] \sqcup [1-w_\mathfrak{t},1]\big) \times [0,1]^2 \cup [0,1]^2 \times \big([0,w_\mathfrak{t}] \sqcup [1-w_\mathfrak{t},1]\big)$, 
    \begin{align*}
        \mathsf{u}_\mathfrak{t}(x,y,z) = \mathsf{u}_{\mathfrak{t}, 1}(x,y,z),
    \end{align*}
    and for every $z \in [0,1]$,
    \begin{align*}
        \vert \mathsf{u}_\mathfrak{t}( \, \cdot\, , z) - \mathrm{id} \vert_{C^0} < 2\epsilon,
    \end{align*}
    and $\mathsf{u}_\mathfrak{t}( \, \cdot \, , z) : [0,1]^2 \rightarrow \R^2$ is a $C^1$ embedding.

Let $\tau_\mathfrak{t} : [0,1] \rightarrow [0,1]$ be a smooth cutoff function satisfying
\begin{itemize}
    \item $\tau_\mathfrak{t} =1$ on $[0, 2w_\mathfrak{t}]$ and $\tau_\mathfrak{t} = 0$ on $[1-2w_\mathfrak{t},1]$,
    \item $\tau_\mathfrak{t}$ is nonincreasing and $\tau_\mathfrak{t}' \geq -5$ (recall that $w_\mathfrak{t} \leq 0.1$).
\end{itemize}
We define
$$\mathsf{V}_\mathfrak{t}(x,y,z) = \mathsf{V}_\mathfrak{t}(x,z) \coloneqq \tau_\mathfrak{t}(x) \mathsf{v}^0_{\mathfrak{t},1}(z) + (1-\tau_\mathfrak{t}(x))\mathsf{v}^1_{\mathfrak{t},1}(z).$$
Note that $\tau_\mathfrak{t}$ only depends on the choice of clean cover, and
$$\partial_z \mathsf{V}_\mathfrak{t} > 0, \qquad \vert \partial_x \mathsf{V}_\mathfrak{t}\vert_{C^0} \leq 5 \big\vert \mathsf{v}^0_{\mathfrak{t},1} - \mathsf{v}^1_{\mathfrak{t},1}\big\vert_{C^0} \lesssim \epsilon_1.$$
Moreover, for every $(x,z) \in [0,1]^2$,
$$\vert\mathsf{V}_\mathfrak{t}(x,z) - z  \vert < \epsilon,$$
and for every $(x,z) \in \big([0,w_\mathfrak{t}] \sqcup [1-w_\mathfrak{t},1]\big) \times [0,1] \bigcup [0,1] \times \big([0,w_\mathfrak{t}] \sqcup [1-w_\mathfrak{t},1]\big)$,
$$\mathsf{V}_\mathfrak{t}(x,z) = \mathsf{v}_{\mathfrak{t},1}(x,z).$$
Therefore, the graphs of $z \mapsto \mathsf{V}_\mathfrak{t}( \, \cdot \, ,z)$, $z \in [0,1]$, define a $C^1$-foliation $\mathsf{F}_\mathfrak{t}$ on $W_\mathfrak{t}$ whose tangent plane field $T\mathsf{F}_\mathfrak{t}$ coincides with $H=\mathrm{span}\{\partial_x, \partial_y\}$ on $\big([0,w_\mathfrak{t}] \sqcup [1-w_\mathfrak{t},1]\big) \times [0,1]^2 \cup [0,1]^2 \times \big([0,w_\mathfrak{t}] \sqcup [1-w_\mathfrak{t},1]\big)$, and satisfies
\begin{align} \label{ineq:TF}
    d_{C^0}(T\mathsf{F}_\mathfrak{t}, H) \lesssim \epsilon_1.
\end{align}

We now define $\mathsf{h}_\mathfrak{t}$ as 
$$\mathsf{h}_\mathfrak{t}(x,y,z) \coloneqq \big(\mathsf{u}_\mathfrak{t}(x,y,z), \mathsf{V}_\mathfrak{t} \big(\mathsf{u}_\mathfrak{t}(x,y,z),z\big)\big),$$
so that $\mathsf{h}_\mathfrak{t} : [0,1]^3 \rightarrow W_\mathfrak{t}$ is a $C^1$ embedding sending the horizontal foliation on $[0,1]^3$ to $\mathsf{F}_\mathfrak{t}$, and 
$$\vert \mathsf{h}_\mathfrak{t} - \mathrm{id}\vert_{C^0} \leq 2\epsilon.$$

Finally, we set
$$\widetilde{h}_\mathfrak{t} \coloneqq \overline{h}_\mathfrak{t} \circ \mathsf{h}_\mathfrak{t} : [0,1]^3 \rightarrow \R^3.$$
By construction, the following hold:
\begin{itemize}
    \item For every $(x,y,z) \in \big([0,w_\mathfrak{t}] \sqcup [1-w_\mathfrak{t},1]\big) \times [0,1]^2 \cup [0,1]^2 \times \big([0,w_\mathfrak{t}] \sqcup [1-w_\mathfrak{t},1]\big)$, 
    $$\widetilde{h}_\mathfrak{t}(x,y,z) = \widetilde{h}_{\mathfrak{t},1}(x,y,z),$$
    \item There is a function $\omega_\mathfrak{t} : [0,\infty) \rightarrow [0,\infty)$ with $\lim_{s \rightarrow 0} \omega_\mathfrak{t}(s) =0$ (obtained from the modulus of continuity of $h_\mathfrak{t}$) such that
    \begin{align*}
        \big\vert \widetilde{h}_\mathfrak{t} - h_\mathfrak{t} \big\vert_{C^0}  &= \big\vert \overline{h}_\mathfrak{t} \circ \mathsf{h}_\mathfrak{t} - h_\mathfrak{t} \big\vert_{C^0} \\
        & \leq  \big\vert \overline{h}_\mathfrak{t} \circ \mathsf{h}_\mathfrak{t} - h_\mathfrak{t} \circ \mathsf{h}_\mathfrak{t} \big\vert_{C^0} + \big\vert h_\mathfrak{t} \circ \mathsf{h}_\mathfrak{t} - h_\mathfrak{t} \big\vert_{C^0} \\
        &\leq \epsilon + \omega_\mathfrak{t}(\epsilon)
    \end{align*}
    \item The image of the horizontal foliation on $[0,1]^3$ by $\widetilde{h}_{\mathfrak{t}}$, denoted by $\widetilde{\mathcal{F}}_\mathfrak{t}$, is the image of $\mathsf{F}_\mathfrak{t}$ by $\overline{h}_\mathfrak{t}$. Writing $H = \mathrm{span}\{ \partial_x, \partial_y\}$ as before, we have
    \begin{align*}
        d_{C^0}\big( T \mathcal{F}_\mathfrak{t}, H\big) &= d_{C^0}\big( d\overline{h}_\mathfrak{t}(T \mathsf{F}_\mathfrak{t}), d\overline{h}_\mathfrak{t}(H)\big) \\
       & \leq \vert d\overline{h}_\mathfrak{t}\vert_{C^0} \, d_{C^0}\big( T \mathsf{F}_\mathfrak{t}, H\big) \\
       &\lesssim \epsilon_1,
    \end{align*}
    so we can shrink $\epsilon_1$ to ensure
    $$d_{C^0}\big( T \mathcal{F}_\mathfrak{t}, H\big) < \epsilon.$$
\end{itemize}

We now combine the $\widetilde{h}_\mathfrak{t}$'s, $\mathfrak{t} \in \mathcal{T}_2$, together into a map $\widetilde{h}_2 : U_2 \rightarrow M$ defined as
\begin{align*}
 \widetilde{h}_2(p) \coloneqq \psi_\mathfrak{t}^{-1} \circ \widetilde{h}_\mathfrak{t} \circ \varphi_\mathfrak{t} (p)
\end{align*}
for $\mathfrak{t} \in \mathcal{T}_2$ and $p \in U_\mathfrak{t}$. As before, this expression makes sense for $\epsilon > 0$ small enough, so that the image of $\widetilde{h}_\mathfrak{t}$ is contained in $\psi_\mathfrak{t}(V_\mathfrak{t})$. Moreover, if $\mathfrak{t}, \mathfrak{t}' \in \mathcal{T}_2$ are such that $\mathfrak{t} \cap \mathfrak{t}' \neq \varnothing$ and if $p \in U_\mathfrak{t} \cap U_{\mathfrak{t}'}$, then
$$\psi_\mathfrak{t}^{-1} \circ \widetilde{h}_\mathfrak{t} \circ \varphi_\mathfrak{t} (p) = \psi_{\mathfrak{t}'}^{-1} \circ \widetilde{h}_{\mathfrak{t}'} \circ \varphi_{\mathfrak{t}'} (p) = \widetilde{h}_1(p),$$
which guarantees that $\widetilde{h}_2$ is well-defined.

By construction, $\widetilde{h}_2$ is a $C^1$ embedding sending $\mathcal{F}_0$ (restricted to $U_2$) to a foliation $\widetilde{\mathcal{F}}_1$ satisfying
$$d_{C^0}\big(T \mathcal{F}_1, T\widetilde{\mathcal{F}}_1 \big) \lesssim \epsilon,$$
and which coincides with $\mathcal{F}_1$ on a neighborhood of $h(\mathcal{T}_1)$ (here, the inequality is independent of $\epsilon_1$ and $\overline{\epsilon}$).
Moreover,
\begin{align*}
    d_{C^0}\big(h_{\vert U_2}, \widetilde{h}_2\big) \leq \omega_2(\epsilon)
\end{align*}
for some function $\omega_2 : [0,\infty) \rightarrow [0,\infty)$ with $\lim_{t \rightarrow 0} \omega_2(t) = 0$, which only depends on $h$ and the clean cover. 

In summary, we have proved:

\begin{lem} \label{lem:skel2}
For every $\epsilon_2 > 0$, there exists a $C^1$ embedding $\widetilde{h}_2 : U_2 \rightarrow M$ satisfying
\begin{align*}
    d_{C^0}\big(h_{\vert U_2}, \widetilde{h}_2\big) < \epsilon_2,
\end{align*}
and $(\widetilde{h}_2)_* \mathcal{F}_0 \eqqcolon \widetilde{\mathcal{F}}_1$ satisfies
\begin{align*}
    d_{C^0}\big({T \mathcal{F}_1}, {T\widetilde{\mathcal{F}}_1} \big) < \epsilon_2, \qquad \widetilde{\mathcal{F}}_1 = \mathcal{F}_1 \textrm{ near } h(\mathcal{T}_1).
\end{align*}
\end{lem}

    \subsubsection{Smoothing on the \texorpdfstring{$3$}{3}-cells}

The final step is similar to the previous one, but relative to the vertical boundaries of the $3$-cells.

Let $\epsilon > 0$ be such that $\epsilon \ll \min \{ w_\mathfrak{t}\  \vert \  \mathfrak{t} \in \mathcal{T}_3\}$, and $0 < \epsilon_2 \ll \epsilon$, to be chosen small enough below. We choose an embedding $\widetilde{h}_2 : U_2 \rightarrow M$ as in Lemma~\ref{lem:skel2} for $\epsilon_2$.

Let $\mathfrak{t} \in \mathcal{T}_3$ and consider the map
$$\widetilde{h}_{\mathfrak{t}, 2} : \varphi_\mathfrak{t}(\overline{U}_\mathfrak{t} \cap U_2) \subset N_\mathfrak{t} \rightarrow \R^3$$
defined by
$$\widetilde{h}_{\mathfrak{t},2} \coloneqq \psi_\mathfrak{t} \circ \widetilde{h}_2 \circ \varphi^{-1}_\mathfrak{t}.$$
Note that $B_\mathfrak{t} = \partial [0,1]^3$.

As in the previous step, we also consider $0 < \overline{\epsilon} \ll \epsilon$ and an auxiliary smoothing $\overline{h}_\mathfrak{t}$ of $h_\mathfrak{t}$, which is a $C^1$ embedding $[0,1]^3 \rightarrow \R^3$ of the form
    \begin{align*}
        \overline{h}_\mathfrak{t}(x,y,z) = (\overline{u}_\mathfrak{t}(x,y,z), \overline{v}_\mathfrak{t}(z))
    \end{align*}
    for $(x,y,z) \in [0,1]^3$,
and such that 
$$\big\vert \overline{h}_\mathfrak{t} - h_\mathfrak{t}\big\vert_{C^0} < \overline{\epsilon} < \epsilon.$$
This embedding can be constructed as in the smoothing of $h$ near the $0$-skeleton, and it depends on $\overline{\epsilon}$, but not on $\epsilon_2$. As before, we may further arrange that $\overline{h}_\mathfrak{t}$ is defined on a neighborhood $W_\mathfrak{t}$ of $[0,1]^3 \subset \R^3$, so that the images of $h_\mathfrak{t}$ and $\widetilde{h}_{\mathfrak{t},2}$ are contained in the image of $\overline{h}_\mathfrak{t}$, and so that
$$\big\vert \overline{h}_\mathfrak{t}^{-1} \circ \widetilde{h}_{\mathfrak{t},2} - \mathrm{id} \big\vert_{C^0} < \epsilon,$$
where the inequality is to be understood on $N_\mathfrak{t}$.

\begin{center}
\emph{We now consider the value of $\overline{\epsilon}$ fixed, and we will shrink $\epsilon_2$ further. Recall that they both depend on $\epsilon$, which was introduced first.}
\end{center}

Let $N^\rho_\mathfrak{t}$ denote the $\ell^\infty$-neighborhood of $\partial[0,1]^3$ of radius $\rho > 0$. We define:
$$\mathsf{h}_{\mathfrak{t},2} \coloneqq \overline{h}_\mathfrak{t}^{-1} \circ \widetilde{h}_{\mathfrak{t},2} : N_\mathfrak{t} \rightarrow W_\mathfrak{t},$$
which is of the form
$$\mathsf{h}_{\mathfrak{t}, 2}(x,y,z) = \big(\mathsf{u}_{\mathfrak{t}, 2}(x,y,z), \mathsf{v}_{\mathfrak{t},2}(x,y,z)\big).$$
Let $\mathsf{F}_{\mathfrak{t},2}$ denote the image of the horizontal foliation on $[0,1]^3$ (restricted to $N^{2w_\mathfrak{t}}_\mathfrak{t}$) by $\mathsf{h}_{\mathfrak{t}, 2}$. By assumption,
$$d_{C^0}\big(T \mathsf{F}_{\mathfrak{t},2}, H\big) \lesssim \epsilon_2,$$
where $H= \mathrm{span}\{\partial_x, \partial_y\}$.
Then, there exists a $C^1$ map $\mathsf{V}_{\mathfrak{t},2} : N^{2w_\mathfrak{t}}_\mathfrak{t} \rightarrow \R$ such that the graphs of $\mathsf{V}_{\mathfrak{t},2}(\, \cdot \, , z)$, $z \in [0,1]$, describe (subsets of) the leaves of $\mathsf{F}_{\mathfrak{t},2}$, and 
\begin{align*}
    \mathsf{v}_{\mathfrak{t},2}(x,y,z) = \mathsf{V}_{\mathfrak{t},2} \big(\mathsf{u}_{\mathfrak{t}, 2}(x,y,z), z \big).
\end{align*}
By the assumptions on $\widetilde{\mathcal{F}}_1$,
$$\partial_z \mathsf{V}_{\mathfrak{t},2}> 0, \qquad \vert \partial_x \mathsf{V}_{\mathfrak{t},2} \vert_{C^0}, \vert \partial_y \mathsf{V}_{\mathfrak{t},2} \vert_{C^0} \lesssim \epsilon_2,$$
and $\mathsf{F}_{\mathfrak{t},2}$ is tangent to $H$ near $\{z=0\} \cup \{z=1\}$. Moreover, $\mathsf{u}_{\mathfrak{t}, 2}( \, \cdot \, , z)$, $z \in [0,1]$, is a family of $C^1$ embeddings.

Lemma~\ref{lem:smoothemb2} provides a $C^1$ map $\mathsf{u}_\mathfrak{t} : [0,1]^3 \rightarrow \R^2$, such that for every $(x,y,z) \in N^{w_\mathfrak{t}}_\mathfrak{t}$, 
    \begin{align*}
        \mathsf{u}_\mathfrak{t}(x,y,z) = \mathsf{u}_{\mathfrak{t}, 2}(x,y,z),
    \end{align*}
    and for every $z \in [0,1]$, 
    \begin{align*}
        \vert \mathsf{u}_\mathfrak{t}( \, \cdot \, , z) - \mathrm{id} \vert_{C^0} < 2\epsilon,
    \end{align*}
    and $\mathsf{u}_\mathfrak{t}( \, \cdot \, , z) : [0,1]^2 \rightarrow \R^2$ is a $C^1$ embedding.

Let $\tau_\mathfrak{t} : [0,1]^2 \rightarrow [0,1]$ be a smooth cutoff function such that $\tau_\mathfrak{t} = 1$ on $[0,1]^2 \setminus (w_\mathfrak{t}, 1-w_\mathfrak{t})^2$ and $\tau_\mathfrak{t} = 0$ on $[2w_\mathfrak{t}, 1-2w_\mathfrak{t}]^2$. This choice of $\tau_\mathfrak{t}$ only depends on $w_\mathfrak{t}$ and not on $\epsilon_2$. For $z \in [0,1]$, we write
$$\mathsf{V}^0_{\mathfrak{t},2}(z) \coloneqq \mathsf{V}_{\mathfrak{t},2}(0,0,z)$$
and we define
$$\mathsf{V}_{\mathfrak{t}}(x,y,z) \coloneqq \tau_\mathfrak{t}(x,y) \mathsf{V}_{\mathfrak{t},2}(x,y,z) + (1-\tau_\mathfrak{t}(x,y)) \mathsf{V}^0_{\mathfrak{t},2}(z),$$
so that
\begin{align*}
    \partial_z \mathsf{V}_{\mathfrak{t}} > 0, \qquad \vert\partial_x \mathsf{V}_{\mathfrak{t}} \vert_{C^0}, \vert \partial_y \mathsf{V}_{\mathfrak{t}} \vert_{C^0} \lesssim \epsilon_2.
\end{align*}
Then, the graphs of $\mathsf{V}_{\mathfrak{t}}(\, \cdot \, , z)$, $z \in [0,1]$, describe a foliation $\mathsf{F}_\mathfrak{t}$ on $[0,1]^3$ which coincides with $\mathsf{F}_{\mathfrak{t},2}$ on $N^{w_\mathfrak{t}}_\mathfrak{t}$, and which satisfies
$$d_{C^0}\big(T \mathsf{F}_{\mathfrak{t}}, H\big) \lesssim \epsilon_2.$$
Setting
$$\mathsf{v}_\mathfrak{t}(x,y,z) \coloneqq \mathsf{V}_\mathfrak{t}\big(\mathsf{u}_\mathfrak{t}(x,y,z), z\big)$$
for $(x,y,z) \in [0,1]^3$, we have
$$\vert \mathsf{v}_\mathfrak{t}(x,y,z)-z\vert \lesssim \epsilon.$$

We now proceed exactly as in the previous step, and we define $\mathsf{h}_\mathfrak{t}$ as 
$$\mathsf{h}_\mathfrak{t}(x,y,z) \coloneqq \big(\mathsf{u}_\mathfrak{t}(x,y,z), \mathsf{v}_\mathfrak{t} (x,y,z)\big),$$
so that $\mathsf{h}_\mathfrak{t} : [0,1]^3 \rightarrow W_\mathfrak{t}$ is a $C^1$ embedding sending the horizontal foliation on $[0,1]^3$ to $\mathsf{F}_\mathfrak{t}$, and 
$$\vert \mathsf{h}_\mathfrak{t} - \mathrm{id}\vert_{C^0} \leq 2\epsilon.$$

Finally, we define
$$\widetilde{h}_\mathfrak{t} \coloneqq \overline{h}_\mathfrak{t} \circ \mathsf{h}_\mathfrak{t} : [0,1]^3 \rightarrow \R^3.$$
By construction, it satisfies:
\begin{itemize}
    \item For every $(x,y,z) \in N^{w_\mathfrak{t}}_\mathfrak{t}$, 
    $$\widetilde{h}_\mathfrak{t}(x,y,z) = \widetilde{h}_{\mathfrak{t},2}(x,y,z),$$
    \item There is a function $\omega_\mathfrak{t} : [0,\infty) \rightarrow [0,\infty)$ with $\lim_{s \rightarrow 0} \omega_\mathfrak{t}(s) =0$ such that
    \begin{align*}
        \big\vert \widetilde{h}_\mathfrak{t} - h_\mathfrak{t} \big\vert_{C^0}  \leq \omega_\mathfrak{t}(\epsilon),
    \end{align*}
    \item The image of the horizontal foliation on $[0,1]^3$ by $\widetilde{h}_{\mathfrak{t}}$, denoted by $\widetilde{\mathcal{F}}_\mathfrak{t}$, satisfies
    \begin{align*}
        d_{C^0}\big( T \mathcal{F}_\mathfrak{t}, H\big) \lesssim \epsilon_2,
    \end{align*}
    so we can shrink $\epsilon_2$ to ensure
    $$d_{C^0}\big( T \mathcal{F}_\mathfrak{t}, H\big) < \epsilon.$$
\end{itemize}

We now combine the $\widetilde{h}_\mathfrak{t}$'s, $\mathfrak{t} \in \mathcal{T}_3$, together into a map $\widetilde{h} : U_3=M \rightarrow M$ defined as
\begin{align*}
 \widetilde{h}(p) \coloneqq \psi_\mathfrak{t}^{-1} \circ \widetilde{h}_\mathfrak{t} \circ \varphi_\mathfrak{t} (p)
\end{align*}
for $\mathfrak{t} \in \mathcal{T}_3$ and $p \in U_\mathfrak{t}$. As before, this expression makes sense for $\epsilon > 0$ small enough so that the image of $\widetilde{h}_\mathfrak{t}$ is contained in $\psi_\mathfrak{t}(V_\mathfrak{t})$. Moreover, if $\mathfrak{t}, \mathfrak{t}' \in \mathcal{T}_3$ are such that $\mathfrak{t} \cap \mathfrak{t}' \neq \varnothing$ and if $p \in U_\mathfrak{t} \cap U_{\mathfrak{t}'}$, then
$$\psi_\mathfrak{t}^{-1} \circ \widetilde{h}_\mathfrak{t} \circ \varphi_\mathfrak{t} (p) = \psi_{\mathfrak{t}'}^{-1} \circ \widetilde{h}_{\mathfrak{t}'} \circ \varphi_{\mathfrak{t}'} (p) = \widetilde{h}_2(p),$$
which guarantees that $\widetilde{h}$ is well-defined.

By construction, $\widetilde{h}$ is a $C^1$ diffeomorphism sending $\mathcal{F}_0$ to a foliation $\widetilde{\mathcal{F}}_1$ satisfying
\begin{align} \label{eq:distfol1}
    d_{C^0}\big(T \mathcal{F}_1, T\widetilde{\mathcal{F}}_1 \big) \lesssim \epsilon,
\end{align}
independently of $\epsilon_2$ and $\overline{\epsilon}$, and
\begin{align*}
    d_{C^0}\big(h, \widetilde{h}\big) \leq \omega(\epsilon)
\end{align*}
for some function $\omega : [0,\infty) \rightarrow [0,\infty)$ with $\lim_{t \rightarrow 0} \omega(t) = 0$, which only depends on $h$ and the clean cover. Finally, $\widetilde{h}$ can be approximated in the $C^1$ topology by a \emph{smooth} diffeomorphism such that~\eqref{eq:distfol1} still holds. This concludes the proof of Theorem~\ref{thmintrobeta:approx}. \qed

        \subsection{Bifoliated homeomorphisms}

We now explain how to adapt the previous strategy to the case of bifoliated homeomorphisms.

\medskip
Let $(\mathcal{F}, \mathcal{G})$ be a $C^1$-bifoliation on $M$.

\begin{defn}
A $C^1$ coordinate system $(x,y,z)$ near $p \in M$ is \textbf{adapted to $(\mathcal{F}, \mathcal{G})$} if in these coordinates, 
\begin{align} \label{eq:bifoladapted}
T \mathcal{F} = \mathrm{span}\big\{ \partial_x, \partial_y\}, \qquad T \mathcal{G} = \mathrm{span}\big\{ \partial_x, \partial_z\}.
\end{align}
\end{defn}

In coordinates adapted to bifoliations, a bifoliated homeomorphism is of the form
$$h(x,y,z) = \big(a(x,y,z), b(y), c(z)\big),$$
where the functions $x \mapsto a(x, \cdot\, , \cdot\, )$, $b$ and $c$ are strictly monotone.

We now fix $\delta_0$ so that every open ball of radius less than $\delta_0$ in $M$ admits coordinates adapted to $(\mathcal{F}, \mathcal{G})$. Let $\mathcal{T}$ be a $\delta$-fine triangulation of $M$ in general position with $\mathcal{F}$ and $\mathcal{G}$, with $0 < \delta < \delta_0$. As before, this can be achieved by Thurston's jiggling (see~\cite{V16} for an argument that generalizes well to the case of multiple line/plane fields).

We would like to adapt the definition of clean covers (Definition~\ref{def:cover}) to the bifoliated case. However, it is not possible to find small neighborhoods of the cells of $\mathcal{T}$ which are diffeomorphic to a standard bifoliated cube. Instead, we will consider two clean covers, one for each foliation, which are compatible in a suitable sense. 

Let $\big(\mathcal{U}^\mathcal{F},\boldsymbol{\varphi}^\mathcal{F}\big)$ and $\big(\mathcal{U}^\mathcal{G},\boldsymbol{\varphi}^\mathcal{G}\big)$ be clean covers of $M$ adapted to $\mathcal{F}$ and $\mathcal{G}$, respectively, and modeled on $\mathcal{T}$. For $\mathfrak{t} \in \mathcal{T}$, we set
$$U_\mathfrak{t} \coloneqq U^\mathcal{F}_\mathfrak{t} \cap U^\mathcal{G}_\mathfrak{t}.$$
We say that $\big(\mathcal{U}^\mathcal{F},\boldsymbol{\varphi}^\mathcal{F}\big)$ and $\big(\mathcal{U}^\mathcal{G},\boldsymbol{\varphi}^\mathcal{G}\big)$ are \textbf{compatible} if for every $\mathfrak{t} \in \mathcal{T}$, the set 
        $$\check{N}^\mathcal{F}_\mathfrak{t} \coloneqq \varphi^\mathcal{F}_\mathfrak{t} \left( \overline{U}^\mathcal{F}_\mathfrak{t} \cap \bigcup_{\mathfrak{t}' \in \partial \mathfrak{t}} U_{\mathfrak{t}'} \right) \subset [0,1]^3$$
contains the $\ell^\infty$-neighborhood of radius $2w^\mathcal{F}_\mathfrak{t}$ of $B^\mathcal{F}_\mathfrak{t}$, and the set 
        $$\check{N}^\mathcal{G}_\mathfrak{t} \coloneqq \varphi^\mathcal{G}_\mathfrak{t} \left( \overline{U}^\mathcal{G}_\mathfrak{t} \cap \bigcup_{\mathfrak{t}' \in \partial \mathfrak{t}} U_{\mathfrak{t}'} \right) \subset [0,1]^3$$
contains the $\ell^\infty$-neighborhood of radius $2w^\mathcal{G}_\mathfrak{t}$ of $B^\mathcal{G}_\mathfrak{t}$.

As before, for $0 \leq i \leq 3$, we write
    $$U_i \coloneqq \bigcup_{\mathfrak{t} \in \mathcal{T}_i} U_\mathfrak{t},$$
so that $U_i$ is a neighborhood of the $i$-skeleton of $\mathcal{T}$. Compatible clean covers can be constructed by induction on the skeleton of $\mathcal{T}$:

\begin{lem} \label{lem:cleancov2}
For every $0 < \delta < \delta_0$, there exists a compatible pair of $\delta$-clean covers of $M$ adapted to $\mathcal{F}$ and $\mathcal{G}$, respectively, and modeled on some sufficiently fine common triangulation of $M$.
\end{lem}

We now consider two (co)orientable $C^1$-bifoliations $(\mathcal{F}_0, \mathcal{G}_0)$ and $(\mathcal{F}_1, \mathcal{G}_1)$ as well as a bifoliated homeomorphism $h : M \rightarrow M$ between them. We may choose (co)orientations so that $h$ sends the (co)orientation of $\mathcal{F}_0$ (resp.~$\mathcal{G}_0$) to that of $\mathcal{F}_1$ (resp.~$\mathcal{G}_1$). We do not require that $h$ preserves the orientation on $M$.

We consider a pair $\big(\mathcal{U}^{\mathcal{F}_0},\boldsymbol{\varphi}^{\mathcal{F}_0}\big)$ and $\big(\mathcal{U}^{\mathcal{G}_0},\boldsymbol{\varphi}^{\mathcal{G}_0}\big)$ of compatible clean covers for $(\mathcal{F}_0, \mathcal{G}_0)$ modeled on a $\delta$-fine triangulation $\mathcal{T}$ in general position with $(\mathcal{F}_0, \mathcal{G}_0)$. For each $\mathfrak{t} \in \mathcal{T}$, we choose an open set 
$$\overline{U^{\mathcal{F}_0}_\mathfrak{t} \cup U^{\mathcal{G}_0}_\mathfrak{t}} \subset \widehat{U}_\mathfrak{t} $$
together with a diffeomorphism $\widehat{\varphi}_\mathfrak{t} : \widehat{U}_\mathfrak{t} \hookrightarrow (0,1)^3$ defining coordinates adapted to $(\mathcal{F}_0, \mathcal{G}_0)$ (this is achievable after possibly shrinking $\delta$). We also choose an open set $h\big(\widehat{U}_\mathfrak{t}\big) \subset V_\mathfrak{t}$ together with coordinates $\psi_\mathfrak{t} : V_\mathfrak{t} \hookrightarrow \R^3$ adapted to $(\mathcal{F}_1, \mathcal{G}_1)$. We set
$$h_\mathfrak{t} \coloneqq \psi_\mathfrak{t}\circ h \circ \widehat{\varphi}^{-1}_\mathfrak{t} : (0,1)^3 \rightarrow \R^3,$$
which is of the form
$$h_\mathfrak{t}(x,y,z) = \big(a_\mathfrak{t}(x,y,z), b_\mathfrak{t}(y), c_\mathfrak{t}(z)\big).$$

We can proceed as in the proofs of Lemma~\ref{lem:skel0} and Lemma~\ref{lem:skel1}, and use items 1 and 2 of Lemma~\ref{lem:smoothincrease2}, to smooth the maps $a_\mathfrak{t}$, $b_\mathfrak{t}$, $c_\mathfrak{t}$ near the $1$-skeleton of $\mathcal{T}$ and obtain:

\begin{lem} \label{lem:biskel1}
For every $\epsilon_1 > 0$, there exists a $C^1$ embedding $\widetilde{h}_1 : U_1 \rightarrow M$ satisfying
\begin{align*}
    d_{C^0}\big(h_{\vert U_1}, \widetilde{h}_1\big) < \epsilon_1, \qquad
    (\widetilde{h}_1)_* \mathcal{F}_0 = \mathcal{F}_1, \qquad
    (\widetilde{h}_1)_* \mathcal{G}_0 = \mathcal{G}_1.
\end{align*}
\end{lem}

We now explain how to adapt the smoothing near the $2$-skeleton. We can apply the same strategy as in the proof of Lemma~\ref{lem:skel2} to first define new $C^1$-foliations $\widetilde{\mathcal{F}}_1$ and $\widetilde{\mathcal{G}}_1$ on $U_2$ which coincide with $\mathcal{F}_1$ and $\mathcal{G}_1$ near $h(\mathcal{T}_1)$, and with $C^0$-close tangent plane fields. We then smooth $h$ on $U_2$ so that it matches some smoothing $\widetilde{h}_1$ provided by Lemma~\ref{lem:biskel1} near the $1$-skeleton, and such that this smoothing $\widetilde{h}_2$ sends $\mathcal{F}_0$ to $\widetilde{\mathcal{F}}_1$ and $\mathcal{G}_0$ to $\widetilde{\mathcal{G}}_1$.

More precisely, we fix some $\epsilon > 0$, and some auxiliary $0 < \epsilon_1 \ll \epsilon$ together with a smoothing $\widetilde{h}_1$ of $h$ on $U_1$ provided by Lemma~\ref{lem:biskel1}. We consider a $2$-simplex $\mathfrak{t} \in \mathcal{T}_2$, and we write
$$\widetilde{h}_{\mathfrak{t},1} \coloneqq \psi_\mathfrak{t} \circ \widetilde{h}_1 \circ \widehat{\varphi}^{-1}_\mathfrak{t},$$
which is defined on $\widehat{N}_{\mathfrak{t},1} = \widehat{\varphi}_\mathfrak{t}\left(\widehat{U}_\mathfrak{t} \cap U_1 \right)$ and is of the form
$$\widetilde{h}_{\mathfrak{t},1}(x,y,z) = \big(\widetilde{a}_{\mathfrak{t},1}(x,y,z), \widetilde{b}_{\mathfrak{t},1}(y), \widetilde{c}_{\mathfrak{t},1}(z) \big),$$
where $\partial_x \widetilde{a}_{\mathfrak{t},1} > 0$, $\partial_y \widetilde{b}_{\mathfrak{t},1} > 0$, and $\partial_z \widetilde{c}_{\mathfrak{t},1} > 0$.

Then, we apply the proof of Lemma~\ref{lem:skel2} (and possibly shrink $\epsilon_1$) to obtain $C^1$-foliations $\widetilde{\mathcal{F}}_\mathfrak{t}$ and $\widetilde{\mathcal{G}}_\mathfrak{t}$ on a neighborhood of the closure of $W_\mathfrak{t} \coloneqq \psi_\mathfrak{t} \circ h(U_\mathfrak{t})$ in $\R^3$. Writing $H = \mathrm{span}\{ \partial_x, \partial_y\}$ and $K = \mathrm{span}\{ \partial_x, \partial_z\}$, these foliations are constructed so that
$$d_{C^0}\big( T\widetilde{\mathcal{F}}_\mathfrak{t}, H \big) < \epsilon, \qquad d_{C^0}\big( T\widetilde{\mathcal{G}}_\mathfrak{t}, K \big) < \epsilon,$$
and $T\widetilde{\mathcal{F}}_\mathfrak{t} = H$ and $T\widetilde{\mathcal{G}}_\mathfrak{t}= K$ near the boundary of $W_\mathfrak{t}$.

Moreover, they can be described as families of graphs of maps
$$(x, y) \mapsto F_\mathfrak{t}(x,y,z), \qquad (x, z) \mapsto G_\mathfrak{t}(x,y,z)$$
with $\partial_z F_\mathfrak{t} > 0$ and $\partial_y G_\mathfrak{t} > 0$, and such that 
$$(F_\mathfrak{t}(x,y,z), G_\mathfrak{t}(x,y,z)) = \big(\widetilde{c}_{\mathfrak{t},1}(z) , \widetilde{b}_{\mathfrak{t},1}(y) \big)$$
near the boundary of $W_\mathfrak{t}$. Then, there exists a $C^1$ embedding $\Phi_\mathfrak{t} : W_\mathfrak{t} \hookrightarrow \R^3$ defining coordinates adapted to $(\widetilde{\mathcal{F}}_\mathfrak{t}, \widetilde{\mathcal{G}}_\mathfrak{t})$ and in which the maps $F_\mathfrak{t}$ and $G_\mathfrak{t}$ simply become $z$ and $y$, respectively. Therefore, 
$$\Phi_\mathfrak{t} \circ \widetilde{h}_{\mathfrak{t},1}(x,y,z) = \big(\overline{a}_{\mathfrak{t},1}(x,y,z), y, z \big)$$
on $\widehat{N}_{\mathfrak{t},1}$, where $\partial_x \overline{a}_{\mathfrak{t},1} > 0$ and $\overline{a}_{\mathfrak{t},1}$ is $C^0$-close to the first coordinate of $\Phi_\mathfrak{t} \circ h_\mathfrak{t}$. We can then use the third item of Lemma~\ref{lem:smoothincrease2} to extend $\overline{a}_{\mathfrak{t},1}$ to a $C^1$ map $\overline{a}_{\mathfrak{t}} : \widehat{\varphi}_\mathfrak{t}(U_\mathfrak{t}) \rightarrow \R$ satisfying $\partial_x \overline{a}_{\mathfrak{t}} > 0$, and which is $C^0$-close to the first coordinate of $\Phi_\mathfrak{t} \circ h_\mathfrak{t}$. We then define
$$\widetilde{h}_\mathfrak{t}(x,y,z) = \Phi^{-1}_\mathfrak{t}\big( \overline{a}_{\mathfrak{t}}(x,y,z), y, z \big)$$
which is the desired extension of $\widetilde{h}_{\mathfrak{t},1}$ over $\widehat{\varphi}_\mathfrak{t}(U_\mathfrak{t})$. Combining these maps together for $\mathfrak{t} \in \mathcal{T}_2$, we obtain:

\begin{lem} \label{lem:biskel2}
For every $\epsilon_2 > 0$, there exists a $C^1$ embedding $\widetilde{h}_2 : U_2 \rightarrow M$ satisfying
\begin{align*}
    d_{C^0}\big(h_{\vert U_2}, \widetilde{h}_2\big) < \epsilon_2,
\end{align*}
and $(\widetilde{h}_2)_* \mathcal{F}_0 \eqqcolon \widetilde{\mathcal{F}}_1$ and $(\widetilde{h}_2)_* \mathcal{G}_0 \eqqcolon \widetilde{\mathcal{G}}_1$ satisfy
\begin{align*}
    d_{C^0}\big({T \mathcal{F}_1}, {T\widetilde{\mathcal{F}}_1} \big) < \epsilon_2, \qquad d_{C^0}\big({T \mathcal{G}_1}, {T\widetilde{\mathcal{G}}_1} \big) < \epsilon_2, \qquad \widetilde{\mathcal{F}}_1 = \mathcal{F}_1 \textrm{ and } \widetilde{\mathcal{G}}_1 = \mathcal{G}_1 \textrm{ near } h(\mathcal{T}_1).
\end{align*}
\end{lem}

To finish the proof of Theorem~\ref{thmintrobeta:bifolapprox}, we then extend such a smoothing $\widetilde{h}_2$ over the $3$-cells by proceeding as in the foliated case. We first extend $\widetilde{\mathcal{F}}_1$ and $\widetilde{\mathcal{G}}_1$ on the $3$-cells so that their plane fields remain $C^0$-close to those of $\mathcal{F}_1$ and $\mathcal{G}_1$, respectively, and we extend $\widetilde{h}_2$ so that it sends $\mathcal{F}_0$ (resp.~$\mathcal{G}_0$) to $\widetilde{\mathcal{F}}_1$ (resp.~ $\widetilde{\mathcal{G}}_1$) while remaining sufficiently $C^0$-close to $h$. The extension in the direction of $\mathcal{F}_0 \cap \mathcal{G}_0$ relies on item 4 of Lemma~\ref{lem:smoothincrease2}. At this point, the details should be clear and are left to the reader. This concludes the proof of Theorem~\ref{thmintrobeta:bifolapprox}. \qed

        \subsection{Stronger versions}

Our methods can be generalized to prove stronger and more precise versions of the previous results. We collect them in this section and leave the proofs to the interested reader. We will not need these versions for our main applications, but they might be of independent interest.

\begin{thm}[Foliated smoothing, strong version] \label{thm:strongapprox}
    Let $\mathcal{F}_0$ and $\mathcal{F}_1$ be two orientable $C^1$-foliations on $M$, and $h : M \rightarrow M$ be a homeomorphism sending the leaves of $\mathcal{F}_0$ to leaves of $\mathcal{F}_1$. Then, there exists a topological isotopy $(h_t)_{0\leq t\leq 1}$ such that
    \begin{enumerate}
        \item $h_0=h$,
        \item The map 
        $$\begin{aligned}
        (0,1] \times M &\longrightarrow M \\
        (t,p) &\longmapsto h_t(p)
        \end{aligned}$$
        is smooth,
        \item The map 
        $$t \in [0,1] \longmapsto 
                                    \begin{cases}
                                    (h_t)_* (T \mathcal{F}_0) & \text{if } t>0,\\
                                    T \mathcal{F}_1 & \text{if } t=0,
                                    \end{cases}$$
    is continuous.
    \end{enumerate}
\end{thm}

This implies that a topological conjugation between orientable $C^1$-foliations can be decomposed into a \emph{smooth} conjugation followed by a homotopy through $C^1$-foliations, for the topology induced by the $C^0$ topology on plane fields. 

\begin{thm}[Bifoliated smoothing, strong version] \label{thm:strongbifolapprox}
    Let $(\mathcal{F}_0,\mathcal{G}_0)$ and $(\mathcal{F}_1,\mathcal{G}_1)$ be orientable $C^1$-bifoliations on $M$, and $h : M \rightarrow M$ be a homeomorphism sending the leaves of $\mathcal{F}_0$ (resp.~$\mathcal{G}_0$) to the leaves of $\mathcal{F}_1$ (resp.~$\mathcal{G}_1$). Then, there exists a topological isotopy $(h_t)_{0\leq t\leq 1}$ such that
    \begin{enumerate}
        \item $h_0=h$,
        \item The map 
        $$\begin{aligned}
        (0,1] \times M &\longrightarrow M \\
        (t,p) &\longmapsto h_t(p)
        \end{aligned}$$
        is smooth,
        \item For $i \in \{0,1\}$, the maps
        \begin{align*}
        t \in [0,1] \longmapsto 
                                    \begin{cases}
                                    (h_t)_* \big(T \mathcal{F}_0\big) & \text{if } t>0,\\
                                    T \mathcal{F}_1 & \text{if } t=0,
                                    \end{cases} & &
        t \in [0,1] \longmapsto 
                                    \begin{cases}
                                    (h_t)_* \big(T \mathcal{G}_0\big) & \text{if } t>0,\\
                                    T \mathcal{G}_1 & \text{if } t=0,
                                    \end{cases}
        \end{align*}
    are continuous.
    \end{enumerate}
\end{thm}

    \section{Uniqueness of contact approximations}

In this section, we refine the main result of Vogel~\cite{V16} and prove Theorem~\ref{thmintrobeta:uniq} from the Introduction. The main task is to maintain transversal control on the plane fields (foliations, confoliations, contact structures) involved in the proof. We also impose minimal regularity assumptions on the foliations under consideration ($C^1$ instead of $C^2$) in order to apply the results to weak foliations of Anosov flows. Along the way, we also provide some more details and fill in some important steps in Vogel's proof for the sake of completeness.

        \subsection{Hypertaut foliations} \label{sec:hypertaut}

Before we describe the class of foliations we are interested in, let us recall some basic definitions: for further background, see Candel--Conlon~\cite{CC00}. Given a foliation, a set is {\bf saturated} if it is a union of leaves. A {\bf minimal set} is a \emph{nonempty, closed} subset that is saturated by leaves and is minimal with respect to inclusion. Such subsets always exist by Zorn's Lemma. A minimal set, on a closed foliated manifold, is {\bf exceptional} if it is neither the whole manifold nor a compact leaf.

A key property of exceptional minimal sets, for $C^2$-foliations, is that they have linear holonomy, meaning that there is a(n embedded) closed leafwise curve $\gamma$, so that the (germinal) map on a transversal given by pushing along leaves has nontrivial derivative (different from $\pm1$) on its first return along $\gamma$. This is the content of Sacksteder's Theorem~\cite{S65}. An argument of Ghys (see~\cite{ET}) improves this to show that minimal foliations of class at least $C^2$ and with holonomy also have linear holonomy. 

It turns out that curves with \emph{attracting} (not necessarily linear) holonomy will be sufficient for our purpose. We will call such a curve an \textbf{attracting curve}. We will say that a foliation $\mathcal{F}$ has \textbf{enough (attracting) holonomy} if every minimal set admits an attracting curve. This notion is closely related to that of hypertautness:

\begin{prop} \label{prop:hypertaut}
    Let $\mathcal{F}$ be a $C^1$-foliation on $M$. Then the following are equivalent:
    \begin{enumerate}
        \item $\mathcal{F}$ is hypertaut,
        \item $\mathcal{F}$ has no nontrivial holonomy invariant transverse measures,
        \item $\mathcal{F}$ has no closed leaves and has enough holonomy.
    \end{enumerate}
\end{prop}

\begin{proof}
    The equivalence between items 1 and 2 is due to Sullivan~\cite{S76}, see also~\cite{CC00}. 
    
    The third item implies the second. Indeed,
    any holonomy invariant transverse measure must vanish near any attracting curve: it must be supported on the leaf containing the attracting curve because of the attracting condition, but this leaf must spiral onto itself since it is noncompact and dense in its minimal set. Moreover, every point can be connected to an arbitrarily small neighborhood of such a curve along a path tangent to the foliation, so the transverse measure must vanish everywhere. 

    We now consider the most difficult implication $(2) \Longrightarrow (3)$. First of all, note that $\mathcal{F}$ cannot have any closed leaves either by item (1), or since a closed leaf would determine an invariant transverse measure (of Dirac type). In particular, the foliation is necessarily taut. Then, we apply the results of~\cite{DKN07} to the holonomy pseudogroup of $\mathcal{F}$ (see~\cite{CC00} for the necessary definitions and background). If $\mathcal{F}$ is minimal, then~\cite[Théorème E]{DKN07} implies that there is an attracting curve (a \emph{``feuille ressort hyperbolique''}) contained in some leaf. Indeed,~\cite[Proposition 4.1]{DKN07} implies that the absence of such an attracting curve would yield the existence of a holonomy invariant probability measure for the foliation. 
    
    If $\mathcal{F}$ is not minimal, then we consider a minimal set $\Lambda$, which is necessarily an exceptional minimal set, and again apply~\cite[Proposition 4.1]{DKN07} to the holonomy pseudogroup along $\Lambda$ to deduce the existence of a curve with attracting holonomy \emph{on at least one side} (a \emph{``feuille ressort''}) in $\Lambda$. Otherwise, there would again be a holonomy invariant transverse probability measure supported on $\Lambda$. Finally, since the foliation is of class $C^1$, we can apply \cite[Section 4.1]{DKN07} to deduce the existence of a curve in $\Lambda$ with attracting holonomy \emph{on both sides}, which concludes the proof.
\end{proof}

        \subsection{Neighborhoods of attracting curves}

Before stating and proving the version of Vogel's uniqueness result that we need, we introduce some more definitions. We will consider special neighborhoods of attracting curves and construct contact approximations which have a specific form in those neighborhoods. This step is implicit and much easier in Vogel's original proof, and we need to make it more precise for the proof in the $C^1$ case. In the rest of this section, we will write $I \coloneqq [-1,1]$.

\medskip

We now describe a normal form of $\mathcal{F}$ near an attracting curve, which may not have linear holonomy; it is certainly well-known to the experts, but we were not able to find a precise proof in the literature.

\begin{lem} \label{lem:normalattracting}
    Let $\mathcal{F}$ be a cooriented $C^1$-foliation with a transverse smooth $1$-dimensional foliation 
    $\mathcal{I}$, and let $\gamma$ be a curve with attracting holonomy. Then there exist a closed neighborhood $N$ of $\gamma$ and $C^1$ coordinates $N \cong S^1_\theta \times I_{y,z}^2$ so that $\gamma \cong S^1 \times \{(0,0)\}$, $\mathcal{I}$ is tangent to $\partial_z$, and $T \mathcal{F}$ is defined by a $1$-form $\alpha$ of the form
    $$\alpha = dz - u(\theta,z) d\theta,$$
    where $u : N \rightarrow \R$ is $C^1$ and satisfies 
    \begin{itemize}
        \item $u(\theta, 0) = 0$,
        \item For $z > 0$, $u(\theta,z) < 0$,
        \item For $z < 0$, $u(\theta,z) > 0$.
    \end{itemize}
\end{lem}

\begin{proof}
    After choosing a transversal to $\mathcal{F}$ intersecting $\gamma$ and which is tangent to $\mathcal{I}$, we may represent the holonomy around $\gamma$ by a $C^1$ map $h : I \rightarrow I$ which satisfies $h' > 0$, $h(0)=0$, $0 < h(z) < z$ for $z > 0$, and $z < h(z) < 0$ for $z < 0$.

    Let $\tau : [0,1] \rightarrow [0,1]$ be a smooth map satisfying $\tau' \geq 0$, $\tau \equiv 0$ on $[0, 0.1]$, $\tau \equiv 1$ on $[0.9,1]$, and $\tau'> 0$ on $(0.1, 0.9)$. We define a (partial) foliation on $[0,1] \times I$ by considering the graphs of the maps
    $$f_t(z) \coloneqq \tau(t) h(z) + (1-\tau(t)) z$$
    for $t \in [0,1]$ and $z \in I$. Note that $(t,z) \mapsto f_t(z)$ is $C^1$ and satisfies $\partial_z f_t > 0$ and $\partial_t f_t(z) = \tau'(t) (h(z) - z)$, hence it defines a $C^1$-foliation tangent to a $C^1$ vector field $V$ of the form $V = \partial_t + v(t,z) \partial_z$ where
    \begin{itemize}
        \item $v(t,0) = 0$ for $t\in [0,1]$,
        \item $v(t,z) = 0$ for $t \in [0,0.1] \cup [0.9,1]$ and $z \in I$,
        \item For $t \in (0.1, 0.9)$, if $z> 0$ then $v(t,z) < 0$, and if $z < 0$, then $v(t,z) > 0$.
    \end{itemize}
    After shrinking the transversal, we may assume that $V$ is defined on the whole of $[0,1] \times I$ and induces a vector field $\widehat{V}$ on $S^1 \times I = [0,1]_{\slash 0 \sim 1} \times I$. Notice that its flow induces the same (germinal) holonomy as $\mathcal{F}$ around $\gamma$.
    
    At this point, $\widehat{V}$ is not quite of the desired form since it is tangent to $\partial_t$ on a strip. However, we may perform a suitable change of coordinates to achieve the desired form. For that, it suffices to find a smooth foliation by circles so that $S^1 \times \{0\} \subset S^1 \times I$ is a leaf, and all the other leaves are transverse to $\widehat{V}$, positively on $\{z > 0\}$ and negatively on $\{z < 0\}$. We may consider a smooth vector field $V_0 = \partial_t + v_0(t,z) \partial_z$ on $[0.2, 0.8] \times I$ where $v_0$ satisfies 
    \begin{itemize}
        \item $v_0(t,0) = 0$ for $t \in [0.2, 0.8]$,
        \item $v_0(t,z)= 0$ for $t \in [0.2,0.3] \cup [0.7,0.8]$ and $z \in I$,
        \item For all $t \in (0.3,0.7)$, if $z > 0$ then $v(t,z) < v_0(t,z) < 0$ and if $z < 0$, $0 < v_0(t,z) < v(t,z)$.\footnote{Here, we may use the following well-known lemma: if $f : [0,1] \rightarrow [0,1]$ is a continuous function satisfying $f(0)=0$ and $f(t) > 0$ for $t \in (0,1]$, then there exists a smooth function $g : [0,1] \rightarrow [0,1]$ such that $g(0)=0$, $0 < g(t) < f(t)$ for $t \in (0,1]$, and all the derivatives of $g$ at $0$ vanish.}
    \end{itemize}
    Then $V_0$ is positively transverse to $V$ for $z > 0$ and negatively transverse to $V$ for $z < 0$, and its flow from $t = 0.2$ to $t = 0.8$ induces a map $h_0 : I \rightarrow I$ satisfying $h_0(0) = 0$, $0 <h_0(z) < z$ for $z > 0$, and $z < h_0(z) < 0$ for $z < 0$. Then, the vector field $V_0$ can be extended to a smooth vector field $\widehat{V}_0$ on $S^1 \times I$ which is tangent to $S^1 \times \{0\}$, transverse to $\widehat{V}$ everywhere with the same signs as before, and whose flow lines are all closed: it suffices to realize the map $h^{-1}_0$ by the flow of a ``repelling'' vector field on the small arc of $S^1$ between $0.8$ and $0.2$. Then, the flow of $\widehat{V}_0$ induces a smooth foliation by circles, which are transverse to $\partial_z$, and after a smooth graphical change of coordinates making its leaves tangent to $\partial_t$, the vector field induced by $\widehat{V}$ becomes tangent to a foliation of the desired form. Note that this change of coordinates preserves the vertical direction spanned by $\partial_z$.

    The product of the latter foliation with the $y$-direction induces the same holonomy as $\mathcal{F}$ along $\gamma$, so there exist $C^1$ coordinates in which $\mathcal{F}$ is defined by a $1$-form $\alpha$ as in the statement. Moreover, since these foliations are transverse to $\partial_z$ and to $\mathcal{I}$, respectively, we may further arrange that this change of coordinates sends $\langle \partial_z \rangle$ to $T\mathcal{I}$.
\end{proof}

\begin{rem} \label{rem:smoothcoord}
    Note that even though $T \mathcal{F}$ is not assumed to be $C^1$, the induced foliation in the $C^1$ coordinates provided by the lemma has a $C^1$ tangent plane field. Moreover, we may assume that these coordinates are smooth, at the expense of an arbitrarily small $C^1$ isotopy of $\mathcal{F}$ transverse to $\mathcal{I}$ near $N$. We may further arrange that the map $u$ is \textbf{smooth} away from an arbitrarily small neighborhood of $S^1 \times \{0\} \subset S^1 \times I$, at the expense of another arbitrarily small $C^1$ isotopy, using the methods from~\cite[Section 3]{B16}.
\end{rem}

Let $U$ be a neighborhood of an attracting curve $\gamma$ for $\mathcal{F}$. A contact structure $\xi$ approximating $\mathcal{F}$ and transverse to a smooth $1$-dimensional foliation $\mathcal{I}$ is $\mathcal{I}$-\textbf{standard} on $U$ if there exists a larger neighborhood $U \subset \widehat{U} \cong S^1 \times I \times I$ with coordinates $(\theta, y, z)$ in which $\partial_z$ is positively tangent to $\mathcal{I}$, and $\xi = \ker \alpha_0$ for a $1$-form $\alpha_0$ of the form
$$\alpha_0 = dz - u(\theta, y, z) d\theta,$$
where the function $u : \widehat{U} \rightarrow \R$ satisfies 
\begin{itemize}
    \item $\partial_y u > 0$ (this is the contact condition),
    \item $\mp u(\theta, y, \pm1) > 0$ (this guarantees that $\xi$ is transverse to the top/bottom faces of $\widehat{U}$),
    \item Near $y=1$, $u$ is of the form
    $$u(\theta, y,z) = u_1(z+ 1-y),$$ where $u_1 : \R \rightarrow \R$ is a smooth function satisfying $u_1(0)=0$ and $u'_1 < 0$.
\end{itemize}

This last condition might seem somewhat obscure at the moment but will be important in the proof of Proposition~\ref{prop:contractible}. In particular, it ensures that the characteristic foliation of $\xi$ along the face $y=1$ has a single closed orbit which has attracting linear holonomy. 

We call $U$ an \textbf{$\mathcal{I}$-standard neighborhood} of $\gamma$. Finally, we say that a finite collection of attracting curves $\{\gamma_1, \dots, \gamma_k\}$ for $\mathcal{F}$ is \textbf{full} if every minimal set of $\mathcal{F}$ contains a $\gamma_i$, $1 \leq i \leq k$.

The following proposition provides convenient `basepoints' for our version of Vogel's theorem. In the case where all the minimal sets have \emph{linear} holonomy and $T\mathcal{F}$ is $C^1$, for example when the foliation is Anosov, the result would essentially follow from~\cite{ET}. However, the lack of linear holonomy and regularity of $T\mathcal{F}$ causes some complications and we resort to the tools in~\cite{B16} for the more general case.

\begin{prop} \label{prop:stdnbdgamma}
    Let $\mathcal{F}$ be a hypertaut $C^1$-foliation on $M$ and $\mathcal{I}$ be a smooth $1$-dimensional foliation transverse to $\mathcal{F}$. For every full collection of attracting curves $\{\gamma_1, \dots, \gamma_k\}$, there exist neighborhoods $U_i$ of $\gamma_i$, $1 \leq i \leq k$, such that $\mathcal{F}$ is $C^0$-approximated by positive contact structures which are $\mathcal{I}$-standard on each $U_i$, $1 \leq i \leq k$, and the corresponding neighborhoods $\widehat{U}_i$ are pairwise disjoint.
\end{prop}

A similar result holds for negative contact structures approximating $\mathcal{F}$. Notice that the neighborhoods $U_i$ are \emph{fixed} and independent of the contact approximations, and that we do not require the attracting curves to have \emph{linear} holonomy. We begin with a technical lemma that will allow us to produce the desired characteristic foliation for the contact approximations along the face $\{y=1\}$; this is achieved via a careful `squashing' procedure.

\begin{lem}[Squashing] \label{lem:betterfence}
    Let $\mathcal{G}$ be a foliation on $A = S^1_\theta \times I_z$ defined by a $1$-form $\alpha$ as in Lemma~\ref{lem:normalattracting}. 
    For every $\epsilon > 0$, there exists a smooth foliation $\widetilde{\mathcal{G}}$ on $A$ such that the following hold:
    \begin{itemize}
        \item $T \widetilde{\mathcal{G}}$ is $\epsilon$-close to $T \mathcal{G}$ in the $C^0$ topology,
        \item $T \widetilde{\mathcal{G}}$ is negatively transverse to $T \mathcal{G}$,
        \item Writing $A' = S^1 \times [-1/2, 1/2] \subset A$, there exists an orientation-preserving diffeomorphism $\varphi : A' \rightarrow A'$ such that $\mathcal{H} \coloneqq \varphi_*\big(\widetilde{\mathcal{G}}\big)$ is defined by a $1$-form $\beta$ of the form
        $$ \beta = dz-v(z) d\theta,$$
        where $v : [-1/2, 1/2] \rightarrow \R$ is a smooth function satisfying  $v(0) = 0$ and $v' < 0$. In particular, $\mathcal{H}$ has linear holonomy along $S^1 \times \{0\}$. Moreover, $\varphi$ is graphical and preserves the line field spanned by $\partial_z$.
    \end{itemize}
\end{lem}

\begin{proof}
    We use the notation of Lemma~\ref{lem:normalattracting} for $\mathcal{G}$. Let $\epsilon > 0$. We consider $\delta > 0$ small enough such that $\vert u \vert < \epsilon$ on $A_{4\delta} = S^1 \times [-4\delta, 4\delta]$, and $\delta < \epsilon$. Moreover, we write $\mu \coloneqq \inf_{S^1 \times [-4\delta, -\delta]} u > 0$.
    
    The construction of $\widetilde{\mathcal{G}}$ proceeds in three steps.

    \begin{itemize}[leftmargin=*]
    \item \textit{Step 1: on $S^1 \times \big([-1, -4\delta] \cup [0,1]\big)$.} Notice that $\mathcal{G}$ is a product on $S^1 \times [-1, -4\delta]$ in suitable coordinates, so we can easily construct a smoothing $\widetilde{\mathcal{G}}$ which is $\epsilon$-close in the $C^1$ topology, and negatively transverse to it. On $S^1 \times [0,1]$, we may first define $\widetilde{\mathcal{G}}$ near $S^1 \times \{0\}$ and extend it on the rest of $S^1 \times [0,1]$ on which $\mathcal{G}$ is also a product.
    
    \item \textit{Step 2: extension to $S^1 \times \big([-1, -3\delta] \cup [-\delta,1]\big)$.} The foliation $\widetilde{\mathcal{G}}$ induces a smooth map from $S^1 \times \{1/2\}$ to $S^1 \times \{0\}$ obtained by following the leaves from the first circle to the second. Lifting to the universal covers gives rise to a smooth map $\tau : \mathbb{R} \rightarrow \mathbb{R}$ which satisfies $\tau(t+1) = \tau(t) + 1$ for all $t \in \mathbb{R}$, and which has a positive translation number. For any sufficiently large $T \gg 1$, we can find another such map $\widehat{\tau} : \mathbb{R} \rightarrow \mathbb{R}$ which has translation number at least $T$, and so that the composition $\widehat{\tau} \circ \tau$ is a genuine translation with translation number $T' > T$. We may then extend $\widetilde{\mathcal{G}}$ to $S^1 \times [-\delta, 1]$ so that it is negatively transverse to $\mathcal{G}$ (note that $T\mathcal{G}$ has nonnegative slope on $S^1 \times [-\delta, 0]$), its leaves are graphs of decreasing functions, and its induced map from $S^1 \times \{1/2\}$ to $S^1 \times \{-\delta\}$ lifts to $\widehat{\tau} \circ \tau$. We may also arrange that $\widetilde{\mathcal{G}}$ is $\theta$-invariant near $S^1 \times \{-\delta\}$. By construction, there is a graphical diffeomorphism $\varphi : S^1 \times [-\delta, 1/2] \rightarrow S^1 \times [-\delta, 1/2]$ such that $\varphi_* \big(\widetilde{\mathcal{G}}\big)$ is $\theta$-invariant, and is defined by a $1$-form of the form $\beta = dz - v(z) d\theta$, where $v< 0$. Moreover, we may easily arrange that $\varphi$ is supported away from $S^1 \times \{-\delta\}$.

    We may perform a similar construction to extend $\widetilde{\mathcal{G}}$ on $S^1 \times [-1, -3 \delta]$, with the additional constraint that $T\widetilde{\mathcal{G}}$ has a positive slope less than $\mu$ on $S^1 \times [-4\delta, -3\delta]$ to ensure transversality to $T \mathcal{G}$. We similarly obtain a graphical diffeomorphism $S^1 \times [-1, -3\delta] \rightarrow S^1 \times [-1, -3\delta]$ making $\widetilde{\mathcal{G}}$ $\theta$-invariant, and which is supported away from $S^1 \times \{-3\delta\}$.

    This step can easily be performed so that $T\widetilde{\mathcal{G}}$ is $\epsilon$-close to $T \mathcal{G}$, thanks to our choice of $\delta$. It suffices to choose the parameter $T$ large enough, in order to control the slope of $\widetilde{\mathcal{G}}$ during the extension. 
    
    \item \textit{Step 3: extension to $A$.} Near $S^1 \times \{-2\delta\}$, we define $\widetilde{\mathcal{G}}$ through the $1$-form
    $$dz + c (z+2\delta) d\theta$$
    for a very small constant $c > 0$, so that it has attracting linear holonomy along $S^1 \times \{-2\delta\}$. We may then connect it to the rest of $\widetilde{\mathcal{G}}$ by a $\theta$-invariant foliation transverse to $T \mathcal{G}$. Note that in the last extension region, $T \mathcal{G}$ has a positive slope and is at distance at least $\mu$ from $\langle \partial_\theta\rangle$. The resulting foliation can easily be made $\epsilon$-close to $T\mathcal{G}$, by making the extension sufficiently flat.
    \end{itemize}
By construction, the diffeomorphisms from Step 2 can be combined into a diffeomorphism $\varphi : S^1 \times [-1/2, 1/2] \rightarrow S^1 \times [-1/2, 1/2]$ supported away from $S^1 \times [-3\delta, -\delta]$, so that $\mathcal{H} \coloneqq \varphi_* \big(\widetilde{\mathcal{G}}\big)$ is $\theta$-invariant, and is defined by a $1$-form $\beta$ of the form $\beta = dz - v(z) d\theta$, where $v$ satisfies:
        \begin{itemize}
            \item $v(-2\delta)  =0$, $v(z) < 0$ for $z > -2\delta$ and $v(z) > 0$ for $z< -2\delta$, 
            \item $v'<0$ near $z=-2\delta$.
        \end{itemize}
    Essentially, $\mathcal{H}$ consists of a family of $\theta$-translates of a decreasing function above $z=-2\delta$, $\theta$-translates of an increasing function below $z=-2\delta$, together with the closed curve $S^1 \times \{-2\delta\}$. Moreover, these functions are \emph{strictly convex} near $z = -2 \delta$. It is then easy to modify $\varphi$ by a graphical diffeomorphism so that these functions become convex everywhere. This ensures that the resulting map $v$ satisfies $v' < 0$. We may also arrange that $v$ vanishes along $\{z=0\}$ in these new coordinates.
\end{proof}

The previous lemma allows us to find a foliation on $S_\theta^1 \times I_z$ which behaves like the restriction of an $\mathcal{I}$-standard contact structure along the face $\{y = 1\}$ of $N$. We can therefore `insert it' on a parallel slice close to that one and replace $\mathcal{F}$ with a confoliation in between those two slices by twisting along $\partial_y$. We also need to modify $\mathcal{F}$ along a slice closer to the face $\{y=-1\}$ in order to obtain a confoliation on the whole of $N$; this modification can be performed by a suitable isotopy of $\mathcal{F}$ making it negatively transverse to itself, at least away from $\gamma$. The final step is to fill the neighborhood of $\gamma$ with an appropriate contact structure. We now make this strategy more precise.

\begin{proof}[Proof of Proposition~\ref{prop:stdnbdgamma}]
We proceed in two steps. First, we approximate $\mathcal{F}$ near the full collection of attracting curves to obtain a $C^1$-confoliation there. Then, we explain how to ``propagate the contactness'' everywhere. Note that even though $\mathcal{F}$ is $C^1$, the plane field $T \mathcal{F}$ is only $C^0$ yet \emph{uniquely} integrable, and we may further assume that it is $C^1$ near these attracting curves (after an arbitrarily small $C^1$ isotopy of $\mathcal{F}$) by Lemma~\ref{lem:normalattracting} and Remark~\ref{rem:smoothcoord}. 

\underline{Let us fix $\epsilon > 0$.}
\begin{itemize}[leftmargin=*]
    \item \textit{Step 1: creating contactness near attracting curves.} For simplicity, we consider a single attracting curve $\gamma$ with a neighborhood $N \cong S_\theta^1 \times I_{y,z}^2 = A_{\theta,z} \times I_y$ as in Lemma~\ref{lem:normalattracting}. We choose $\delta > 0$ small enough so that $\vert u \vert < \epsilon$ on $A_\delta = S^1 \times [-\delta, \delta]$. We may shrink $\delta$ further so that $\delta < \epsilon$. We then modify $\mathcal{F}$ so that $u$ is \emph{smooth} on a neighborhood of $S^1 \times \big(I \setminus (-\delta, \delta) \big)$, at the expense of a $C^1$-small isotopy which is much smaller than $\epsilon$. This can be achieved while keeping $\mathcal{F}$ tangent to $\partial_y$.
    
    We denote by $\mathcal{G}$ the foliation on $A$ induced by $\mathcal{F}$. Since $\mathcal{G}$ is smooth and trivial on $S^1 \times [\delta, 1]$ and $S^1 \times [-1, -\delta]$, it is easy to construct a smooth isotopy $(\phi_t)_{0 \leq t \leq 1}$ of $A$ with the following properties:
    \begin{itemize}
        \item[(a)] $\phi_t$ is supported on $S^1 \times \big([-0.9, -\delta/2] \cup [\delta/2, 0.9]\big)$,
        \item[(b)] $(\phi_t)_t$ is $\epsilon$-small in the $C^1$ topology,
        \item[(c)] $T \mathcal{G}_t$ is pointwise equal to or negatively transverse to $T \mathcal{G}$ on $S^1 \times \big(I \setminus (-\delta, \delta)\big)$, and is transverse to it on $S^1 \times \big( [-0.8, -\delta]\cup [\delta, 0.8]\big)$.
    \end{itemize}
    We now fix $\mu > 0$ such that $T \mathcal{G}_1$ is at distance at least $\mu$ from $T\mathcal{G}$ on $S^1 \times \big( [-0.8, -\delta]\cup [\delta, 0.8]\big)$, and we consider a smooth foliation $\widetilde{\mathcal{G}}$ as in Lemma~\ref{lem:betterfence}, so that $T\widetilde{\mathcal{G}}$ is at distance less than $\mu$ from $T \mathcal{G}$. We then modify $\widetilde{\mathcal{G}}$ on $S^1 \times \big([-1, -0.8] \cup [0.8, 1]\big)$ so that the resulting foliation $\widehat{\mathcal{G}}$ coincides with $\mathcal{G}$ on $S^1 \times \big([-1, -0.9] \cup [0.9, 1]\big)$, and $T\widehat{\mathcal{G}}$ lies between $T \mathcal{G}_1$ and $T\mathcal{G}$. Note that if $\varphi$ denotes the diffeomorphism from the third bullet point of Lemma~\ref{lem:betterfence}, then $\mathcal{H} \coloneqq \varphi_*(\widehat{\mathcal{G}})$ is of the desired form on the subannulus $A' = S^1 \times [-1/2, 1/2] \subset A$. 

    We now construct a $C^1$-confoliation $\xi$ on $N$ as follows:
    \begin{itemize}
        \item On $A \times [-1, -1/2]$, $\xi$ is obtained as the trace of the isotopy $\mathcal{G}_t$, with $t$ appropriately rescaled and cut off near $y = -1$ and $y=-1/2$. The resulting plane field is at distance at most $C \epsilon$ from $T \mathcal{F}$, for some universal constant $C$. Moreover, $\xi$ is tangent to $\partial_y$ near $y=-1/2$ and its restriction to $A \times \{-1/2\}$ coincides with $\mathcal{G}_1$.
        
        \item Near $A \times \{1/2\}$, $\xi$ is tangent to $\partial_y$ and its restriction to $A \times\{1/2\}$ coincides with $\widehat{\mathcal{G}}$. Moreover, we may arrange that after applying $\varphi \times \mathrm{id}$ near $A' \times \{1/2\}$, the resulting confoliation is defined by a $1$-form of the form $dz-u_1(z+1/2-y)d\theta$ where $u_1$ satisfies $u_1(0)=0$ and $u_1' <0$. In particular, $\xi$ is contact in a neighborhood of $A' \times \{1/2\}$. Furthermore, we can arrange that for a very small $\sigma > 0$, the restriction of $\xi$ to $A \times \{1/2+\sigma\}$ is pointwise negatively transverse to or equal to $T \mathcal{G}$, and its restriction to $S^1 \times \big(I \setminus (-\delta, \delta) \big)\times \{1/2 - \sigma\}$ is pointwise positively transverse to or equal to $T\mathcal{G}_1$.

        \item We can now extend $\xi$ to $A \times [1/2, 1]$ by an appropriate twisting along $\partial_y$, so that $\xi$ is tangent to $\mathcal{F}$ near $\partial A \times [1/2, 1]$. We can also extend it by twisting on $S^1 \times \big([-1, -\delta] \cup [\delta, 1] \big) \times [-1/2, 1/2]$. We shall also modify $\xi$ near $S^1 \times \{-1/2\} \times [-\delta, \delta]$ to add some twisting along $\partial_y$. The resulting $C^1$ plane field is then contact and tangent to $\partial_y$ near $S^1 \times [-1/2, 1/2]^2 \setminus \left(S^1 \times (-1/2, 1/2)_y \times (-\delta, \delta)_z \right)$. 

        \item Finally, we extend $\xi$ to $N$ by filling $S^1 \times [-1/2, 1/2]_y \times [-\delta, \delta]_z$ using~\cite[Lemma 5.6]{B16}. The resulting confoliation $\xi$ is at a distance at most $C \epsilon$ from $T \mathcal{F}$, for some universal constant $C$ independent of $\epsilon$. Note that this extension is not necessarily tangent to $\partial_y$.
    \end{itemize}

    This construction is summarized in Figure~\ref{fig:confoliation}.
    After some very small $C^1$ perturbation, we can arrange that $\xi$ is smooth on a neighborhood $V$ of $S^1 \times [-1/2, 1/2]^2$. Moreover, there exists a smooth vector field $Y$ on $V$ spanning the intersection of $\xi$ with $\mathrm{span}\{ \partial_y, \partial_z\}$ and which is $C\epsilon$-close to $\partial_y$. For $\epsilon > 0$ small enough, we can use its flow to define new coordinates $(y', z')$ near $S^1 \times [-1/2, 1/2]^2$ such that $\partial_{z'}=\partial_z$ and $\partial_{y'}$ is tangent to $\xi$, and after applying the diffeomorphism $\varphi \times \mathrm{id}$ of $A_{\theta, z} \times [-1/2, 1/2]_y$, we obtain a contact structure on $S^1 \times [-1/2, 1/2]^2$ of the desired form.

    \begin{figure}[h]
        \centering
        \includestandalone{tikz/square}
        \caption[Short caption]{The construction of $\xi$. In blue: $\xi$ is tangent to (a small $C^1$-isotopy of) $\mathcal{F}$. In green: adjusted trace of the isotopy $\mathcal{G}_t$. In orange and red: insertion of $\widehat{\mathcal{G}}$ with twisting along $\partial_y$. The part in red is made ``standard'' after a change of coordinates induced by $\varphi$. In yellow: $\xi$ is tangent to $\partial_y$ and twists along $\partial_y$ to interpolate between the different regions. In purple: extension obtained by applying~\cite[Lemma 5.6]{B16}. It is not necessarily tangent to $\partial_y$.}
        \label{fig:confoliation}
    \end{figure}
    
    \item \textit{Step 2: propagating the contactness.} By the previous step, we obtain a plane field $\xi$ on $M$ with the following properties:
    \begin{itemize}
        \item $\xi$ is $C^0$, and it is $C^1$ in neighborhoods of the attracting curves $\gamma_i$, $1 \leq i \leq k$,
        \item $\xi$ is a $C^1$-confoliation in those neighborhoods, and contact and ``standard'' in smaller neighborhoods of those curves (these neighborhoods have a fixed size),
        \item There is a $C^1$-foliation $\widetilde{\mathcal{F}}$, obtained from $\mathcal{F}$ by a $C^1$-small isotopy, such that $\xi$ is tangent to $\widetilde{\mathcal{F}}$ at the points where it is not contact (in particular at the points where it is not $C^1$).
    \end{itemize}
    Therefore, $\xi$ is a \emph{$C^0$-confoliation} in the sense of~\cite[Definition 2.17]{B16}. It is moreover \textit{transitive}: every point in $M$ can be joined to the contact region by a curve tangent to $\xi$. We may then use the strategy from~\cite[Section 8]{B16} to approximate $\xi$ by a contact structure, relative to a neighborhood of the attracting curves where $\xi$ is of the desired form. This strategy simplifies in our setting, since $\xi$ is already contact near the attracting curves, and it is moreover \emph{uniquely} integrable in (the interior of) its foliated region. Therefore, it is much easier to control how the holonomy is modified under perturbation. \qedhere
\end{itemize}
\end{proof}

        \subsection{Uniqueness with transversal constraint}

Recall that $\mathcal{F}$ is a cooriented hypertaut $C^1$-foliation. We fix a smooth $1$-dimensional foliation $\mathcal{I}$ positively transverse to $\mathcal{F}$. The space of continuous plane fields transverse to $\mathcal{I}$ is denoted by $\mathcal{P}_\mathcal{I}$ and is endowed with a metric induced by the choice of an auxiliary Riemannian metric on $M$. The main result of this section is

\begin{thm}[Theorem~\ref{thmintrobeta:uniq}] \label{thm:uniq}
    There exists a $C^0$-neighborhood $\mathcal{V} \subset \mathcal{P}_\mathcal{I}$ of $T\mathcal{F}$ such that any two positive (resp.~negative) contact structures in $\mathcal{V}$ are contact homotopic \underline{within $\mathcal{P}_\mathcal{I}$}.
\end{thm}

We will call such a neighborhood $\mathcal{V}$ a \textbf{Vogel neighborhood of $\mathcal{F}$ adapted to $\mathcal{I}$}. The proof we give follows~\cite{V16} closely with some additional (crucial!) details. Keeping contact structures transverse to $\mathcal{I}$ requires a bit more care at various steps of the proof. For the reader's convenience, we outline the proof of the theorem, following the steps of~\cite{V16}, and we explain how to fill some gaps and fix some inaccuracies. We will assume that the reader is already familiar with Vogel's article. Additional details are covered in the next sections.

\begin{rem}
    We do not claim a version of Theorem~\ref{thm:uniq} with more parameters. Indeed, we will make some generic modifications to an approximating contact structure in order to ensure that certain annuli are (essentially) convex, and it is not clear how to achieve this for a \emph{family} of contact structures. However, it seems plausible that the strategy can be adapted to that case at the expense of further technicalities. 
\end{rem}

\paragraph{Adapted polyhedral decompositions.} We quickly review~\cite[Definition 4.12]{V16}.
We consider a triangulation on $M$ which is in general position, after an appropriate jiggling, with respect to a plane field $\xi$, which will for the most part be given by the tangent distribution $T\mathcal{F}$ or a perturbation thereof. Following Colin~\cite{C99}, one can then alter this triangulation by modifying the tetrahedra near supporting vertices, i.e., at vertices where the tangent plane to the foliation meets the tetrahedron at a single point. This way, one obtains a polyhedral decomposition so that each vertex is supporting for at most one polyhedron. We can further assume that exactly {\bf three} edges of the polyhedron meet at a supporting vertex.

Let $P$ be such a polyhedron, and denote by $\partial P(\xi)$ the induced characteristic foliation. Our assumptions will always be such that this foliation is piecewise $C^1$, which is a slight generalization of Vogel's set-up, where all foliations are $C^2$. By general position, this characteristic foliation on $\partial P$ has a global transversal $\gamma(P) \subset P^{(1)}$, which can be assumed to lie in the $1$-skeleton.

We further assume that the resulting decomposition is in general position with the given transverse foliation $\mathcal{I}$ and is, in addition, \emph{graphical}. This means that the projections of the smooth segments on the boundary intersect transversely, and the projection of each (piecewise $C^1$-smooth) first return curve of the characteristic foliation meets itself (transversely) at at most one point. We refer to~\cite{V16} for more details.

We shall call such polyhedral decompositions {\bf adapted} to $\xi$ and $\mathcal{I}$. 

\paragraph{Polyhedral decomposition and neighborhoods of attracting curves.}

We choose a full collection of attracting curves $\{ \gamma_1, \dots, \gamma_k\}$ and $\mathcal{I}$-standard neighborhoods $\gamma_i \subset U_i$, $1 \leq i \leq k$, as in Proposition~\ref{prop:stdnbdgamma}, and we consider a polyhedral decomposition $\mathbb{P}$ adapted to $\mathcal{F}$ and $\mathcal{I}$ as in~\cite[Section 4A2]{V16}. We further choose product neighborhoods $N_i \subset \widehat{N}_i \subset U_i$ as in~\cite[Section 5B]{V16} whose boundaries are in general position with the polyhedral decomposition. We can choose $\mathbb{P}$ fine enough so that there is a layer of polyhedra in $\widehat{N}_i$ separating $\partial \widehat{N}_i$ from $N_i$. We then write
$$N \coloneqq \bigcup_{1 \leq i \leq k} N_i, \qquad \widehat{N} \coloneqq \bigcup_{1 \leq i \leq k} \widehat{N}_i, \qquad U \coloneqq \bigcup_{1 \leq i \leq k} U_i.$$

\paragraph{Ribbons.} 

Let $\xi$ be a coorientable plane field transverse to $\mathcal{I}$. We consider a $C^1$ embedding of a closed strip $R=[0,1]_y\times [-\delta,\delta]_z$, $\delta > 0$, so that the intervals $\{y_0\}\times [-\delta,\delta]$ are tangent to $\mathcal{I}$, hence transverse to $\xi$, and such that the intervals $[0,1]\times \{z_0\}$ are tangent to $\xi$. By slight abuse of terminology, the intervals $[0,1]\times \{z_0\}$ are referred to as \emph{Legendrians}. Such subsets are called $\mathcal{I}$-adapted {\bf ribbons}, or simply ribbons. Ribbons are easily constructed in our setting and the resulting characteristic foliation induced by $\xi$ is nonsingular and transverse to the $z$-intervals. This line field is $C^0$ and uniquely integrable, as it integrates to a $1$-dimensional foliation of class $C^1$. By the unique integrability, one can control changes to the characteristic foliation under $C^0$-perturbations, whereas for general $C^0$-plane fields this is not the case.

We shall consider ribbons that start at faces of $\mathbb{P}$, in the sense that the interval $\{0\} \times [-\delta,\delta] \subseteq \mathbb{P}^{(2)}$ is contained in a unique face of the decomposition and intersects the $1$-skeleton only in at most one supporting vertex. We further require that the ribbons enter $N$ through its top or bottom faces in the product coordinates near the attracting curves. This is possible in our setting since the minimal sets are not closed leaves, so the leaf through a chosen attracting curve $\gamma_i$ also intersects the top or bottom of its corresponding product neighborhood $N_i$. 

\begin{figure}[h]
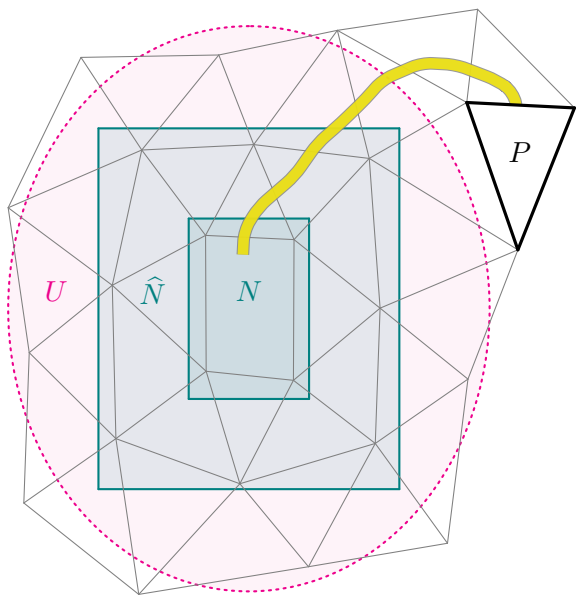

    \centering
    \includestandalone{tikz/vogel}
    \caption{Blueprint of the setup of Vogel's proof.}
    \label{fig:vogel}
\end{figure}

\paragraph{Full collections of ribbons.}

We now review~\cite[Definition 5.4]{V16}. Given a plane field in general position with an adapted polyhedral decomposition, a collection of ribbons is {\bf full} if every ribbon starts in a face that is disjoint from $N$, ends in $N$, and remains disjoint from $N$ otherwise (it is not allowed to re-enter $N$). Moreover, we assume that every leaf of the characteristic foliation of $P \subset M \setminus \text{int}(N)$, including near supporting vertices, intersects the interior of the starting interval of a ribbon. We also assume that every supporting vertex is contained in some ribbon.

There is an additional subtlety in Vogel's proof concerning the choice of ribbons: because a ribbon starting at a given polyhedron crosses other polyhedra, one has to consider \emph{induced ribbons} as in~\cite[Section 4B]{V16}. We refer the reader to that section for additional details.

\paragraph{Vogel neighborhood.}

Now, given a fixed collection of data (neighborhoods of attracting curves, polyhedral decomposition, ribbons), Vogel~\cite[Section 5C]{V16} describes a neighborhood of $\xi=T\mathcal{F}$ for a $C^2$-foliation with holonomy and without closed leaves. The definition of this neighborhood readily extends to the case of hypertaut foliations of class $C^1$, since by unique integrability, small $C^0$ changes in the plane field induce small changes of holonomy. We denote this neighborhood by $\mathcal{V} = \mathcal{V}_\mathcal{I}$.

\bigskip

We now consider an arbitrary contact structure $\xi \in \mathcal{V}$, as well as a contact structure $\widetilde{\xi} \in \mathcal{V}$ which is $\mathcal{I}$-standard in each of the $U_i$'s as in Proposition~\ref{prop:stdnbdgamma}. We think of the latter as a basepoint in $\mathcal{V}$.

\paragraph{Homotopy near supporting vertices.}

One first deforms the contact structure $\xi$ near the supporting vertices by considering small disks transverse to $\xi$ and tangent to $\mathcal{I}$, deforming the resulting characteristic foliations of $\xi$ into those of $\widetilde{\xi}$, and finally applying Gray stability exactly as in Vogel~\cite{V16}. We then get a $1$-parameter family $\xi^0_t$ defined near the supporting vertices from $\xi^0_0 = \xi$ to $\xi^0_1 = \widetilde{\xi}$. This family is $C^0$-close to $T\mathcal{F}$ (see~\cite[Section 4C2]{V16}).

\paragraph{Extending the homotopy near the $2$-skeleton.}

We first consider an extension $\eta_t$ of $\xi^0_t$, $0 \leq t \leq 1$, where $\eta_t$ is a smooth family of plane fields in $\mathcal{V}$ such that
\begin{itemize}
    \item $\eta_0 = \xi$, 
    \item On all the polyhedra intersecting the complement of $\widehat{N}$, $\eta_1 = \widetilde{\xi}$,
    \item On all the polyhedra intersecting $N$, $\eta_t =  \xi$ for all $t \in [0,1]$.
\end{itemize}
Recall that there is a layer of polyhedra between $N$ and $\widehat{N}$.

\paragraph{Correcting holonomy with semi-infinite ribbons.}

We then modify the holonomy of $\eta_t$ by attaching semi-infinite ribbons to all the polyhedra disjoint from $N$. We can first arrange that the ribbons land in convex annuli, by modifying $\xi$ generically. There is a path of vector fields $X_t$, $t \in [0,1]$, supported on the union of the semi-infinite ribbons and tangent to $\eta_t$, such that flowing along $X_t$ for a long time $T > 0$ (independent of $t$) makes the holonomy of $\eta_t$ on the polyhedra that are disjoint from $N$ \emph{negative}---the polyhedra intersecting $N$ are also modified since ribbons might cross them, but the plane fields there are all \emph{contact}.

We write $\widehat{\eta}_t \coloneqq (\phi_{X_t}^T)^*\eta_t$ and note that for each $s \in \R$, the distribution $(\phi_{X_t}^s)^*\eta_t$ is transverse to $\mathcal{I}$. In particular, $\widehat{\eta}_0$ and $\xi$ are homotopic through contact structures transverse to $\mathcal{I}$. We may further arrange that $\widehat{\eta}_t$ is contact near the $2$-skeleton; see~\cite[Remark 4.15]{V16}. Note that for all $t \in [0,1]$, $\widehat{\eta}_t$ is contact on the polyhedra intersecting $N$.

The details of how the twisting along ribbons occurs are presented in Section~\ref{sec:ribbontwist} below.

\paragraph{Extending the homotopy to the interiors of polyhedra.}

We modify $\widehat{\eta}_t$ in the interior of the polyhedra disjoint from $N$ to obtain a family of contact structures $\widehat{\xi}_t$, $0 \leq t \leq 1$, such that 
\begin{itemize}
    \item For $t \in \{0, 1\}$, $\widehat{\xi}_t = \widehat{\eta}_t$,
    \item For all $t \in [0,1]$, $\widehat{\xi}_t$ is transverse to $\mathcal{I}$.
\end{itemize}

\noindent At this point, we need to fill in some missing details in Vogel's proof. It is crucial that $\widehat{\eta}_t$ is \emph{graphical} on each polyhedron, which is ensured by the choice of Vogel neighborhood $\mathcal{V}$.

The details of that step are presented in Section~\ref{sec:fillpoly} below.

\paragraph{Correcting the homotopy near attracting curves.}

By definition, $\widehat{\xi}_1$ is homotopic through contact structures transverse to $\mathcal{I}$ to a contact structure $\xi_1 = \eta_1$ which coincides with $\widetilde{\xi}$ on the polyhedra intersecting the complement of $\widehat{N}$ (by inverting the `ribbon flow', we can assume that $T$ is large enough to ensure that the region where $\widetilde{\xi}$ and $\eta_1$ coincide contains the aforementioned polyhedra). We modify $\xi_1$ on the standard neighborhoods of attracting curves relative to their boundaries, while keeping it transverse to $\mathcal{I}$. This part is also skipped over in Vogel's proof (as it is not necessary there) and we provide the key steps in Section~\ref{sec:uniqannulus} below.

\bigskip 

After all these steps, the proof of Theorem~\ref{thm:uniq} is complete. \qed

        \subsection{Correcting holonomy with semi-infinite ribbons} \label{sec:ribbontwist}

In this section, we explain how to use ribbons ending near attracting curves to modify the holonomy along faces of the polyhedral decomposition. This is essentially a local computation. We remark that the corresponding computation in Vogel's paper~\cite{V16}, namely equation (3-4), is not correct. We explain how to modify it.

Let $A = S^1 \times I$. For $t \in I$, we set $A_t \coloneqq A \times \{t\} \subset A \times I$. We consider a contact structure $\xi$ on $A \times I$ such that the characteristic foliation of $\xi$ on $A_t$, $t \in I$, is nonsingular, transverse to the boundary and outward pointing, and its closed characteristics are nondegenerate.

Notice that the condition on the nondegeneracy of the closed characteristics is not made explicit in Vogel's paper but will be crucial for the computations below.

We further consider a nowhere vanishing vector field $X$ on $A \times I$ which is tangent to $A_t$ for every $t \in I$, and which generates the characteristic foliation $A_t(\xi)$ with its \emph{opposite} orientation. For $s \geq 0$, we set
    $$\xi_s \coloneqq (\phi_X^s)^* \xi.$$

\begin{lem} \label{lem:ribbonconv}
    Under the above assumptions, there exists a neighborhood $U$ of $\partial A \times I$ such that $\xi_s$ converges uniformly to $H = \ker dt$ on $U$ in the $C^0$ topology as $s \rightarrow +\infty$.
\end{lem}

\begin{proof}
Let $\alpha$ be a contact form for $\xi$ which can be written as
$$\alpha \coloneqq \lambda_t + u_t dt,$$
where $u_t : A \rightarrow \R$ and $\lambda_t$, $t \in I$, are families of functions and $1$-forms on $A$.

The contact condition is equivalent to
$$\omega_t \coloneqq u_t d\lambda_t + \lambda_t \wedge du_t - \lambda_t \wedge \partial_t\lambda_t > 0.$$

We define $h, f : A \times I \rightarrow \R$ by 
\begin{align*}
    \iota_X d\lambda_t &\coloneqq h \lambda_t, \\
    \iota_X \omega_t &\coloneqq - f \lambda_t,
\end{align*}
so that $f > 0$. By definition, we have
$$u_t h - X\cdot u_t + \partial_t \lambda_t(X) = -f.$$
We compute:
\begin{align*}
    \mathcal{L}_X \alpha = \iota_X d\alpha &= \iota_X \left(du_t \wedge dt + d\lambda_t - \partial_t \lambda_t \wedge dt \right) \\
    &= (X \cdot u_t) dt + h\lambda_t - \partial_t\lambda_t(X) dt\\
    &= h(\lambda_t + u_t dt) + fdt\\
    &= h \alpha + fdt.
\end{align*}
Let $a_s, b_s : A \times I \rightarrow \R$ be functions such that 
$$(\phi_X^s)^* \alpha = a_s \alpha + b_s dt,$$
with $a_0 = 1$ and $b_0 = 0$. Writing $h_s \coloneqq (\phi_X^s)^*h$ and $f_s \coloneqq (\phi_X^s)^*f$, we compute:
\begin{align*}
    \frac{\partial}{\partial_s} (\phi_X^s)^*\alpha = (\phi_X^s)^* \mathcal{L}_X \alpha &= (\phi_X^s)^*(h \alpha+ fdt)\\
    &=h_s (a_s \alpha + b_s dt) + (\phi_X^s)^*(fdt) \\
    &= h_s a_s \alpha + (h_s b_s + f_s)dt,
\end{align*}
hence
\begin{align*}
    \partial_s a_s &= h_s a_s,\\
    \partial_s b_s &= h_s b_s + f_s.
\end{align*}
Solving these ODEs, we then get
$$a_s = \exp\left(\int_0^s h_\sigma \, d\sigma\right)$$
and 
\begin{align*}
    b_s &= \exp\left(\int_0^s h_\sigma \, d\sigma\right) \int_0^s\exp\left(- \int_0^\sigma h_v \, dv\right)f_\sigma \, d\sigma \\
    &= a_s \int_0^s a^{-1}_\sigma f_\sigma \, d\sigma.
\end{align*}
In conclusion:
$$\alpha_s \coloneqq (\phi_X^s)^* \alpha = a_s \left\{ \lambda_t + \left(u_t + \int_0^s a^{-1}_\sigma f_\sigma \, d\sigma\right) dt\right\}.$$

We are now left to show that 
\begin{align} \label{eq:lim}
    \lim_{s \rightarrow + \infty} \int_0^s a^{-1}_\sigma f_\sigma \, d\sigma = +\infty
\end{align}
on some neighborhood of $\partial A \times I$, so that the contact structure $\xi_s = \ker \alpha_s$ converges to the horizontal plane field $H$ in the $C^0$ topology there. Since the angle between $\xi_s$ and $H$ is decreasing as $s$ increases, by the contact condition, the convergence will be uniform (possibly on a slightly smaller open set).

On each $A_t$, $t \in I$, each flow line of $X$ intersecting $\partial A_t$ converges in positive time to a closed orbit $\gamma^\pm_t \subset A_t$, by Poincar\'{e}--Bendixson. Moreover, these orbits are nondegenerate by assumption. We denote by $U \subset A \times I$ the union of all the flow lines of $X$ intersecting $\partial A \times I$, so that setting $U_t \coloneqq U \cap A_t$, $\partial U_t = \partial A_t \cup \{\gamma^\pm_t\}$. It might be the case that $\gamma^-_t = \gamma^+_t$, if the characteristic foliation of $\xi$ on $A_t$ has a single closed orbit. See Figure~\ref{fig:charfol}.

\begin{figure}[H]
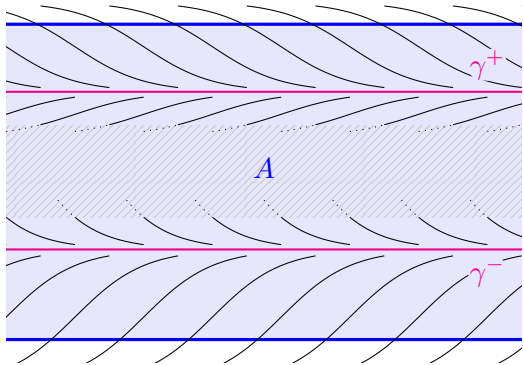

    \centering
    \includestandalone{tikz/annulus}
    \caption{The characteristic foliation near the boundary of the annulus $A$.}
    \label{fig:charfol}
\end{figure}

Note that there is a slightly larger open set $\overline{U} \subset V$ such that $A_t \cap V$ is convex for $\xi$. Then, we can further assume that the contact form $\alpha$ satisfies $d\alpha_{\vert U_t} > 0$, i.e., $d\lambda_t > 0$ on $U_t$. This implies that $h < 0$ on $U$, so that 
$$\int_0^s a^{-1}_\sigma f_\sigma \, d\sigma = \int_0^s\exp\left(- \int_0^\sigma h_v \, dv\right)f_\sigma \, d\sigma \geq \int_0^s f_\sigma \, d\sigma \geq (\min f) s$$
on $U$.
Since $f$ is bounded from below by some positive constant, the limit~\eqref{eq:lim} holds on $U$.
\end{proof}

\begin{rem}
    The attentive reader will have noticed that we only need the nonstrict inequality $h \leq 0$ on $U$. We do not know how to ensure this condition without some nondegeneracy condition on the $\gamma^\pm_t$'s, which essentially  
    amounts to a convexity condition.
\end{rem}

    \subsection{Filling polyhedra} \label{sec:fillpoly}

We next add in some missing details to complete the proof of~\cite[Lemma 4.14]{V16}. Vogel proves the existence part of the lemma and claims that the uniqueness should be clear. However, it seemed to us that this part is nontrivial and requires an explicit proof. We need to be particularly careful to ensure that the contact structures under consideration remain transverse to the fixed $1$-dimensional foliation $\mathcal{I}$.

Recall that $\mathcal{I}$ denotes a smooth $1$-dimensional foliation transverse to $\mathcal{F}$. In this section, we consider a fixed polyhedron $P$ in $M$, which is adapted to $\mathcal{I}$ and $\mathcal{F}$. We consider $\xi^t_\circ$, $t \in [0,1]$, a path of germs of contact structures near $\partial P$ and we assume that for every $t \in [0,1]$, $P$ is adapted to $(\xi^t_\circ, \mathcal{I})$. 

\begin{prop} \label{prop:extfol}
 Let $\xi^0$ and $\xi^1$ be contact structures on $P$ transverse to $\mathcal{I}$ which coincide with $\xi^0_\circ$ and $\xi^1_\circ$ near $\partial P$, respectively. Then there exists a path of contact structures $\xi^t$, $t \in [0,1]$, on $P$ from $\xi^0$ to $\xi^1$ such that for every $t \in [0,1]$, $\xi^t$ is transverse to $\mathcal{I}$ and coincides with $\xi^t_\circ$ near $\partial P$.
\end{prop}

        \subsubsection{Supporting foliations by disks}

The first ingredient in the proof of Proposition~\ref{prop:extfol} is the existence of contact structures on $P$ transverse to $\mathcal{I}$. In this section, we review the construction of Vogel and provide some more details.

Let $\xi_\circ$ be a germ of contact structure near $\partial P$ such that $P$ is adapted to $\xi_\circ$ and $\mathcal{I}$. We denote by $x^\pm_P$ the supporting vertices of $P$ and we assume that the characteristic foliation of $\xi_\circ$ on $\partial P$ spirals from $x^+_P$ to $x^-_P$.

We say that a (smooth) foliation by disks $\mathcal{D}$ on a neighborhood of $P$ \textbf{supports} $(\xi_\circ, \mathcal{I})$ if the following hold:
\begin{itemize}
    \item $\mathcal{I}$ is transverse to $\mathcal{D}$.
    \item $\mathcal{D}$ is transverse to $\partial P \setminus \{x^\pm_P\}$, and $\mathcal{D} \cap \partial P$ is transverse to the characteristic foliation of $\xi_\circ$.
    \item $\mathcal{D}$ is tangent to $\xi_\circ$ at $x^\pm_P$, and the characteristic foliation of $\xi_\circ$ on the leaves of $\mathcal{D}$ near $x^\pm_P$ has only one singularity, which is nondegenerate and elliptic. The singularities on the leaves just below $x^+_P$ and just above $x^-_P$ are contained in $\mathrm{int}(P)$.
\end{itemize}

Note that if $\mathcal{D}$ supports $(\xi_\circ, \mathcal{I})$ and if $\xi'_\circ$ is sufficiently $C^2$-close to $\xi_\circ$, then $\mathcal{D}$ supports $(\xi'_\circ, \mathcal{I})$ as well.

\begin{lem} \label{lem:existdisk}
    There exists a foliation $\mathcal{D}$ supporting $(\xi_\circ, \mathcal{I})$.
\end{lem}

\begin{proof}
    The proof is essentially the same as that of~\cite[Lemma 4.14]{V16}, but we provide more details near the supporting vertices. 

    In~\cite[Lemma 4.14]{V16}, Vogel constructs a foliation by disks on $P$ away from an arbitrarily small neighborhood of the supporting vertices which supports $(\xi_\circ, \mathcal{I})$. This crucially uses that $P$ is adapted to $\xi_\circ$ and $\mathcal{I}$. To extend this foliation over neighborhoods of the supporting vertices, we pick coordinates near $x^\pm_P$ in which $\xi_\circ$ becomes the standard contact structure $\ker \big( r^2 d\theta + dz\big)$, the polyhedron is very close to the tip of a tetrahedron whose boundary is transverse to $\partial_z$, and $\mathcal{I}$ is very close to a linear foliation (which is not necessarily tangent to $\partial_z$) transverse to the horizontal plane. In this local model, it is then easy to interpolate between horizontal disks and the disks constructed by Vogel: one can first complete the collection of circles on $\partial P$ transverse to $\xi_\circ$ so that they become horizontal as they approach $x^\pm_P$, and then fill them with disks transverse to $\mathcal{I}$. By realizing these disks as suitable graphs of functions matching the boundary circles, we can further arrange that the disks become horizontal when approaching the supporting vertices.
\end{proof}

Now, let $\xi^t_\circ$, $t \in [0,1]$, be a family of germs of contact structures near $\partial P$ such that for every $t \in [0,1]$, $P$ is adapted to $(\xi^t_\circ, \mathcal{I})$. We prove a ``short-term existence'' result for appropriate fillings of these boundary contact structures:

\begin{prop} \label{prop:extfol2}
    There exist $\epsilon > 0$ and a family of contact structures $\xi_\bullet^t$, $t \in [0,\epsilon)$, on a neighborhood of $P$ such that for every $t \in [0,\epsilon)$, $\xi_\bullet^t$ is transverse to $\mathcal{I}$ and coincides with $\xi^t_\circ$ near $\partial P$.
\end{prop}

\begin{proof}
    We essentially follow Vogel's strategy. We first pick a foliation $\mathcal{D}$ supporting $(\xi^0_\circ, \mathcal{I})$, and we fix $\epsilon > 0$ such that $\mathcal{D}$ supports $(\xi^t_\circ, \mathcal{I})$ as well, for every $t \in [0, \epsilon)$. Then, we choose a smooth family of embedded curves $\gamma_t$, $t \in [0,\epsilon)$, such that
    \begin{itemize}
        \item $\gamma_t$ starts at $x^-_P$, ends at $x^+_P$, and is contained in $\mathrm{int}(P)$ away from its endpoints,
        \item $\gamma_t$ is transverse to $\mathcal{D}$,
        \item $\gamma_t$ coincides with the locus where $\mathcal{D}$ and $\xi^t_\circ$ are tangent near $x^\pm_P$.
    \end{itemize}
    We can then pick a smooth path of vector fields $X_t$, $t \in [0,\epsilon)$, defined in a neighborhood of $P$ such that
    \begin{itemize}
        \item $X_t$ is tangent to $\mathcal{D}$,
        \item $X_t$ vanishes exactly along $\gamma_t$, and these singularities are elliptic in each leaf of $\mathcal{D}$,
        \item $X_t$ spans the characteristic foliation of $\xi^t_\circ$ on $\mathcal{D}$ near $\partial P$.
    \end{itemize}
    We then construct $\xi^t_\bullet$ by twisting $T\mathcal{D}$ along $X_t$, so that it matches $\xi^t_\circ$ near $\partial P$. By choosing the twisting to be very small away from $\partial P$, we can guarantee that $\xi^t_\bullet$ is transverse to $\mathcal{I}$.
\end{proof}

\begin{rem}
    The same methods can be used to prove a version with more parameters, for families of germs of contact structures $\xi^t_\circ$, $t \in D^k$, indexed by a $k$-dimensional disk.
\end{rem}

        \subsubsection{Uniqueness on the cylinder}

The second ingredient in the proof of Proposition~\ref{prop:extfol} is a uniqueness result for contact structures on the cylinder with prescribed characteristic foliation on the boundary that are transverse to a given \emph{fixed} vector field.

Let $\mathcal{C} \coloneqq D^2 \times [0,1]$ be a cylinder with coordinates $(x,y,z)$. We consider polar coordinates $(r, \theta)$ on the $(x,y)$-disk and set $c \coloneqq \{r=0\}$. Let $\xi_\bullet$ be a contact structure on $\mathcal{C}$ defined by a $1$-form $\alpha_\bullet$ of the form
\begin{align} \label{eq:radial}
    \alpha_\bullet \coloneqq dz + f r^2 d\theta,
\end{align}
where $f: \mathcal{C} \rightarrow \R_{>0}$ is a positive function such that $\partial_r (fr^2) > 0$ away from $c$. We further consider a smooth vector field $Z$ on $\mathcal{C}$ which is positively transverse to the horizontal disks $D_z \coloneqq D^2 \times\{z\}$, $z \in [0,1]$, and we assume that $Z$ is positively transverse to $\xi_\bullet$.

Notice that the characteristic foliation of $\xi_\bullet$ on the vertical boundary $\partial_v \mathcal{C} = \bigcup_{0 \leq z \leq 1} \mathcal{C}_z$ is nonsingular and every leaf spirals from the top circle $\mathcal{C}_1$ to the bottom circle $\mathcal{C}_0$, where $\mathcal{C}_z \coloneqq \partial D_z$, $z \in [0,1]$. Moreover, the characteristic foliation on each horizontal disk $D_z$ is radial and has a standard elliptic singularity at $0$.

Let $\Xi_Z(\mathcal{C}, \xi_\bullet)$ denote the space of positive contact structures $\xi$ on $\mathcal{C}$ satisfying
\begin{itemize}
    \item $\xi$ is positively transverse to $Z$,
    \item $\xi$ coincides with $\xi_\bullet$ near $\partial \mathcal{C}$.
\end{itemize}

\begin{prop} \label{prop:contractcyl}
    $\Xi_Z(\mathcal{C}, \xi_\bullet)$ is contractible.
\end{prop}

\begin{proof}
    We will show that every $\xi \in \Xi_Z(\mathcal{C}, \xi_\bullet)$ is homotopic within $\Xi_Z(\mathcal{C}, \xi_\bullet)$ to $\xi_\bullet$. The strategy will readily extend to families of contact structures parametrized by compact spaces.
    
    Let $\alpha$ be a contact form for $\xi$ of the form
    $$\alpha = u dz + \lambda_z,$$
    where $(\lambda_z)_{0 \leq z \leq 1}$ is a family of $1$-forms on $D^2$, and $u : \mathcal{C} \rightarrow \R$ is a function which equals $1$ near $\partial \mathcal{C}$. Be aware that $u$ might not be positive a priori, since we only assume that $\xi$ is transverse to $Z$ but not to $\partial_z$. Nonetheless, $\xi$ is \emph{tight}: we can extend it to a contact structure $\widehat{\xi}$ on $\R^3$ which is standard at infinity and is transverse to a vector field $\widehat{Z}$ extending $Z$ which is transverse to the horizontal planes $\R^2 \times \{z\}$ and coincides with $\partial_z$ near infinity. After a compactly supported isotopy, we can arrange that $\widehat{Z} = \partial_z$ everywhere. Then~\cite[Proposition 3.5.6]{ET} guarantees that $\widehat{\xi}$ is tight, and so is $\xi$.

    We will construct the desired homotopy in two steps.
    \begin{itemize}[leftmargin=*]
    \item \textit{Step 1: $\xi$ is homotopic within $\Xi_Z(\mathcal{C}, \xi_\bullet)$ to a contact structure $\widetilde{\xi}$ which is transverse to $\partial_z$ and which admits a contact form $\widetilde{\alpha}$ satisfying $d\widetilde{\alpha}_{\vert D_z} > 0$ for all $z \in [0,1]$.}
    
    After a small $C^\infty$ perturbation of $\xi$ away from $\partial \mathcal{C}$, we can assume that for every $z$, the characteristic foliation of $\xi$ on $D_z$ has isolated singularities which are nodes, saddles, or saddle-nodes (embryonic). Since $\xi$ is transverse to $Z$, all these singularities are positive, and since $\xi$ is tight, $D_z(\xi)$ has no closed leaf (this would be a Legendrian curve bounding an embedded disk with vanishing Thurston--Bennequin number); similarly, it has no cycles of saddle/embryonic singularities. Therefore, all those disks are convex; in particular the characteristic foliation is outward transverse to the boundary and there exists a family of \emph{positive} functions $v_z : D^2 \rightarrow \R_{>0}$ satisfying 
    \begin{align*}
        d\left(\frac{1}{v_z} \lambda_z\right)= \frac{1}{v_z^2} \left(\lambda_z \wedge dv_z + v_z d\lambda_z \right) > 0.
    \end{align*}
    We can further arrange that $v_z = 1$ for $z$ near $\partial [0,1]$.
    For $k > 0$, we consider the $1$-form 
    $$\widetilde{\alpha} \coloneqq k v_z dz + \lambda_z.$$
    Then $\widetilde{\alpha}(\partial_z) > 0$, and for $k$ large enough, $\widetilde{\alpha}$ is a positive contact form. However, $\ker \widetilde\alpha$ does not coincide with $\xi_\bullet$ near $\partial \mathcal{C}$. This can be fixed by modifying $v_z$. Near $\partial \mathcal{C}$, $\widetilde{\alpha} = k v_z dz + f r^2 d\theta$ and the contact condition reduces to 
    \begin{align} \label{eq:contactcond}
    \partial_r \ln v_z < \partial_r \ln (fr^2).  
    \end{align}
    We may assume that $k v_z > 1$. Recall that for $z$ near $\partial [0,1]$, $v_z = 1 > 1/k$ so we can replace it by a function $v_z = \varphi(z)$ which is monotonically increasing (resp.~decreasing) from $1/k$ to $1$ (resp.~$1$ to $1/k$) for $z$ near $0$ (resp.~near $1$), and~\eqref{eq:contactcond} is still satisfied. Moreover, we can arrange that the resulting contact structure is still transverse to $Z$, by making this modification sufficiently close to $\partial_v \mathcal{C}$. We then modify $v_z$ near $\partial D^2$ such that~\eqref{eq:contactcond} is still satisfied, and $v_z$ radially decreases to the constant $1/k$ near $\partial \mathcal{C}$. Once again, we can ensure that the resulting contact structure stays transverse to $Z$.

    One easily checks that for all $t \in [0,1]$, $(1-t) \alpha + t \widetilde{\alpha}$ is a contact form which defines a contact structure in $\Xi_Z(\mathcal{C}, \xi_\bullet)$. This procedure does not change the characteristic foliation on $D_z$ but deforms $\xi$ into a contact structure transverse to $\partial_z$ and which admits a contact form $\widetilde{\alpha}$ satisfying $d\widetilde{\alpha}_{\vert D_z} > 0$. 
    
    \item \textit{Step 2: $\widetilde{\xi}$ is homotopic within $\Xi_Z(\mathcal{C}, \xi_\bullet)$ to $\xi_\bullet$.} After rescaling $\widetilde{\alpha}$ by a positive function, we may assume that it is of the form $\widetilde{\alpha} = dz + \widetilde{\lambda}_z$, where $d\widetilde{\lambda}_z > 0$ and $\widetilde{\lambda}_z = g d\theta$ near $\partial \mathcal{C}$. Then for every $k' \geq 1$, $k' dz + \widetilde{\alpha}$ is a contact form which is positive on $Z$, and similarly for $k' dz+ g d\theta$. For $k'$ large enough, the same holds for $k' dz + (1-t) \widetilde{\lambda}_z + t f r^2 d\theta$, for all $t \in [0,1]$. We readily obtain a path of contact forms positive on $Z$ from $\widetilde{\alpha}$ to $\alpha_\bullet$. However, these forms do not quite coincide with $\alpha_\bullet$ near $\partial \mathcal{C}$, but they can easily be modified as in Step 1 to yield a path of contact structures from $\widetilde{\xi}$ to $\xi_\bullet$ within $\Xi_Z(\mathcal{C}, \xi_\bullet)$. \qedhere
    \end{itemize}
\end{proof}

        \subsubsection{Proof of Proposition~\ref{prop:extfol}}

By compactness and using Proposition~\ref{prop:extfol2}, there exist $N \geq 0$ and intervals $I_k = [a_k, a_{k+1}]$, $0 \leq k \leq N$, such that $a_0 = 0$, $a_{N+1} = 1$, and $a_k < a_{k+1}$, as well as paths of contact structures $\xi^t_{\bullet, k}$, $t \in I_k$, such that $\xi^t_{\bullet, k}$ is transverse to $\mathcal{I}$ and coincides with $\xi^t_\circ$ near $\partial P$. We can then modify these paths so that they agree at their consecutive endpoints, and also agree with $\xi^0$ and $\xi^1$ at $t=0$ and $t=1$, using Proposition~\ref{prop:contractcyl}. 

We treat the case $t=0$, the other cases being similar. We can find an embedded cylinder $\mathcal{C} \subset \mathrm{int}(P)$, obtained by rounding the corners of $P$ and shrinking it slightly, such that $\xi^0_\bullet$ and $\xi^0_{\bullet, 0}$ both coincide with $\xi^0_\circ$ near $\partial \mathcal{C}$, and such that $\xi^0_\circ$ is of the form~\eqref{eq:radial} near $\partial \mathcal{C}$ in suitable coordinates on $\mathcal{C}$. Then, Proposition~\ref{prop:contractcyl} provides a path $\widetilde{\xi}^t_{\bullet, 0}$, $t \in [0,1]$, of contact structures transverse to $\mathcal{I}$ from $\xi^0_\bullet$ to $\xi^0_{\bullet, 0}$ and which coincides with $\xi^0_\circ$ near $\partial P$. We can then concatenate this path with $\xi^t_{\bullet, 0}$ (and potentially perform some necessary yet irrelevant smoothing to make this path smooth). The picky reader will notice that the boundary condition is constant for the first half of the latter path, but it is easy to modify it by some small isotopy near $\partial P$ and reparametrize the time variable accordingly. Details are left to the reader.\qed

        \subsection{Horizontal contact structures on the thickened annulus} \label{sec:uniqannulus}

Recall that $I = [-1, 1]$. Let $N = S^1 \times I \times I$ be a thickened annulus with coordinates $(\theta, y,z)$. We consider a contact structure $\xi_0$ on $N$ defined by the contact form 
$$\alpha_0 = dz - u(\theta, y, z) d\theta$$
such that the following hold: 
\begin{itemize}
    \item $\partial_y u > 0$, which corresponds to the contact condition,
    \item Near $y=1$, $u$ is of the form
    \begin{align} \label{eq:vforynear1}
        u(\theta, y,z) = u(\theta, 1, z+1-y) = u_1(z+ 1-y),
    \end{align}
    where $u_1 : \R \rightarrow \R$ is a smooth function satisfying $u_1(0)=0$ and $u'_1 < 0$.
\end{itemize}
Notice that $\partial_z$ is positively transverse to $\xi_0$, $\partial_y$ is a Legendrian vector field, and by the second condition the characteristic foliation on the annulus $A_1 = S^1 \times \{1\} \times I$ has exactly one nondegenerate periodic orbit (along $z = 0$) and is inward pointing along the boundary. 

Let $\Xi_h(N,\xi_0)$ denote the space of contact structures $\xi$ on $N$ satisfying the following:
\begin{itemize}
    \item $\xi$ is transverse to $\partial_z$,
    \item $\xi$ coincides with $\xi_0$ near $\partial N$.
\end{itemize}
The subscript $h$ stands for `horizontal'. The following proposition might be well-known to experts but we were not able to find a complete proof in the literature.

\begin{prop} \label{prop:contractible}
    $\Xi_h(N,\xi_0)$ is contractible.
\end{prop}

Let $\xi \in \Xi_h(N,\xi_0)$. Then $\xi$ induces a family of diffeomorphisms $\varphi_\theta : I \rightarrow I$, $\theta \in S^1$, which coincide with the identity near $\partial I$, by considering the parallel transport on $D_\theta = \{\theta\} \times I \times I$ along $\xi \cap D_\theta$. More precisely, for every $\theta \in S^1$, we consider the vector field $X_\theta$ on $D_\theta$ tangent to $\xi \cap D_\theta$ and of the form $X_\theta = \partial_y + g(y,z) \partial_z$, and we define $\varphi_\theta$ as the time-$2$ map of the flow of $X_\theta$.

\begin{lem} \label{lem:samePT}
    If $\xi, \xi' \in \Xi_h(N,\xi_0)$ induce the same parallel transport $(\varphi_\theta)_{\theta\in S^1}$, then $\xi$ and $\xi'$ are homotopic within $\Xi_h(N,\xi_0)$.
\end{lem}

\begin{proof}
    We first apply a change of coordinates of the form $(\theta, y, z) \mapsto (\theta, y, f(\theta, y,z))$ such that $\partial_y$ becomes Legendrian for $\xi$. Note that the direction of $\partial_z$ remains unchanged. In these new coordinates, $\xi'$ induces a trivial parallel transport map. Therefore, there exists an isotopy $(\phi_t)_{0 \leq t \leq 1}$ of $N$ relative to $\partial N$ and transverse to $\partial_z$ such that $\partial_y$ is Legendrian for $\xi'_1 \coloneqq (\phi_1)_*\xi'$. Note that $\xi$ and $\xi_1'$ admit contact forms of the form
    \begin{align*}
        dz - v(\theta, y, z) d\theta
    \end{align*}
    where $v$ is determined by $\xi_0$ near $\partial N$, and the contact condition is simply $\partial_y v > 0$. Since this condition is convex in $v$, one easily constructs a homotopy between $\xi$ and $\xi_1'$ transverse to $\partial_z$ and supported away from $\partial N$.
\end{proof}

The key technical lemma to prove Proposition~\ref{prop:contractible} is an adaptation of the ``pulling-down the window'' argument from~\cite[Section 2.5]{ET}. This will allow us to modify the parallel transport of any contact structure in $\Xi_h(N,\xi_0)$ in a suitable way.

\begin{lem}[Pulling-down] \label{lem:pulldown}
    Let $\psi_\theta : I \rightarrow I$, $\theta \in S^1$, be a family of diffeomorphisms of $I$ coinciding with the identity near $\partial I$. There exists a diffeomorphism $f : I \rightarrow I$ coinciding with the identity near $\partial I$ such that the following holds. For every $\delta >0$ small enough, there exist $\xi, \xi' \in \Xi_h(N,\xi_0)$ such that
    \begin{enumerate}
        \item $\xi$ and $\xi'$ coincide with $\xi_0$ on $S^1 \times [-1, 1-\delta] \times I$ and are homotopic to $\xi_0$ via a homotopy in $\Xi_h(N,\xi_0)$ supported in $S^1 \times [1-\delta, 1] \times I$,
        \item The parallel transport induced by $\xi$ is $(f \circ \psi_\theta)_\theta$,
        \item The parallel transport induced by $\xi'$ is $f$.
    \end{enumerate}
\end{lem}

\begin{proof}
    Let $\sigma > 0$ be such that $\bigcup_{\theta \in S^1} \mathrm{supp}(\psi_\theta) \subset [-1+\sigma, 1-\sigma]$, and $\sigma \leq 0.1$. For $\epsilon > 0$ small enough, we consider a diffeomorphism $f = f_\epsilon : I \rightarrow I$ satisfying
    \begin{itemize}
        \item $f(z) = z$ for $z$ close to $\partial I$,
        \item $\forall z \in I$, $f(z) \leq z$,
        \item $\forall z \in [-1, -1+\sigma]$, $f'(z) \leq 1$,
        \item $\forall z \in [-1+\sigma, 1-\sigma]$, $f(z) = \epsilon (z+2) -1$,
        \item If $f(z) \geq 0$ then $f'(z) \geq 1$.
    \end{itemize}
    See Figure~\ref{fig:function}. We emphasize that this ``pull-down profile'' is \emph{different} from the one in~\cite[Proposition 2.5.1]{ET} as it has a much more sizeable effect. As a result, more care is needed in the specific choice of pull-down profile to ensure that any sufficiently ``negative'' parallel transport can be matched, and that this pull-down can be undone by a contact isotopy.
    
\begin{figure}[ht]
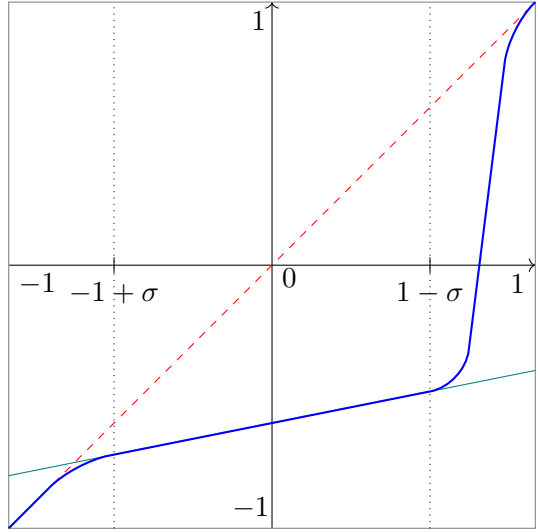

    \centering
    \includestandalone{tikz/function}
    \caption{The function $f$ in blue. The dashed red line is the identity, and the teal line has slope $\epsilon$.}
    \label{fig:function}
\end{figure}
    
    We write $f_\theta \coloneqq f \circ \psi_\theta$. Let $\delta > 0$ (independent of $\epsilon$) be small enough so that~\eqref{eq:vforynear1} holds for $y \in [1-\delta, 1]$, and $w(z) \coloneqq u(\theta, 1, z + \delta/2)=u_1(z+\delta/2)$ is positive for $z \in [-1, -1/2]$ and negative for $z \in [1/2, 1]$. Note that $w$ is decreasing. We choose a bump function $\kappa : I \rightarrow [0,1]$ such that $\kappa$ is supported on $[1-\delta, 1]$ and $\kappa \equiv 1$ near $1-\delta/2$. Let $\Phi : N \rightarrow N$ be the diffeomorphism defined by 
    $$\Phi(\theta, y, z) = \left(\theta, y, \big(1-\kappa(y)\big)z + \kappa(y) f_\theta(z)\right).$$
    The slope of the restriction of $\Phi_* \xi_0$ to $A_{1-\delta/2} = S^1 \times \{1-\delta/2\} \times I$ at the point $(\theta, f_\theta(z))$ is given by
    $$s(\theta,z) = \partial_\theta f_\theta(z) + f'_\theta(z) u_1(z+\delta/2).$$
    We want to choose $\epsilon$ such that
    \begin{align} \label{ineq:slope}
    s(\theta, z) < u_1 \circ f_\theta(z),
    \end{align}
    where the right-hand side is the slope of $\xi_0$ at the corresponding point along the face $y=1$. We consider several cases.
    \begin{itemize}
        \item $z \in [-1, -1+\sigma]$. Then $f_\theta(z) = f(z)$ and $f'(z) \leq 1$ so 
        $$s(\theta,z) = f'(z) u_1(z+\delta/2) < u_1(z) \leq u_1\big(f(z)\big)$$
        since $u_1(z+\delta/2) \geq 0$ and $u_1$ is decreasing.
        \item $z \in [-1+\sigma, 1-\sigma]$. Then $f(z) = \epsilon (z+2) -1$ and
        \begin{align*}
            s(\theta, z) &= \epsilon \big( \partial_\theta \psi_\theta(z) + \psi'_\theta(z) u_1(z+\delta/2) \big), \\
            u_1 \circ f_\theta(z) &= u_1\big(\epsilon (\psi_\theta(z)+2) - 1\big),
        \end{align*}
        and~\eqref{ineq:slope} is satisfied for $\epsilon > 0$ small enough, since $u_1(-1) > 0$.
        \item $z \in [1-\sigma, 1]$. Then $f_\theta(z) = f(z)$ and $u_1(z+\delta/2) < 0$. If $u_1\big(f(z)\big) > 0$, then inequality~\eqref{ineq:slope} is automatically satisfied since the left-hand side is negative and the right-hand side is positive. Otherwise, $u_1\big(f(z)\big) \leq 0$ so $f(z) \geq 0$
        and $f'(z) \geq 1$ by assumption, hence inequality~\eqref{ineq:slope} is satisfied again.
    \end{itemize}

We now define $\xi$ as follows:
\begin{itemize}
    \item On $S^1 \times [-1, 1-\delta/2] \times I$, $\xi \coloneqq \Phi_* \xi_0$,
    \item On $S^1 \times [1-\delta/2, 1] \times I$, we choose $\xi$ of the form 
    $$\xi \coloneqq \ker \big( dz - v(\theta, y,z) d\theta \big),$$
    where $v : S^1 \times [1-\delta/2, 1] \times I \rightarrow \R$ is a smooth function so that $\xi$ coincides with $\Phi_* \xi_0$ near $y = 1- \delta/2$, $v$ coincides with $u$ near $y=1$, and $\partial_yv(\theta, y,z) > 0$.
    The existence of this function is guaranteed by the inequality~\eqref{ineq:slope}, and the last condition ensures that $\xi$ is contact.
\end{itemize}
By definition, $\xi \in \Xi_h(N,\xi_0)$ and its induced parallel transport is exactly $(f_\theta)_\theta$. To construct $\xi'$, we apply the same procedure using $f$ instead of $(f_\theta)_\theta$.

We now argue that $\xi$ and $\xi'$ are homotopic to $\xi_0$ as in the first item of the lemma. For that, we observe that for $\epsilon > 0$ small enough, the construction of $\xi$ can be upgraded to a $1$-parameter family of contact structures $(\xi^t)_{0 \leq t \leq 1}$, $\xi^t \in \Xi_h(N,\xi_0)$, by considering an isotopy $\Phi_t$ of the form
$$\Phi_t(\theta, y, z) = \big(\theta, y, (1-\kappa(y))z + \kappa(y) f\circ \psi^t_\theta(z)\big),$$
where $\psi^t_\theta(z) = (1-t) z + t \psi_\theta(z)$. We define $\xi^t$ as $\Phi^t_*\xi_0$ on $S^1 \times [-1, 1-\delta/2] \times I$ and we complete it by twisting along the $y$-direction as before, so that $\xi^0 = \xi'$ and $\xi^1 = \xi$.

We finally apply the same procedure to a family of increasing diffeomorphisms $f_t : I \rightarrow I$, $t \in [0,1]$, from $f$ to $\mathrm{id}$ and satisfying
\begin{itemize}
    \item $f_t = \mathrm{id}$ near $\partial I$,
    \item $f_t(z) \leq z$,
    \item If $z \leq 0$ then $f_t'(z) \leq 1$,
    \item If $f_t(z) \geq 0$ then $f_t'(z) \geq 1$,
\end{itemize}
to get a homotopy of contact structures in $\Xi_h(N,\xi_0)$ between $\xi'$ and $\xi_0$ supported in $S^1 \times [1-\delta, 1] \times I$. The construction of such a family $(f_t)$ is not hard and is left to the reader. The key observation is that the slope inequality needed to perform the previous construction is 
$$\forall z \in I, \qquad f_t'(z) \, u_1(z + \delta/2) < u_1 \circ f_t(z),$$
which is satisfied for our choice of $f_t$'s.
\end{proof}

\begin{proof}[Proof of Proposition~\ref{prop:contractible}]
    We will prove that $\Xi_h(N,\xi_0)$ is path-connected, which is enough for the purpose of this paper. The proof can easily be upgraded to show that $\Xi_h(N,\xi_0)$ is weakly contractible, which implies contractibility by Whitehead's theorem.

    Let $\xi_1 \in \Xi_h(N,\xi_0)$ and let $(\varphi_\theta)_\theta$ denote its induced parallel transport. Let $\delta > 0$ be small enough so that $\xi_0$ and $\xi_1$ coincide on a $\delta$-neighborhood of $\partial N$. Applying Lemma~\ref{lem:pulldown} to the family $\psi_\theta \coloneqq \varphi^{-1}_\theta$, $\theta \in S^1$, yields a diffeomorphism $f : I \rightarrow I$ and a homotopy in $\Xi_h(N,\xi_0)$ between $\xi_1$ and a contact structure $\xi_1'$ inducing the parallel transport $f \circ \psi_\theta \circ \varphi_\theta = f$. We also obtain a homotopy in $\Xi_h(N,\xi_0)$ between $\xi_0$ and a contact structure $\xi'_0$ inducing the parallel transport $f$. By Lemma~\ref{lem:samePT}, $\xi'_0$ and $\xi'_1$ are homotopic in $\Xi_h(N,\xi_0)$, and so are $\xi_0$ and $\xi_1$.
\end{proof}

    \section{Deformation of weak symplectic fillings}\label{sec:Def_weak}

In this section, $V$ denotes a compact oriented $4$-manifold with (possibly disconnected) boundary $\partial V = M$. The results we prove are more general than necessary for our applications (in particular Proposition~\ref{propintro:deffill}) but may be of independent interest. We will essentially adapt and streamline ideas from~\cite{E91} and~\cite{E04}, which are probably well-known to the experts.

We believe that this strategy extends to higher dimensions, by considering exact weak symplectic fillings in the sense of~\cite[Definition 4]{MNW13}.

        \subsection{Pre-Liouville structures}

Let $\mathfrak{L}(V)$ denote the space of Liouville structures on $V$, i.e., the space of (smooth) $1$-forms $\lambda \in \Omega^1(V)$ satisfying
\begin{itemize}
    \item $d\lambda$ is a symplectic form on $V$,
    \item $\lambda_{\vert \partial V}$ is a contact form on $\partial V$.
\end{itemize}
Let 
$\mathscr{L}(V) \coloneqq \pi_0 \big(\mathfrak{L}(V)\big)$
denote the set of homotopy classes of Liouville structures on $V$.
We also define a set of (homotopy classes of) \textbf{pre-Liouville structures} on $V$:

\begin{defn}
    A \textbf{pre-Liouville structure} on $V$ is a pair $(\lambda, \xi)$, where $\lambda \in \Omega^1(V)$ and $\xi$ is a (cooriented) contact structure on $\partial V$, such that $d \lambda$ is symplectic and dominates $\xi$ along $\partial V$: $d\lambda_{\vert \xi}$ is a nondegenerate $2$-form on $\xi$ realizing the same orientation as $\xi$.
\end{defn}

In other words, a pre-Liouville structure is the data of a contact structure $\xi$ on $\partial V$ together with (the primitive of) an exact weak symplectic filling of $\xi$. Pre-Liouville structures are at least $C^1$-regular and form an open set for the $C^1$ topology. They can easily be smoothed, in a homotopically unique way, so we will not specify the precise regularity we are considering.

We denote by $\mathfrak{pL}(V)$ the space of pre-Liouville structures on $V$. There is a natural continuous `forgetful map'
$$
\begin{array}{rrcl}
    \mathfrak{p} : & \mathfrak{L}(V) & \longrightarrow & \mathfrak{pL}(V) \\
                & \lambda               & \longmapsto & \big(\lambda, \ker \lambda_{\vert \partial V}\big).
\end{array}
$$
We also define
$$\wp\mathscr{L}(V) \coloneqq \pi_0\big( \mathfrak{pL}(V) \big)$$
as the set of homotopy classes of pre-Liouville structures on $V$. The map $\mathfrak{p}$ above naturally induces a map
$$\wp : \mathscr{L}(V) \longrightarrow \wp\mathscr{L}(V).$$

The main result of this section is

\begin{thm} \label{thm:preliouv}
    The map $\wp$ is bijective.
\end{thm}

Lemma~\ref{lem:filling} below will show that $\wp$ is surjective while Proposition~\ref{prop:deffill} will imply that $\wp$ is injective. Our techniques can easily be extended to show:

\begin{thm}
    The map $\mathfrak{p}$ is a (weak) homotopy equivalence.
\end{thm}

We will not need this stronger result and leave details of the proof to the interested reader.

        \subsection{Straightening near the boundary}

The next lemma is a well-known result due to Eliashberg: a weak symplectic filling of a contact $3$-manifold, which is in addition exact, can be deformed near the boundary into a Liouville filling. It serves as a motivation and warm-up for Proposition~\ref{prop:deffill} below.

\begin{lem} \label{lem:filling}
Let $(\lambda, \xi)$ be a pre-Liouville structure on $V$. Then, there exists a $1$-parameter family $(\lambda_t)_{0\leq t \leq 1}$ of $1$-forms on $V$ such that
    \begin{itemize}
    \item[(a)] $\lambda_0 = \lambda$ and ${\lambda_1}_{\vert \partial V}$ is a contact form for $\xi$,
    \item[(b)] For every $0 \leq t \leq 1$, $(\lambda_t, \xi)$ is a pre-Liouville structure.
    \end{itemize}
Furthermore, we can assume that $(\lambda_t)_t$ is constant away from an arbitrarily small neighborhood of $\partial V$.
\end{lem}

\begin{proof}
We essentially follow the proofs of~\cite[Proposition 3.1]{E91} and~\cite[Proposition 4.1]{E04} with some additional details. We write $\beta \coloneqq \lambda_{\vert \partial V}$. By assumption, $d\beta_{\vert \xi} > 0$ and there exists a (unique) contact form $\alpha$ for $\xi$ such that $d\beta_{\vert \xi} = d\alpha_{\vert \xi}$.\footnote{Here, we crucially use that $M$ is $3$-dimensional. In higher dimensions, one should use an appropriate notion of weak filling (see~\cite[Definition 4]{MNW13}).} We define $\gamma \coloneqq \beta - \alpha$, so that $\alpha \wedge d\gamma = 0$. On $(0,1] \times M$, we consider the Liouville form $$\widetilde{\lambda} \coloneqq t\alpha + \gamma$$ and we write $\widetilde{\omega} \coloneqq d\widetilde{\lambda}.$ We obtain two coisotropic embeddings of the presymplectic manifold $(M, d\beta)$: one as the boundary of $\big(V, \omega = d\lambda \big)$ and one as the boundary $\{1\} \times M$ of $\big( (0,1] \times M, \widetilde{\omega})$. By the local uniqueness of coisotropic embeddings~\cite{G82}, there exist
    \begin{itemize}
    \item A neighborhood $\mathcal{U}$ of $\{1\} \times M$ in $(0,1] \times M$,
    \item A neighborhood $\mathcal{V}$ of $\partial V$ in $V$,
    \item A symplectomorphism $\psi : \big(\mathcal{U}, \widetilde{\omega} \big) \rightarrow \big( \mathcal{V}, \omega\big)$ such that $\psi_{\vert \{1\} \times M}$ coincides with the identification $M \cong \partial V$.
    \end{itemize}
Let us briefly sketch the proof in our context. First, we fix coordinates $V \cong (0, 1]_t \times M$ near $\partial V$ extending the identification $\partial V \cong M$, and we define a vector field $X$ near $\partial V$ by $\omega(X, \cdot \ ) = \alpha$. Since $\omega_{\vert \partial V} = d\beta$ and $d\beta_{\vert \xi} = d\alpha_{\vert \xi}$, it is easy to see that $X$ is transverse to $\partial V$ and is outward pointing. We can use the flow of $X$ near $\partial V$ to define a diffeomorphism $\psi_0 : \mathcal{U}_0 \rightarrow \mathcal{V}_0$ from a neighborhood of $\{1\} \times M \subset (0,1] \times M$ to a neighborhood of $\partial V \subset V$ such that ${\psi_0}_{\vert \partial V}$ coincides with $\partial V \cong M$, and $\omega_0 \coloneqq \psi_0^*\omega$ satisfies $(\iota_{\partial_t} \omega_0)_{\vert \{1\} \times M} = (\iota_{\partial_t}\widetilde{\omega})_{\vert \{1\} \times M}$. Then, $\omega_0$ and $\widetilde{\omega}$ agree on $\{1\} \times M$ and we can apply the Darboux--Moser--Weinstein theorem to obtain the desired symplectomorphism $\psi$. Note that this strategy adapts in a straightforward way to a parametric setting; this will be useful in the proof of Proposition~\ref{prop:deffill} below.

By shrinking $\mathcal{U}$ and $\mathcal{V}$, we can further assume that $\mathcal{U}$ is of the form $(1-\epsilon, 1] \times M$ for a sufficiently small $\epsilon > 0$. We obtain coordinates on a tubular neighborhood of $\partial V$ in $V$ in which $\lambda$ becomes $$t\alpha + \gamma + \theta$$ for some closed $1$-form $\theta$ satisfying $\theta_{\vert \partial V} = 0$. This readily implies that $\theta$ is exact, and we write $\theta = df$ for some function $f : (1-\epsilon, 1] \times M \rightarrow \R$. Let $\varphi_0 : (1-\epsilon , 1] \rightarrow [0,1]$ be a smooth nonincreasing cutoff function such that $\varphi_0 \equiv 1$ near $1-\epsilon$ and $\varphi_0 \equiv 0$ on $(1 - \epsilon/2, 1]$. We then define a $1$-form $\lambda_{1/2}$ on $V$ such that $\lambda_{1/2} = \lambda$ outside of $\mathcal{V}$, and 
$$\lambda_{1/2} \coloneqq t\alpha + \gamma + d(\varphi_0 f)$$
in $\mathcal{V}$. Note that $d\lambda_{1/2} = d\lambda$, $\lambda_{1/2} = t\alpha + \gamma$ near $\partial V$, and there is an obvious homotopy between $\lambda$ and $\lambda_{1/2}$ satisfying condition (b). Let $C > 1$ and consider a smooth nondecreasing function $\varphi_1 : (1-\epsilon/2, 1] \rightarrow [1, C]$ such that $\varphi_1 \equiv 1$ near $1-\epsilon/2$ and $\varphi_1 \equiv C$ on $(1-\epsilon/4, 1]$. We then define
$$\lambda_{3/4} \coloneqq t \varphi_1(t) \alpha + \gamma$$
on $(1-\epsilon/2, 1] \times M$ and extend it to the rest of $V$ by $\lambda_{1/2}$. Since $\alpha \wedge d\gamma = 0$, one easily checks that $d \lambda_{3/4}$ is symplectic and dominates $\xi$ along $\partial V$. Once again, there is an obvious homotopy from $\lambda_{1/2}$ to $\lambda_{3/4}$ satisfying condition (b). Finally, let $\varphi_2 : (1-\epsilon/4 , 1] \rightarrow [0,1]$ be a smooth nonincreasing cutoff function such that $\varphi_2 \equiv 1$ near $1-\epsilon/4$ and $\varphi_2 \equiv 0$ near $1$. We define
$$\lambda_1 \coloneqq C t \alpha + \varphi_2(t) \gamma$$
and extend it to the rest of $V$ by $\lambda_{3/4}$. For $C$ large enough, $d\lambda_1$ is symplectic and $\lambda_1$ is homotopic to $\lambda_{3/4}$ through $1$-forms satisfying condition $(b)$. On $\partial V$, ${\lambda_1}_{\vert \partial V} = C \alpha$ is a contact form for $\xi$.
\end{proof}

\begin{prop} \label{prop:deffill}
Let $(\lambda_t)_{t \in [0,1]}$ be a path of $1$-forms on $V$, and $(\xi_t)_{t \in [0,1]}$ be a path of contact structures on $\partial V$. Assume the following:
\begin{itemize}
    \item[(a)] For $i \in \{0,1\}$, $\lambda_i$ is a Liouville form and ${\lambda_i}_{\vert \partial V}$ is a contact form for $\xi_i$.
    \item[(b)] For every $t \in [0,1]$, $(\lambda_t, \xi_t)$ is a pre-Liouville structure.
\end{itemize}
Then $\lambda_0$ and $\lambda_1$ are Liouville homotopic, hence their completions are exact symplectomorphic.
\end{prop}

\begin{proof}
First of all, we can assume that the path $(\xi_t)_t$ is constant and equal to a fixed contact structure $\xi$ after pulling back $(\lambda_t)_t$ by an isotopy of $V$ supported near $\partial V$. 

The deformation of $\lambda$ near $\partial V$ in the proof of Lemma~\ref{lem:filling} can be performed in a parametric way, by first using a parametric version of the local uniqueness of coisotropic embeddings. Moreover, if $\lambda$ already restricts to a contact form for $\xi$ on $\partial V$, then so do the forms $\lambda_t$, $0 \leq t \leq 1$, since in that case $\gamma = 0$. Applying this to the path $(\lambda_t)_t$, we obtain a family of $1$-forms $(\lambda_{s,t})_{0 \leq s,t \leq 1}$ such that 
\begin{itemize}
    \item For every $t$, $\lambda_{0,t} = \lambda_t$,
    \item For every $t$, $(\lambda_{s,t})_s$ satisfies the conditions of Lemma~\ref{lem:filling},
    \item For every $s$, ${\lambda_{s,0}}_{\vert \partial V}$ and ${\lambda_{s,1}}_{\vert \partial V}$ are contact forms for $\xi$.
\end{itemize}
We obtain a Liouville homotopy between $\lambda_0$ and $\lambda_1$ by concatenating the Liouville homotopies $(\lambda_{s,0})_s$, $(\lambda_{1,t})_t$ and $(\lambda_{1-s,1})_s$.
\end{proof}

    \section{Liouville structures from foliations}

In this section, $\mathcal{F}$ denotes a hypertaut $C^1$-foliation on $M$. By Construction~\ref{construction}, we can associate to it a Liouville structure on $[-1,1] \times M$, after making a number of choices. We will now show that the resulting Liouville structure is unique up to deformation. We shall call this a(n infinitesimal) \emph{Liouville thickening} of $\mathcal{F}$. We then consider the special case when $\mathcal{F}$ is $C^2$ and compare it with a construction of Jonathan Zung~\cite{Z21}.

        \subsection{Liouville thickenings and proof of Theorem~\ref{thmintro:liouv}}

Recall that a $C^1$-foliation $\mathcal{F}$ is hypertaut if there exists an exact $2$-form $\omega = d\beta$ satisfying $\omega_{\vert T \mathcal{F}} > 0$ (see Definition~\ref{def:hypertaut}). Foliations without closed leaves are taut, and many of them are automatically hypertaut:

\begin{lem} \label{lem:hypertaut}
    Let $\mathcal{F}$ be a taut $C^1$-foliation on $M$.
    \begin{itemize}
        \item If $M$ is a rational homology sphere, then $\mathcal{F}$ is hypertaut.
        \item If $\mathcal{F}$ is $C^2$, then it is hypertaut if and only if it has no closed leaves and it is not $C^0$-approximated by fibrations.
    \end{itemize}
\end{lem}

\begin{proof}
    The first item immediately follows from the fact that a taut $C^1$-foliation admits a \emph{closed} dominating $2$-form, which is automatically exact on a rational homology sphere.

    For the second item, Stokes' Theorem implies that hypertaut foliations have no closed leaves and cannot be approximated by fibrations. For the converse direction, $\mathcal{F}$ must have holonomy since a coorientable $C^2$-foliation without holonomy on a closed $3$-manifold is approximated by fibrations by~\cite[Corollary 1.2.3]{ET}. Therefore, every minimal set has a curve with attracting (linear) holonomy by Sacksteder's Theorem~\cite{S65} and an argument of Ghys in the minimal case; see~\cite[Theorem 1.2.7]{ET}. It follows that $\mathcal{F}$ has enough holonomy, and Proposition~\ref{prop:hypertaut} implies that $\mathcal{F}$ is hypertaut.
\end{proof}

Construction~\ref{construction} from the Introduction involves a certain number of choices that we recall here:
\begin{itemize}
    \item A $1$-form $\beta$ such that $d\beta_{\vert T\mathcal{F}} > 0$,
    \item A (continuous) $1$-form $\alpha$ such that $\ker \alpha = T \mathcal{F}$, and a smoothing $\widetilde{\alpha}$ thereof satisfying $\widetilde{\alpha} \wedge d\beta > 0$,
    \item Contact approximations $\xi_\pm$ of $\mathcal{F}$ such that $d\beta_{\vert \xi_\pm} > 0$,
    \item An $\epsilon > 0$ small enough.
\end{itemize}
In this way, we obtain a pre-Liouville structure $(\lambda_\mathrm{pre}, \xi_- \sqcup \xi_+)$ defined by 
\begin{align} \label{eq:preliouv}
    \lambda_\mathrm{pre} \coloneqq \beta + \epsilon \tau \widetilde{\alpha},
\end{align}
and we can apply Lemma~\ref{lem:filling} to obtain the desired Liouville structure on $[-1,1]_\tau \times M$. In particular, we have:

\begin{prop} \label{prop:welldef}
    The previous procedure yields a Liouville structure which is well-defined up to Liouville homotopy.
\end{prop}

\begin{proof}
    Let $\mathrm{dvol}$ be an arbitrary volume form on $M$. We proceed in two steps.
    
    \begin{itemize}[leftmargin=*]
    \item \textit{Step 1.} We first consider a smooth $1$-form $\beta$ satisfying $d\beta_{\vert T\mathcal{F}} >0$, and show that the Liouville structures obtained from Construction~\ref{construction} are all Liouville homotopic, for this specific choice of $\beta$. Let $Z = Z_\beta$ be the vector field defined by $\iota_{Z} \mathrm{dvol} = d\beta$. By assumption, $Z$ is positively transverse to $\mathcal{F}$ and induces a smooth foliation $\mathcal{I} = \mathcal{I}_\beta$.
    
    Let $\mathcal{V} = \mathcal{V}_\beta$ be a neighborhood of $T\mathcal{F}$ as in Theorem~\ref{thmintrobeta:uniq} for the line field $\mathcal{I}$. We then choose 
    \begin{itemize}
        \item A smooth $1$-form $\widetilde{\alpha}$ satisfying $\widetilde{\alpha} \wedge d\beta > 0$,
        \item Positive and negative contact approximations $\xi_\pm$ of $T \mathcal{F}$ in $\mathcal{V}$, both transverse to $\mathcal{I}$. 
    \end{itemize}
    Then there exists $\overline{\epsilon} = \overline{\epsilon}_{\beta, \widetilde{\alpha}, \xi_\pm}$ such that for every $0 < \epsilon < \overline{\epsilon}$, the $1$-form defined by~\eqref{eq:preliouv} induces a pre-Liouville structure $(\lambda_\mathrm{pre}, \xi_- \sqcup \xi_+)$. In particular, the latter does not depend on $\epsilon$ up to pre-Liouville homotopy.

    Let us now consider another smooth $1$-form satisfying $\widetilde{\alpha}'\wedge d\beta > 0$, and different contact approximations $\xi'_\pm \in \mathcal{V}$. We set
    $$\lambda'_\mathrm{pre} \coloneqq \beta + \epsilon \tau \widetilde{\alpha}',$$
    for $\epsilon > 0$ small enough, so that $(\lambda'_\mathrm{pre}, \xi'_- \sqcup \xi'_+)$ is also a pre-Liouville structure.

    By Theorem~\ref{thmintrobeta:uniq}, there exist paths of contact structures $\big(\xi^t_\pm\big)_{t \in [0,1]}$ such that $\xi^0_\pm = \xi_\pm$, $\xi^1_\pm = \xi'_\pm$, and every $\xi^t_\pm$ is transverse to $\mathcal{I}$ for $t \in [0,1]$. This means that $d\beta_{\xi^t_\pm} > 0$.

    We construct a path of pre-Liouville structures from $(\lambda_\mathrm{pre}, \xi_- \sqcup \xi_+)$ to $(\lambda'_\mathrm{pre}, \xi'_- \sqcup \xi'_+)$ as follows. First, note that for every $\kappa \geq 1$, both 
    \begin{align*}
        \lambda_{\mathrm{pre},\kappa} &\coloneqq \kappa \beta + \epsilon \tau \widetilde{\alpha}, \\
        \lambda'_{\mathrm{pre},\kappa} &\coloneqq \kappa \beta + \epsilon \tau \widetilde{\alpha}',
    \end{align*}
    induce pre-Liouville structures $\big(\lambda_{\mathrm{pre},\kappa}, \xi_- \sqcup \xi_+\big)$ and $\big(\lambda'_{\mathrm{pre},\kappa}, \xi'_- \sqcup \xi'_+\big)$, for all $\epsilon$ sufficiently small. Moreover, a simple computation shows that for $\kappa$ large enough, the $1$-forms
    $$\lambda^t_{\mathrm{pre}, \kappa} \coloneqq \kappa \beta + \tau \epsilon \big((1-t) \widetilde{\alpha} + t \widetilde{\alpha}' \big)$$
    determine a pre-Liouville structure $\big(\lambda^t_{\mathrm{pre}, \kappa}, \xi^t_- \sqcup \xi^t_+\big)$ for every $t \in [0,1]$. Therefore, we obtain the desired path of pre-Liouville structures by choosing $K \gg 1$ large enough and concatenating
    \begin{itemize}
        \item $\big(\lambda_{\mathrm{pre}, \kappa}, \xi_- \sqcup \xi_+ \big)$ for $1 \leq \kappa \leq K$,
        \item $\big(\lambda^t_{\mathrm{pre}, K}, \xi^t_- \sqcup \xi^t_+ \big)$ for $0 \leq t \leq 1$,
        \item $\big(\lambda'_{\mathrm{pre}, \kappa}, \xi'_- \sqcup \xi'_+ \big)$ for $1 \leq \kappa \leq K$.
    \end{itemize}
    All the paths of pre-Liouville structures can now be deformed to paths of Liouville structures by Theorem~\ref{thm:preliouv}. This shows that the Liouville thickening depends only on $\beta$.

    \item \textit{Step 2.} We now consider another $1$-form $\beta'$ satisfying $d\beta'_{\vert T \mathcal{F}} > 0$. Let $\widetilde{\alpha}$ be a $1$-form such that $\widetilde{\alpha} \wedge d\beta > 0$ and $\widetilde{\alpha} \wedge d\beta' > 0$, and let $\xi_\pm$ be contact approximations to $\mathcal{F}$ such that $d\beta_{\vert \xi_\pm} > 0$ and $d\beta'_{\vert \xi_\pm} > 0$. We then define
    \begin{align*}
        \lambda'_\mathrm{pre} &\coloneqq \beta' + \epsilon \tau \widetilde{\alpha},
    \end{align*}
    which induces a pre-Liouville structure for $\xi_- \sqcup \xi_+$ for $\epsilon > 0$ small enough. It is then easy to check that 
    $$\lambda^t_\mathrm{pre} \coloneqq \big((1-t) \beta + t \beta'\big) + \epsilon \tau \widetilde{\alpha}$$
    also induces a pre-Liouville structure for $\xi_- \sqcup \xi_+$, for every $t \in [0,1]$. Therefore, the Liouville thickening of $\mathcal{F}$ does not depend on the choice of $\beta$ up to Liouville homotopy. \qedhere
    \end{itemize}
\end{proof}

We call such a Liouville structure a/the \textbf{Liouville thickening} of the foliation $\mathcal{F}$.

\begin{rem}
    More generally, if $\mathcal{F}$ is a hypertaut $C^0$-foliation (in the sense of~\cite{B16, KR17}), one can still use Construction~\ref{construction} to associate to it a collection of homotopy classes of Liouville structures $\mathscr{L}_\mathcal{F}$ that \emph{only} depends on the choice of contact approximations. However, if $\mathcal{F}$ is not $C^1$, it can admit several nonisotopic contact approximations, so $\mathscr{L}_\mathcal{F}$ might not be reduced to a point. On the other hand, all known nonequivalent contact approximations of foliations are distinguished by \emph{Giroux torsion}, and Liouville fillable contact structures are known to have vanishing Giroux torsion. Thus, it could still be possible that this construction yields a well-defined Liouville structure for hypertaut foliations of class $C^0$, although we refrain from positing this as a conjecture.
\end{rem}

We now turn to the main result of this section. We will consider several natural equivalence relations for hypertaut foliations, and describe their effect on Liouville thickenings:

\begin{defn}
Let $\mathcal{F}_0$ and $\mathcal{F}_1$ be two hypertaut $C^1$-foliations on $M$.
\begin{itemize}
    \item $\mathcal{F}_0$ and $\mathcal{F}_1$ are \textbf{$C^0$-homotopic} if there exists a $1$-parameter family $(\mathcal{F}_t)_{t \in [0,1]}$ of hypertaut $C^1$-foliations such that the map $t \in [0,1] \mapsto T \mathcal{F}_t$ is continuous.
    \item $\mathcal{F}_0$ and $\mathcal{F}_1$ are \textbf{$C^0$-conjugate} if there exists a foliated homeomorphism $h : (M, \mathcal{F}_0) \rightarrow (M, \mathcal{F}_1)$ sending the coorientation of $\mathcal{F}_0$ to that of $\mathcal{F}_1$.
    \item $\mathcal{F}_0$ and $\mathcal{F}_1$ are \textbf{$C^0$-deformation equivalent} if they are related by a sequence of $C^0$-homotopies and $C^0$-conjugations.
\end{itemize}
\end{defn}

We now prove a slightly more general version of Theorem~\ref{thmintro:liouv} from the Introduction:

\begin{thm} \label{thm:uniqhyp}
If $\mathcal{F}_0$ and $\mathcal{F}_1$ are $C^0$-deformation equivalent hypertaut $C^1$-foliations, then their Liouville thickenings are exact symplectomorphic after completion.
\end{thm}

\begin{proof}
We first assume that $\mathcal{F}_0$ and $\mathcal{F}_1$ are $C^0$-homotopic. By compactness, it suffices to show that if $\mathcal{F}_1$ is sufficiently $C^0$-close to $\mathcal{F}_0$ (in the sense of plane fields), then their Liouville thickenings are homotopic. This will essentially follow from the strategy of the proof of Proposition~\ref{prop:welldef}.

Let $\beta$ be a smooth $1$-form such that $d\beta_{\vert T\mathcal{F}_i} >0$ for $i \in \{0,1\}$. Furthermore, let $\mathcal{I} = \mathcal{I}_\beta$ be a line field as in the proof of Proposition~\ref{prop:welldef} and let $\widetilde{\alpha}$ be a $1$-form satisfying $\widetilde{\alpha} \wedge d\beta > 0$. Denote by $\mathcal{V}_i$ a neighborhood of $T\mathcal{F}_i$ as in Theorem~\ref{thmintrobeta:uniq} for the line field $\mathcal{I}$, for $i \in \{0,1\}$, and set $\mathcal{V} \coloneqq \mathcal{V}_0 \cap \mathcal{V}_1$. We can further assume that $T \mathcal{F}_1$ lies in $\mathcal{V}_0$, so that $\mathcal{V} \neq \varnothing$. We then consider contact structures $\xi_\pm \in \mathcal{V}$ approximating $\mathcal{F}_1$. Applying Construction~\ref{construction} to $\mathcal{F}_1$ for $\beta$, $\widetilde{\alpha}$, and $\xi_\pm$ yields a Liouville thickening which is also homotopic to a Liouville thickening of $\mathcal{F}_0$ by (the proof of) Proposition~\ref{prop:welldef}, as desired.

We now assume that $\mathcal{F}_0$ and $\mathcal{F}_1$ are $C^0$-conjugate, via a homeomorphism $h : M \rightarrow M$. Let $\beta_1$ be a $1$-form satisfying $d{\beta_1}_{\vert T \mathcal{F}_1} > 0$ and choose a Vogel neighborhood $\mathcal{V}_1$ for $\mathcal{F}_1$ as in Proposition~\ref{prop:welldef}. Using Theorem~\ref{thmintrobeta:approx}, we can find a smoothing $\widetilde{h}$ of $h$ and a smooth $1$-form $\widetilde{\alpha}_0$ approximating a $1$-form $\alpha_0$ with $\ker \alpha_0 = T \mathcal{F}_0$ such that the following conditions are satisfied:
\begin{align*}
    \widetilde{h}_* \alpha_0 \wedge d\beta_1 >0, \qquad
    \widetilde{h}_* \widetilde{\alpha}_0 \wedge d\beta_1 >0, \qquad
    \widetilde{h}_* \big(T\mathcal{F}_0\big) \in \mathcal{V}_1.
\end{align*}
We can now run Construction~\ref{construction} for $\mathcal{F}_1$ using $\beta_1$, $\widetilde{\alpha}_1 \coloneqq \widetilde{h}_* \widetilde{\alpha}_0$ and contact structures in $\mathcal{V}_1$ obtained by pushing forward contact structures $\xi^0_\pm$ approximating $T \mathcal{F}_0$ along $\widetilde{h}$. Moreover, the $1$-form $\beta_0 \coloneqq \widetilde{h}^*\beta_1$ satisfies $d{\beta_0}_{\vert T\mathcal{F}_0} >0$, and we can run Construction~\ref{construction} for $\mathcal{F}_0$ using $\beta_0$, $\widetilde{\alpha}_0$ and $\xi^0_\pm$. Therefore, we obtain (pre-)Liouville thickenings $\lambda_0$ and $\lambda_1$ of $\mathcal{F}_0$ and $\mathcal{F}_1$, respectively, which satisfy $\lambda_0 = \big(\mathrm{id} \times\widetilde{h}\big)^*\lambda_1$. Finally, Proposition~\ref{prop:welldef} implies that the Liouville thickenings of $\mathcal{F}_0$ and $\mathcal{F}_1$ are exact symplectomorphic after completion.
\end{proof}

\begin{rem}
    The proof shows that Liouville thickenings of $C^0$-homotopic hypertaut foliations are homotopic, and Liouville thickenings of $C^0$-conjugate hypertaut foliations are deformation equivalent via an equivalence (topologically) isotopic to $\mathrm{id} \times h$, where $h$ is the conjugation.
\end{rem}

Therefore, every $C^0$-deformation equivalence class of hypertaut foliations on $M$ has an associated $A_\infty$-category, well-defined up to quasi-isomorphism, obtained as the wrapped Fukaya category of the Liouville thickening $\lambda_\mathcal{F}$. The special case of Anosov foliations was studied in~\cite{CLMM}.

\medskip

The proof of Theorem~\ref{thmintro:poscont} follows from the same arguments \emph{mutatis mutandis} and is left to the reader.

        \subsection{Liouville pairs}

We say that a pair of contact forms $(\alpha_-, \alpha_+)$ on $M$ is a (linear) \textbf{Liouville pair} if the $1$-form
$$\lambda \coloneqq (1-\tau) \alpha_- + (\tau+1) \alpha_+$$
defines a Liouville form on $[-1,1]_\tau \times M$, i.e., if $d\lambda$ is symplectic. These structures already appear in~\cite{Mit95} and~\cite{MNW13} and were extensively studied in~\cite{Mas24} (with a slightly different definition).

Jonathan Zung implicitly showed in~\cite{Z21} that every hypertaut $C^2$-foliation on a closed $3$-manifold induces such a Liouville pair. More precisely, he proved:

\begin{prop}[\cite{Z21}] \label{prop:Zung}
If $\mathcal{F}$ is a hypertaut $C^2$-foliation, then there exist $1$-forms $\alpha$ and $\beta$ of class $C^1$ such that 
\begin{align}
\ker \alpha = T \mathcal{F}, \qquad \alpha \wedge d\beta > 0, \qquad \beta \wedge d\alpha \geq 0. \label{eq:Zung}
\end{align}
\end{prop}

The first inequality simply means that $d\beta$ is a dominating $2$-form for $\mathcal{F}$.

\begin{cor} \label{cor:Zung}
If $\mathcal{F}$ is a hypertaut $C^2$ foliation, then there exists a Liouville pair $(\alpha_-, \alpha_+)$ on $M$ such that the contact structures $\xi_\pm = \ker \alpha_\pm$ are $C^0$-close to $T\mathcal{F}$.
\end{cor}

\begin{proof} 
For $\delta > 0$, we define
\begin{align*} 
    \alpha_\pm &\coloneqq \delta \beta \pm \alpha, \\
    \lambda &\coloneqq (1-\tau) \alpha_- + (1+\tau) \alpha_+ \\
    &= 2\big( \delta \beta + \tau \alpha \big),
\end{align*}
where $\alpha$ and $\beta$ are as in Proposition~\ref{prop:Zung}. Following~\cite{ET}, we write $$\langle \alpha, \beta \rangle \coloneqq \alpha \wedge d \beta + \beta \wedge d \alpha.$$ One computes 
\begin{align*}
    \alpha_+ \wedge d\alpha_+ &= \delta \langle \alpha, \beta \rangle  + O(\delta^2), \\
    \alpha_- \wedge d\alpha_- &= - \delta \langle \alpha, \beta \rangle + O(\delta^2), \\
    d\lambda \wedge d\lambda &= 8\delta d\tau \wedge \alpha \wedge d\beta > 0.
\end{align*}
Hence, for $\delta$ small enough, $\alpha_\pm$ are contact forms with opposite orientations defining contact structures $C^0$-close to $T \mathcal{F}$, and $\lambda$ is a Liouville form, so $(\alpha_-, \alpha_+)$ is a Liouville pair. The $1$-forms $\alpha_\pm$ might only be $C^1$, but they can easily be smoothed to yield a smooth Liouville pair.
\end{proof}

Notice that the conditions in~\eqref{eq:Zung} are convex in both $\alpha$ and $\beta$, but might \emph{fail} to be convex in $(\alpha, \beta)$. Therefore, it is not immediately clear that two such pairs induce equivalent Liouville structures.

\begin{rem}
If the stronger condition $$\beta \wedge d\alpha > 0$$ is satisfied, then $(-\alpha_-, \alpha_+)$ is also a Liouville pair (for $\delta$ small enough). In that case, $(\alpha_-, \alpha_+)$ is an \emph{Anosov Liouville pair}; see~\cite{Hoz24, Mas25a}. This implies that the contact structures $\xi_\pm = \ker \alpha_\pm$ are transverse and their intersection is spanned by an Anosov flow. Moreover, $\mathcal{F}$ is the weak-unstable foliation of this flow. The case of Anosov flows and foliations will be studied in the next section.
\end{rem}

The proof of Corollary~\ref{cor:Zung} shows that the relevant conditions that $\alpha$ and $\beta$ have to satisfy to obtain a Liouville pair are 
\begin{align} \label{eq:condliouv}
\alpha \wedge d\beta > 0, \qquad \langle \alpha, \beta \rangle > 0.
\end{align}
We consider the space $\mathcal{Z}_\mathcal{F}$ of pairs of $1$-forms $(\alpha, \beta)$ of class $C^1$ with $\ker \alpha = T \mathcal{F}$ and satisfying~\eqref{eq:condliouv}. For every $(\alpha, \beta) \in \mathcal{Z}_\mathcal{F}$, there exists $\overline{\delta} = \overline{\delta}(\alpha, \beta) > 0$ such that for every $0 < \delta < \overline{\delta}$, $\big(\delta \beta - \alpha, \delta \beta + \alpha \big)$ is a Liouville pair. Notice that the $\delta$ factor is in front of $\beta$ whereas the $\epsilon$ factor is in front of $\alpha$ in Construction~\ref{construction}. Its associated Liouville structure does not depend on the choice of $\delta$ up to Liouville homotopy, so it defines a homotopy class of Liouville structures $[\lambda_{\alpha, \beta}] \in \mathscr{L}(V)$.

\begin{lem} \label{lem:zungpair}
For every $(\alpha, \beta) \in \mathcal{Z}_\mathcal{F}$, $\lambda_{\alpha, \beta}$ is Liouville homotopic to a Liouville thickening of $\mathcal{F}$.
\end{lem}

\begin{proof}
    Our task is to show that $\lambda_{\alpha, \beta}$ is Liouville homotopic to a Liouville structure coming from Construction~\ref{construction}. Let $t, \epsilon \in (0,1]$ and consider 
    \begin{align*}
    \lambda_\epsilon &\coloneqq 2\big( \delta \beta + \epsilon \tau \alpha \big), \\
    \xi^t_\pm &\coloneqq \ker\big( t \delta \beta \pm \alpha \big).
    \end{align*}
    Using~\eqref{eq:condliouv}, one checks that the following hold for any $\delta$ small enough (independent of $t$ and $\epsilon$), 
    \begin{itemize}
        \item For all $t, \epsilon \in (0, 1]$, $\big(\lambda_\epsilon, \xi^t_\pm\big)$ is a pre-Liouville structure on $V$,
        \item The contact structures $\xi^t_\pm$ converge uniformly in the $C^0$ sense to $T\mathcal{F}$ as $t \rightarrow 0$. 
    \end{itemize} 
Therefore, $\lambda_{\alpha, \beta}$ is (pre-)Liouville homotopic to a Liouville thickening obtained from the $1$-forms $\delta \beta$, $\alpha$ and contact approximations $\xi^t_\pm$ for $t$ small enough\footnote{Technically speaking, $\alpha$ is only $C^1$ rather than smooth, but this does not impact the argument since we can also consider (pre-)Liouville structures which are only $C^1$ and smooth them afterwards.}. 
One then concludes by applying Proposition~\ref{propintro:deffill}.
\end{proof}

As a consequence, the Liouville structures obtained in Corollary~\ref{cor:Zung} are independent of all choices up to Liouville homotopy. Therefore, any Liouville thickening of $\mathcal{F}$ is Liouville homotopic to one coming from a Liouville pair. The latter enjoy nice properties; for instance, their Liouville vector fields are easy to compute and their skeleta are codimension-$1$ submanifolds diffeomorphic to $M$; see~\cite{Mas24}.

\section{Consequences for Anosov flows}

%
        \subsection{Anosov flows}

Recall that a nonsingular flow $\Phi = (\varphi^t)_t$ generated by a smooth vector field $X$ is \textbf{Anosov} if the tangent bundle of $M$ has a continuous splitting
 $$TM =   E^{ss} \oplus \langle X \rangle \oplus E^{uu}$$
that is $\Phi$-invariant, and so that there exist a Riemannian metric and constants $C, a >0$ for which the inequalities
\begin{align*}
    \Vert d\varphi_t(v^{s})\Vert \leq C e^{-a t} \Vert v^{s} \Vert, \qquad
    \Vert d\varphi_t(v^{u})\Vert \geq C^{-1} e^{a t} \Vert v^{u}\Vert
\end{align*}
hold for all $ t \ge0$ and all $v^{u}\in E^{uu}, v^{s} \in E^{ss}$. 

The subbundles $E^{uu}, E^{ss}$ are called the \emph{strong-unstable} and \emph{strong-stable} directions of the flow, respectively. It is a classical fact due to Anosov that these distributions are uniquely integrable and integrate to foliations whose leaves consist of points that are asymptotic under the flow in forward, respectively backward, time. In the case where the manifold is $3$-dimensional and closed, these distributions are line fields.

In this case one also obtains $2$-dimensional foliations $\mathcal{F}^{ws}$ and $\mathcal{F}^{wu}$ tangent to the integrable plane fields 
$$ E^{ws}= E^{ss} \oplus \langle X \rangle, \qquad  E^{wu} = E^{uu} \oplus \langle X \rangle,$$
respectively; these are called the \emph{weak-stable} and \emph{weak-unstable} 
foliations of the flow. It is a special feature of Anosov flows in dimension $3$ that the weak-(un)stable foliations are $C^1$ by~\cite{HPS77}.

Anosov flows also have a Spectral Decomposition~\cite{Sma67}, meaning that the non-wandering set decomposes (uniquely) into finitely many transitive components. Using these structural results we have the following, which will imply that the weak-(un)stable foliations of Anosov flows are hypertaut.

\begin{lem}[Folklore] \label{lem:closedorbit}
    Let $\Phi$ be an Anosov flow on a closed $3$-manifold $M$ and let $\mathcal{F} = \mathcal{F}^{wu},$ be its weak-unstable foliation. Then there exists a finite set $\Gamma$ of closed orbits of $\Phi$ such that for every leaf $L$ of $\mathcal{F}$, the closure $\overline{L}$ contains a closed orbit of $\Phi$ in $\Gamma$.
\end{lem}

\begin{proof}
Let $\Omega \subset M$ denote the non-wandering set of $\Phi$. By the Anosov Closing Lemma~\cite[Theorem 5.3.11]{FH19}, the union of the closed orbits of $\Phi$ is a dense subset of $\Omega$. By compactness of $M$, we can find some $\epsilon > 0$ such that for every $p \in M$, every weak-unstable leaf $L$ of $\mathcal{F}$ which intersects the $\epsilon$-neighborhood $U_\epsilon(p)$ of $p$ intersects the weak-stable leaf passing through $p$. In particular, if $\gamma$ is a periodic orbit of $\Phi$, then every weak-unstable leaf $L$ of $\mathcal{F}$ which intersects the $\epsilon$-neighborhood $U_\epsilon(\gamma)$ of $\gamma$ also intersects the weak-stable leaf passing through $\gamma$. Let $U$ denote the union of the open sets of the form $U_\epsilon(\gamma)$, for $\gamma$ a closed orbit of $\Phi$. Then $\Omega \subset U$, and by compactness of $\Omega$, there exists a finite collection of closed orbits $\Gamma$ such that $\Omega \subset \bigcup_{\gamma \in \Gamma} U_\epsilon(\gamma) \eqqcolon U'$. Since a leaf $L$ of $\mathcal{F}$ is saturated by $\Phi$, it intersects $U'$; in particular, it intersects $U_\epsilon(\gamma)$ for some $\gamma \in \Gamma$, so it intersects the weak-stable leaf of $\gamma$ and its closure $\overline{L}$ contains $\gamma$.
\end{proof}

\begin{prop}\label{prop:Anosovhypertaut}
    The weak foliations $\mathcal{F}^{ws}$ and $\mathcal{F}^{wu}$ of an Anosov flow $\Phi$ on $M$ are hypertaut.
\end{prop}

\begin{proof}
    First recall that the weak foliations are $C^1$ by~\cite{HPS77}. There are no closed leaves as the flow (uniformly) expands area, and every minimal set contains a closed orbit by the previous lemma. Those closed orbits are attracting/repelling curves, since the weak-stable (resp.~unstable) foliation has linearly expanding (resp.~contracting) holonomy along periodic orbits. Therefore, Proposition~\ref{prop:hypertaut} implies that the weak foliations are hypertaut.
\end{proof}

    \subsection{Anosov Liouville structures and proof of Theorem~\ref{thmintro:anosov}}

We recall Mitsumatsu's construction~\cite{Mit95}, later generalized and streamlined by Hozoori~\cite{Hoz24}. See also~\cite{Mas25a, Mas25b}.

For a smooth Anosov flow $\Phi$ generated by a vector field $X$ with \underline{oriented} weak bundles, there exist $C^1$ $1$-forms $\alpha_s$ and $\alpha_u$ satisfying:
\begin{align*}
    \alpha_s(X)=0, && \ker \alpha_s &= E^{wu}, & \mathcal{L}_X \alpha_s &= r_s \, \alpha_s,\\
    \alpha_u(X) = 0, && \ker \alpha_u &= E^{ws}, & \mathcal{L}_X \alpha_u &= r_u \, \alpha_u,
\end{align*}
where $r_s$ and $r_u$ are $C^1$ functions satisfying $r_s < 0 < r_u$, called the \textbf{expansion rates} of $\Phi$ in the stable and unstable directions, respectively. Such a pair $(\alpha_s, \alpha_u)$ is called a \textbf{defining pair} for $\Phi$ in~\cite{Mas25b}. Then, one considers
\begin{align*}
    \alpha_- \coloneqq \alpha_u + \alpha_s,\qquad
    \alpha_+ \coloneqq \alpha_u - \alpha_s,
\end{align*}
and checks that $(\alpha_-, \alpha_+)$ is a Liouville pair. While these forms are only of class $C^1$, they can easily be smoothed while still containing $X$ in their kernels. The resulting Liouville structure
$$\lambda \coloneqq (1-\tau) \alpha_- + (1+\tau) \alpha_+$$
on $[-1,1]_\tau \times M$ is called a (linear) \textbf{Anosov Liouville structure} supporting $\Phi$. It was shown in~\cite{Mas25a, Mas25b} that it is well-defined and its Liouville homotopy class does not depend on the auxiliary choices of defining pairs and smoothings.

We briefly explain how this fits into the framework of Construction~\ref{construction}:

\begin{lem} \label{lem:ALthick}
    If $\Phi$ is a smooth oriented Anosov flow, then any Anosov Liouville structure supporting it is a Liouville thickening of its weak-unstable foliation.
\end{lem}

\begin{proof}
    Let $(\alpha_s, \alpha_u)$ be a defining pair for $\Phi$. Then for every $\delta > 0$,
    \begin{align*}
        \alpha^\delta_- \coloneqq \delta \alpha_u + \alpha_s,\qquad
        \alpha^\delta_+ \coloneqq \delta \alpha_u - \alpha_s,
    \end{align*}
    also define a Liouville pair $\big(\alpha^\delta_-, \alpha^\delta_+\big)$, whose underlying contact structures converge to $E^{wu}$ (with appropriate orientation) as $\delta \rightarrow 0$. Moreover, writing
    $$\alpha \coloneqq -\alpha_s, \qquad \beta \coloneqq \alpha_u,$$
    one easily checks that these $C^1$ $1$-forms satisfy~\eqref{eq:condliouv},\footnote{In particular, $d\beta_{\vert T\mathcal{F}^{wu}} > 0$, which is an alternative proof of the fact that $\mathcal{F}^{wu}$ is hypertaut.} and the proof of Lemma~\ref{lem:zungpair} implies that the Liouville structure induced by the Liouville pair $\big(\alpha^\delta_-, \alpha^\delta_+\big)$ is Liouville homotopic to a Liouville thickening of $\mathcal{F}^{wu}$ (with appropriate orientation). The structures under consideration are not necessarily smooth but they are $C^1$, which is enough for our arguments to go through since they can be smoothed in a unique way up to homotopy.
\end{proof}

\medskip

\begin{proof}[Proof of Theorem~\ref{thmintro:anosov}] Let $\Phi_0$ and $\Phi_1$ be two smooth oriented Anosov flows which are orbit equivalent, via an orbit equivalence $h : M \rightarrow M$. Then $h$ sends the weak-stable (resp.~weak-unstable) foliation of $\Phi_0$ to that of $\Phi_1$, and we assume that it preserves their (co)orientations. Therefore, $h$ is a $C^0$-conjugacy between $\mathcal{F}^{wu}_0$ and $\mathcal{F}^{wu}_1$. These foliations are hypertaut by Proposition~\ref{prop:Anosovhypertaut}, so $h$ induces an exact symplectomorphism between the completions of their Liouville thickenings by Theorem~\ref{thmintro:liouv}. Finally, Lemma~\ref{lem:ALthick} implies that these Liouville thickenings are Liouville homotopic to the Anosov Liouville structures associated with $\Phi_0$ and $\Phi_1$, respectively.
\end{proof}

    \subsection{Bicontact structures and proof of Theorem~\ref{thmintro:anosovbicontact}}

Mitsumatsu~\cite{Mit95} and independently Eliashberg--Thurston~\cite{ET} observed that an Anosov flow gives rise to a pair of transverse contact structures $(\xi_-,\xi_+)$ that determine opposite orientations and are tangent to the flow. Such a pair is called a \textbf{bicontact structure}. These contact structures are necessarily nowhere tangent to the weak-(un)stable bundles of the flow. See Figure~\ref{fig:bicontact}. One can rephrase this condition in terms of \emph{projectively Anosov flows}, also called \emph{ conformally Anosov flows}, which generalize Anosov flows. In the dynamics literature, this condition is called a \emph{dominated splitting} and goes back to Ma\~{n}é, Liao and Pliss, although the connection to contact geometry came much later. 

These connections lead, on the one hand, to interesting ways of studying dynamics through contact geometry and, on the other hand, to investigating bicontact structures in their own right. However, our results only apply to Anosov flows since projectively Anosov flows typically do not have invariant foliations and are somewhat more flexible than their Anosov counterparts.

\setlength{\unitlength}{1mm}
\begin{figure}[H]
    \begin{center}
        \begin{picture}(60, 60)(0,0)
        \put(0,0){\includegraphics[width=60mm]{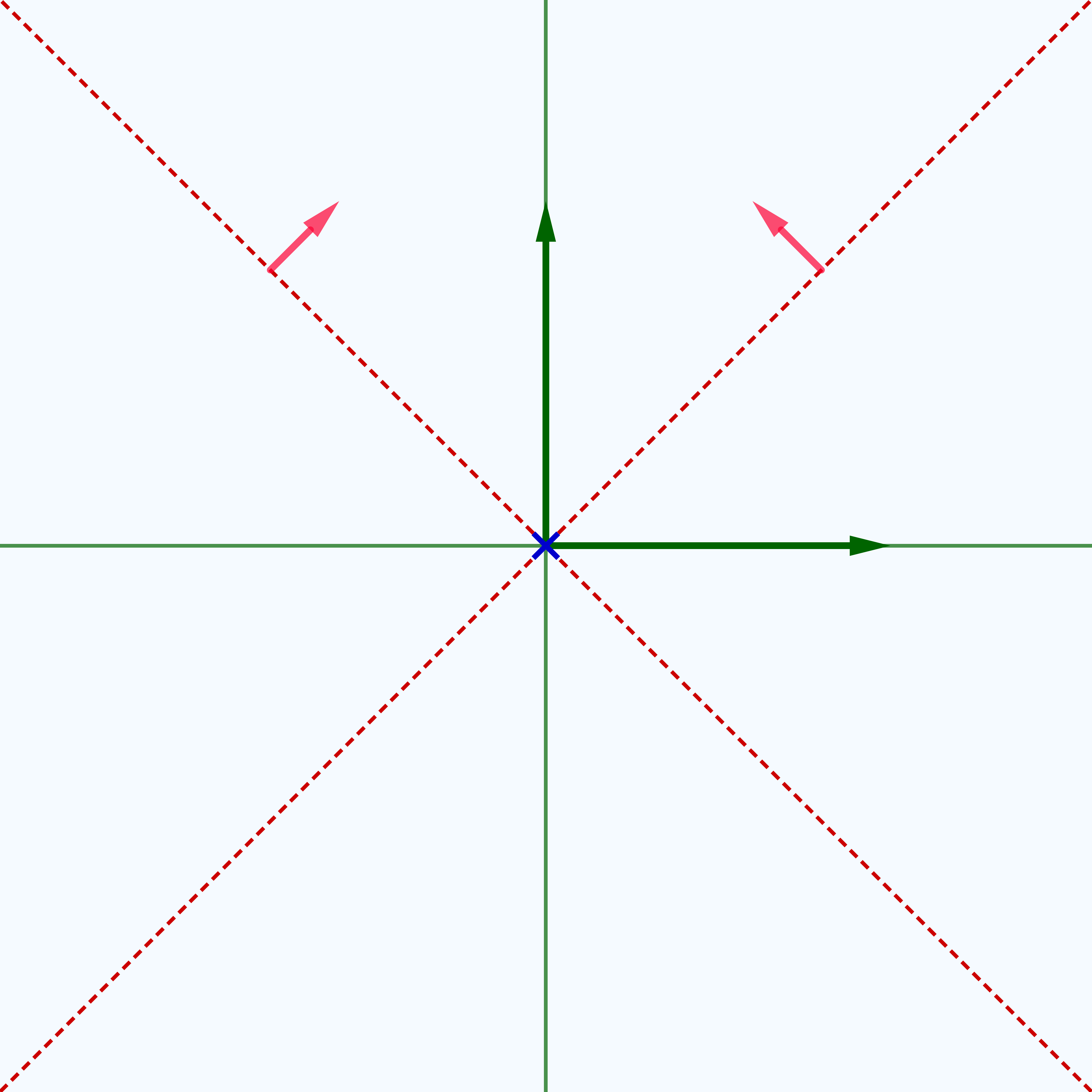}}
        \color{blue1}
        \put(26.5,24){$X$}
        \color{green1}
        \put(40,27){$e_s$}
        \put(52,26){$E^s$}
        \put(31,53){$E^u$}
        \put(31,41){$e_u$}
        \color{red1}
        \put(5,47){$\xi_-$}
        \put(52,47){$\xi_+$}
        \end{picture}
        \caption{Anosov/dominated splitting and supporting bicontact structure.}
        \label{fig:bicontact}
    \end{center}
\end{figure}

We now consider an Anosov flow $\Phi$ generated by a smooth vector field $X$ and with orientable (un)stable foliations, and we prove Theorem~\ref{thmintro:anosovbicontact} from the Introduction. We start with an elementary but somewhat technical lemma.

\begin{lem} \label{lem:flow}
    Let $\eta$ be a continuous plane field transverse to the strong-stable bundle $E^{ss}$ of $\Phi=(\varphi^t)_t$. For $t \geq 0$, we write $\eta_t \coloneqq (\varphi^t)_* \eta$. Then the following hold:
    \begin{enumerate}
        \item $\displaystyle \lim_{t \rightarrow + \infty} \eta_t = E^{wu}$ in the $C^0$ topology.
        \item There is a neighborhood $\mathcal{U}$ of $E^{ss}$ in the space of continuous line fields, and a neighborhood $\mathcal{V}$ of $E^{wu}$ in the space of continuous plane fields, such that for all $\ell \in \mathcal{U}$, $\eta \in \mathcal{V}$, and $t \geq 0$, $\ell$ is transverse to $\eta_t$.
    \end{enumerate}
\end{lem}

\begin{proof}
Let $(\alpha_s, \alpha_u)$ be a defining pair for the flow $\Phi$, with corresponding expansion rates $r_s$ and $r_u$. We further consider a (continuous) $1$-form $\vartheta$ such that $\vartheta(X) = 1$ and $\ker \vartheta = E^{ss} \oplus E^{uu}$. Then it is easy to see that
\begin{align*}
    (\varphi^t_X)^*\alpha_s &= \exp\left(\int_0^t r_s \circ \varphi^\tau_X \, d\tau \right) \alpha_s = R_s^t \alpha_s, \\
    (\varphi^t_X)^*\alpha_u &= \exp\left(\int_0^t r_u \circ \varphi^\tau_X \, d\tau \right) \alpha_u = R_u^t \alpha_u, \\
    (\varphi^t_X)^*\vartheta &= \vartheta.
\end{align*}

    Let $\eta$ be a continuous plane field transverse to $E^{ss}$. There exist continuous functions $f, g : M \rightarrow \R$ such that $\eta$ is the kernel of the $1$-form $\alpha = \alpha_s + f \alpha_u + g \vartheta$. Note that $\eta_t = \ker \big(\varphi^{-t}_X\big)^* \alpha$, and
    \begin{align*}
    \big(\varphi^{-t}_X\big)^* \alpha &= R^{-t}_s \alpha_s  + R^{-t}_u (f\circ\varphi^{-t}_X) \alpha_u + (g\circ \varphi^{-t}_X) \vartheta \\
    &= R^{-t}_s\left(\alpha_s + \frac{R^{-t}_u}{R^{-t}_s} (f\circ\varphi^{-t}_X) \alpha_u + \frac{1}{R^{-t}_s} (g\circ \varphi^{-t}_X) \vartheta\right),
    \end{align*}
    and since $f\circ \varphi^{-t}_X$ and $g\circ \varphi^{-t}_X$ are uniformly bounded in $t$ and 
    $$\lim_{t \rightarrow + \infty} R^{-t}_u = \lim_{t \rightarrow + \infty} \frac{1}{R^{-t}_s} = 0,$$
    we obtain $\lim_{t \rightarrow +\infty}\eta_t = \ker \alpha_s = E^{wu}$. This proves the first item.

    Let $a>0$ and let $\mathcal{U}$ denote the space of line fields which stay at distance greater than $a$ from $E^{wu}$. For $a$ sufficiently small, $\mathcal{U}$ is an open neighborhood of $E^{ss}$ which only contains line fields transverse to $E^{wu}$. There exists $\epsilon > 0$ such that if $f, g : M \rightarrow \R$ are continuous functions with $\vert f \vert, \vert g \vert < \epsilon$, then the $1$-form $\alpha_s + f \alpha_u + g \vartheta$ is nowhere vanishing on each line field $\ell \in \mathcal{U}$. The kernels of all such $1$-forms define a neighborhood $\mathcal{V}$ of $E^{wu}$, and all the plane fields in $\mathcal{V}$ are transverse to all the line fields in $\mathcal{U}$. The flow of $X$ naturally acts on the space of continuous plane fields, and we claim that $\mathcal{V}$ is preserved by the flow of $X$ in positive time, which suffices to prove the second item.

    If $\eta \in \mathcal{V}$ is defined by $\alpha = \alpha_s + f\alpha_u + g \vartheta$, with $\vert f \vert, \vert g \vert < \epsilon$, then $\eta_t$ is defined by
    \begin{align*}
    \alpha_t &= \alpha_s + \frac{R^{-t}_u}{R^{-t}_s} (f\circ\varphi^{-t}_X) \alpha_u + \frac{1}{R^{-t}_s} (g\circ \varphi^{-t}_X) \vartheta \\
    &= \alpha_s + f_t \alpha_u + g_t \vartheta.
    \end{align*}
    Since $r_s < 0 < r_u$, there exists $\delta > 0$ such that $\delta < r_u$ and $\delta < -r_s$. Then, it is easy to see that
    \begin{align*}
        R^{-t}_u &\leq e^{-\delta t}, \qquad R^{-t}_s \geq e^{\delta t},
    \end{align*}
    hence for $t \geq 0$, $\vert f_t \vert \leq \vert f \circ \varphi^{-t}_X \vert < \epsilon$ and $\vert g_t \vert \leq \vert g \circ \varphi^{-t}_X\vert < \epsilon$, as desired.
\end{proof}

\medskip

\begin{proof}[Proof of Theorem~\ref{thmintro:anosovbicontact}]
    For $i \in \{0,1\}$ and $\heartsuit \in \{ss, uu, ws, wu\}$, we denote by $E^{\heartsuit}_i$ the strong-stable, strong-unstable, weak-stable, and weak-unstable bundles of $\Phi_i$, respectively.
    
    Let $\mathcal{U}_1$ be a neighborhood of $E^{ss}_1$ and let $\mathcal{V}_1$ be a neighborhood of $E_1^{wu}$ as in Lemma~\ref{lem:flow}. Similarly, let $\mathcal{U}'_1$ be a neighborhood of $E^{uu}_1$ and $\mathcal{V}'_1$ be a neighborhood of $E_1^{ws}$ such that Lemma~\ref{lem:flow} applies for $t \leq 0$ instead of $t \geq 0$.

    Let $\widetilde{h} : M \rightarrow M$ denote a smoothing of $h$, topologically isotopic to $h$, satisfying
    \begin{itemize}
        \item $\widetilde{h}_* (E_0^{wu}) \in \mathcal{V}_1$ and $\widetilde{h}_* (E_0^{ws}) \in \mathcal{V}'_1$,
        \item There exists a $C^1$ line field in $\mathcal{U}_1$ tangent to $\widetilde{h}_* (E_0^{ws})$.
    \end{itemize}
    This can be achieved by applying Theorem~\ref{thmintrobeta:bifolapprox} for a sufficiently small $\epsilon > 0$.

    Let $\big(\alpha^0_s, \alpha^0_u\big)$ be a defining pair for $\Phi_0$. We consider a bicontact structure supporting $\Phi_0$ of the form
    \begin{align*}
        \xi^0_- = \ker \left(A \alpha^0_u + \alpha^0_s \right), \qquad
        \xi^0_+ = \ker \left(\alpha^0_u - A \alpha^0_s \right),
    \end{align*}
    for a large $A > 0$, so that $\xi^0_-$ is $C^0$-close to $E_0^{ws}$, and $\xi^0_+$ is $C^0$-close to $E_0^{wu}$. For $A$ sufficiently large, we can further assume that
    \begin{itemize}
        \item $\widetilde{\xi}^1_- \coloneqq \widetilde{h}_* (\xi^0_-) \in \mathcal{V}_1$ and $\widetilde{\xi}^1_+ \coloneqq \widetilde{h}_* (\xi^0_+) \in \mathcal{V}'_1$,
        \item There exists a smooth line field $\ell_-$ in $\mathcal{U}_1$ tangent to $\widetilde{\xi}^1_-$.
    \end{itemize}
    Then by Theorem~\ref{thmintrobeta:uniq}, there exists a Vogel neighborhood $\mathcal{N}_1$ of $E_1^{wu}$ such that any two positive contact structures in $\mathcal{N}_1$ are homotopic through contact structures transverse to $\ell_-$ (and in particular transverse to $\widetilde{\xi}^1_-$). However, $\widetilde{\xi}^1_+$ might not be contained in $\mathcal{N}_1$ yet. We can remedy this by applying Lemma~\ref{lem:flow} and flowing $\widetilde{\xi}^1_+$ along $\Phi_1 = (\varphi^t_1)_t$ for a large time $T > 0$ until $\widehat{\xi}^1_+ \coloneqq (\varphi^T_1)_* [\widetilde{\xi}^1_+] \in \mathcal{N}_1$; this induces a homotopy of contact structures transverse to $\widetilde{\xi}^1_-$ from $\widetilde{\xi}^1_+$ to $\widehat{\xi}^1_+$. We then apply Theorem~\ref{thmintrobeta:uniq} to find a homotopy of contact structures transverse to $\widetilde{\xi}^1_-$ from $\widehat{\xi}^1_+$ to a positive supporting contact structure $\xi^1_+$ for $\Phi_1$. We can further arrange that $\xi^1_+$ is so close to $E_1^{wu}$ that it contains a line field $\ell_+ \in \mathcal{U}'_1$. We then apply the same procedure to $\widetilde{\xi}^1_-$ to obtain a homotopy of contact structures transverse to $\xi^1_+$ from $\widetilde{\xi}^1_-$ to a negative supporting contact structure $\xi^1_-$ for $\Phi_1$. In summary, we constructed a homotopy of \emph{bicontact structures} from $\big(\widetilde{h}_* (\xi^0_-), \widetilde{h}_* (\xi^0_+)\big)$ to a bicontact structure $\big(\xi^1_-, \xi^1_+\big)$ supporting $\Phi_1$. This finishes the proof since the space of bicontact structures supporting a given Anosov flow is path-connected (and even contractible); see~\cite{Hoz24, Mas25a}.
\end{proof}

\appendix

\newpage

    \section{Technical smoothing lemmas} \label{appendixA}

In this appendix, we collect some technical lemmas on smoothing (families of) increasing functions and topological embeddings of the $2$-disk in the plane, which are extensively used in Section~\ref{sec:smooth}. These are probably well-known to the experts and the proofs are quite standard, but we were not able to find precise statements in the literature.

        \subsection{Smoothing increasing functions}

\begin{lem} \label{lem:smoothincrease1}
Let $v : [0,1]_z \rightarrow \R$ be a continuous, strictly increasing function. We fix $\delta \in (0,1/4)$ and $\epsilon > 0$.
\begin{enumerate}
    \item There exists $\widetilde{v} \in C^1([0,1], \R)$ such that $\widetilde{v}(0) = v(0)$, $\widetilde{v}(1) = v(1)$, and 
    $$ \partial_z \widetilde{v} > 0, \qquad \big\vert \widetilde{v} - v\big\vert_{C^0} < \epsilon.$$
    \item Let $\widetilde{v}_\partial : [0, 2\delta) \cup (1-2\delta, 1] \rightarrow \R$ be a $C^1$ function such that 
    $$ \partial_z \widetilde{v}_\partial > 0, \qquad \big\vert \widetilde{v}_\partial - v\big\vert_{C^0} < \epsilon, \qquad \widetilde{v}_\partial(\delta) < \widetilde{v}_\partial(1-\delta).$$
    Then there exists $\widetilde{v} \in C^1([0,1],\R)$ satisfying
    $$\forall z \in [0, \delta) \cup (1-\delta, 1], \quad \widetilde{v}(z) = \widetilde{v}_\partial(z),$$
    and $$\partial_z \widetilde{v} > 0, \qquad \big\vert \widetilde{v} - v\big\vert_{C^0} < 2\epsilon.$$
\end{enumerate}
\end{lem}

\begin{proof}
    For the first item, it suffices to approximate $v$ with a piecewise linear map and then smooth it. 
    
    The second item can also be proved using the previous approach. Another method that generalizes well to parametric versions is as follows: one can first use the previous method to construct a smoothing $\widetilde{v}$ satisfying $\partial_z \widetilde{v} > 0$, and such that 
    $$
    \begin{aligned}
    \forall z \in [\delta, 2\delta] & , \quad & v_\partial &\leq \widetilde{v}, \\
    \forall z \in [1 - 2\delta, 1 - \delta] & , \quad & \widetilde{v} &\leq v_\partial.
    \end{aligned}$$
    Then one can connect $v_\partial$ and $\widetilde{v}$ on $[\delta, 2\delta]$ and $[1-2\delta, 1-\delta]$ using a monotone cutoff function. Details are left to the reader.
\end{proof}

For $1 \leq n \leq 3$ and $0 < \delta < 1/4$, we write $N^n_\delta \coloneqq [0,1]^n \setminus [\delta, 1-\delta]^n$.

\begin{lem} \label{lem:smoothincrease2}
Let $v : [0,1]^3 \rightarrow \R$ be a continuous function such that for every $(x,y) \in [0,1]^2$, $v(x, y, \, \cdot \, ) : [0,1] \rightarrow \R$ is strictly increasing. We fix $\delta \in (0,1/4)$ and $\epsilon > 0$.
\begin{enumerate}
    \item  There exists $\widetilde{v} \in C^1([0,1]^3, \R)$ such that 
    $$ \partial_z \widetilde{v} > 0, \qquad \big\vert \widetilde{v} - v\big\vert_{C^0} < \epsilon.$$
    
    \item Let $\widetilde{v}_\partial : N^1_{2\delta} \times [0,1]^2 \rightarrow \R$ be a $C^1$ function such that 
    $$ \partial_z \widetilde{v}_\partial > 0, \qquad \big\vert \widetilde{v}_\partial - v\big\vert_{C^0} < \epsilon.$$
    Then there exists $\widetilde{v} \in C^1([0,1]^3,\R)$ satisfying
    $$\forall (x,y,z) \in N^1_\delta \times [0,1]^2, \quad \widetilde{v}(x,y,z) = \widetilde{v}_\partial(x,y,z),$$
    and $$\partial_z \widetilde{v} > 0, \qquad \big\vert \widetilde{v} - v\big\vert_{C^0} < 2\epsilon.$$

    \item Let $\widetilde{v}_\partial : N^2_{2\delta} \times [0,1] \rightarrow \R$ be a $C^1$ function such that 
    $$ \partial_z \widetilde{v}_\partial > 0, \qquad \big\vert \widetilde{v}_\partial - v\big\vert_{C^0} < \epsilon.$$
    Then there exists $\widetilde{v} \in C^1([0,1]^3,\R)$ satisfying
    $$\forall (x,y,z) \in N^2_\delta \times [0,1], \quad \widetilde{v}(x,y,z) = \widetilde{v}_\partial(x,y,z),$$
    and $$\partial_z \widetilde{v} > 0, \qquad \big\vert \widetilde{v} - v\big\vert_{C^0} < 2\epsilon.$$

    \item Let $\widetilde{v}_\partial : N^3_{2\delta} \rightarrow \R$ be a $C^1$ function such that 
    $$ \partial_z \widetilde{v}_\partial > 0, \qquad \big\vert \widetilde{v}_\partial - v\big\vert_{C^0} < \epsilon, \qquad \widetilde{v}_\partial(\, \cdot\, , \, \cdot \, , \delta) < \widetilde{v}_\partial(\, \cdot\, , \, \cdot \, , 1-\delta).$$
    Then there exists $\widetilde{v} \in C^1([0,1]^3,\R)$ satisfying
    $$\forall (x,y,z) \in N^3_\delta, \quad \widetilde{v}(x,y,z) = \widetilde{v}_\partial(x,y,z),$$
    and $$\partial_z \widetilde{v} > 0, \qquad \big\vert \widetilde{v} - v\big\vert_{C^0} < 2\epsilon.$$
\end{enumerate}
\end{lem}

\begin{proof}
For the first item, it suffices to consider a sufficiently fine grid on $[0,1]^2$, apply the first item of Lemma~\ref{lem:smoothincrease1} at each vertex of this grid, and connect those smoothings via a partition of unity supported near those vertices. Here, we are using that the condition of having a positive derivative is convex.

For the second and third items, one can combine the previous strategy with the approach outlined in the proof of the second item of Lemma~\ref{lem:smoothincrease1}.

Finally, for the fourth item, one can apply the third item to obtain a smoothing $\widetilde{v}$ coinciding with $v_\partial$ on $N^2_{\delta'} \times [0,1]$ for a slightly larger $\delta' > \delta$, and interpolate between $v_\partial$ and $\widetilde{v}$ using a cutoff function in the variables $x$ and $y$ which vanishes on $N^2_\delta$ and is equal to $1$ on $[\delta', 1-\delta']^2$.
\end{proof}

        \subsection{Smoothing embeddings of the \texorpdfstring{$2$}{2}-disk}

We now consider a $2$-dimensional version of the previous lemmas.

\begin{lem} \label{lem:smoothemb1}
    Let $u : [0,1]^3 \rightarrow \R^2$ be a continuous map such that for every $z \in [0,1]$, the map $(x,y) \mapsto u(x,y,z)$ is an embedding (i.e., a homeomorphism onto its image). We fix $\delta \in (0,1/4)$ and $\epsilon > 0$.
    \begin{enumerate}
        \item There exists a smooth map $\widetilde{u} : [0,1]^3 \rightarrow \R^2$ such that for every $z \in [0,1]$, the map $(x,y) \mapsto \widetilde{u}(x,y,z)$ is a smooth embedding, and 
        $$\vert \widetilde{u} - u\vert_{C^0} < \epsilon.$$
        
        \item Let $\widetilde{u}_\partial : [0,1]^2 \times N^1_{2\delta} \rightarrow \R^2$ be a smooth map such that for every $z \in N^1_{2\delta}$, the map $(x,y) \mapsto \widetilde{u}_\partial(x,y,z)$ is a smooth embedding, and
        $$\vert \widetilde{u}_\partial - u \vert_{C^0} < \epsilon.$$
        Then there exists a smooth map  $\widetilde{u} : [0,1]^3 \rightarrow \R^2$ such that for every $z \in [0,1]$, the map $(x,y) \mapsto \widetilde{u}(x,y,z)$ is a smooth embedding, for every $z \in N^1_\delta$, $\widetilde{u}(\, \cdot\, , z) = \widetilde{u}_\partial(\, \cdot\, , z) $, and 
        $$\vert \widetilde{u} - u\vert_{C^0} < 2\epsilon.$$ 
    \end{enumerate}
\end{lem}

\begin{proof}
   For the first item, we first consider the case of a single topological embedding $u : [0,1]^2 \rightarrow \R^2$. It is well-known that $u$ can be approximated by smooth embeddings. The strategy goes as follows. 
   \begin{itemize}
       \item We first subdivide $[0,1]^2$ into a sufficiently fine grid.
       \item Then, we can find a small (topological) isotopy supported near $u([0,1]^2)$ which ``straightens'' the image of the grid under $u$ and makes it smooth. To that end, we first perform this isotopy near the images of the vertices, using the Jordan--Schoenflies theorem. We then isotope the edges relative to neighborhoods of the vertices by smoothing the edges as parametrized maps, and then remove potential self-intersections.
       
       We denote the resulting map by $\overline{u} : [0,1]^2 \rightarrow \R^2$, which is arbitrarily $C^0$-close to $u$ (independently of the size of the grid).
       \item We can then replace $\overline{u}$ by a smooth map near the vertices, which sends edges to edges there, extend it by a smooth map in the neighborhoods of the edges, and finally extend it over the squares. Choosing the original grid fine enough, we can ensure that the resulting smooth map is arbitrarily $C^0$-close to $u$.
   \end{itemize} 
   
   To extend this to a family of topological embeddings as in item 1, we can choose a very fine subdivision $(\sigma_0, \dots, \sigma_n)$ of $[0,1]_z$, apply the smoothing procedure to $u(\, \cdot\, , \sigma_k)$ to obtain smooth maps $\widetilde{u}_k : [0,1]^2 \rightarrow \R^2$, $0 \leq k \leq n$.\footnote{Technically speaking, we might first have to extend $u_z$ to a small neighborhood of $[0,1]^2$ before applying the smoothing, as we might have to shrink the domains later; details are left to the reader.} We now define $f_k \coloneqq \widetilde{u}^{-1}_{k+1} \circ \widetilde{u}_k :[0,1]^2 \rightarrow \R^2$, which is a smooth embedding $C^0$-close to the identity. Using Lemma~\ref{lem:technsmoothing} below, we can find a smooth isotopy $f^z_k$, $z \in [\sigma_k, \sigma_{k+1}]$, from $f_k$ to $\mathrm{id}$ which stays $C^0$-close to the identity. We then define $\widetilde{u}_z$ for $z \in [\sigma_k, \sigma_{k+1}]$ as $\widetilde{u}_z \coloneqq \widetilde{u}_{k+1} \circ f^z_k$. We might have to use suitable cutoffs to ensure that this path is smooth; details are left to the reader.

   The previous strategy immediately applies to the relative version of item 2.
\end{proof}

The key technical result used in the proof is:

\begin{lem} \label{lem:technsmoothing}
    For every $\epsilon > 0$ small enough, the following holds. If $f : [0,1]^2 \rightarrow \R^2$ is a smooth embedding such that 
    $$\vert f- \mathrm{id} \vert_{C^0} < \epsilon,$$
    then there exists an smooth isotopy $f_t : [0,1]^2 \rightarrow \R^2$, $t\in [0,1],$ such that $f_0 = \mathrm{id}$, $f_1 = f$, and for every $t \in [0,1]$, 
    $$\vert f_t - \mathrm{id}\vert_{C^0} < 2 \epsilon.$$
\end{lem}

\begin{proof}
    We use a strategy similar to that of the previous lemma. Note that we are now dealing with smooth maps.

    We first choose a sufficiently fine grid on $[0,1]^2$ (of size roughly $2\epsilon$), and use an isotopy to straighten its image under $f$. After this, we obtain an isotopy $\overline{f}_t$, $t\in [0,1]$, from $f$ to an embedding $\overline{f}$ which preserves the chosen grid. Moreover, this isotopy remains $2\epsilon$-close to $\mathrm{id}$. We can further arrange that $\overline{f}$ restricts to the identity on the grid. Then, using Smale's theorem on the contractibility of the space of diffeomorphisms of the $2$-disk that restrict to the identity along the boundary, it is easy to construct a smooth isotopy from $\overline{f}$ to $\mathrm{id}$ which stays $2\epsilon$-close to $\mathrm{id}$. Concatenating these two isotopies yields the desired isotopy between $f$ and $\mathrm{id}$.
\end{proof}

We will also need a relative version of Lemma~\ref{lem:smoothemb1}. It will be sufficient to consider the case $u = \mathrm{id}$ only. The proof follows from similar arguments and is left to the reader.

\begin{lem} \label{lem:smoothemb2}
    Let $\delta \in (0,1/4)$ and $\epsilon > 0$. We write $N^*_{r} = [0,1] \times N^2_{r}$ or $N^*_{r} = N^3_{r}$.

    Let $\widetilde{f}_\partial : N^*_{2\delta} \rightarrow \R^2$ be a smooth map such that for every $z \in [0,1]$, the map $\widetilde{f}^z_\partial: (x,y) \mapsto \widetilde{f}_\partial(x,y,z)$ is a smooth embedding (on its domain of definition), and
        $$\vert \widetilde{f}^z_\partial - \mathrm{id} \vert_{C^0} < \epsilon.$$
    Then there exists a smooth map $\widetilde{f} : [0,1]^3 \rightarrow \R^2$ such that for every $z \in [0,1]$, the map $\widetilde{f}^z:(x,y) \mapsto \widetilde{f}(x,y,z)$ is a smooth embedding, for every $(x,y,z) \in N^*_\delta$, $\widetilde{f}(x,y,z) = \widetilde{f}_\partial(x,y,z) $, and 
        $$\vert \widetilde{f}^z - \mathrm{id}\vert_{C^0} < 2\epsilon.$$ 
\end{lem}


\setcounter{footnote}{0}
\renewcommand{\thefootnote}{\fnsymbol{footnote}}
\newpage

    \section{Realizing self orbit equivalences by partially hyperbolic diffeomorphisms} \label{appendixB}

\begin{center}
\large
\begin{tabular*}{0.8\textwidth}{@{\extracolsep{\fill}} c c c}
Thomas Barthelm\'{e}\footnotemark[1] & S\'{e}rgio R. Fenley\footnotemark[2] & Rafael Potrie\footnotemark[3]
\end{tabular*}
\end{center}
\footnotetext[1]{Queen's University, Kingston, Ontario, Canada. Email address: \href{mailto:thomas.barthelme@queensu.ca}{\texttt{thomas.barthelme@queensu.ca}}. Website: \url{https://sites.google.com/site/thomasbarthelme}.}
\footnotetext[2]{Florida State University, Tallahassee, FL 32306, USA. Email address: \href{mailto:sfenley@fsu.edu}{\texttt{sfenley@fsu.edu}}.}
\footnotetext[3]{Centro de Matem\'atica, Universidad de la Rep\'ublica (Uruguay) \& IRL-IFUMI CNRS (France). Email address: \href{mailto:rpotrie@cmat.edu.uy}{\texttt{rpotrie@cmat.edu.uy}}. Website: \url{https://www.cmat.edu.uy/~rpotrie/}.}

\bigskip

This appendix uses the results from the main paper to solve an important problem in the classification of partially hyperbolic diffeomorphisms. We refer the reader to~\cite{BGHP,BFP} for a presentation of this problem as well as precise definitions. After Pujals' conjecture was shown not to hold, new examples of partially hyperbolic diffeomorphisms in 3-manifolds started to appear. In particular, in~\cite{BGHP}, a general criterion for constructing examples was devised. In~\cite{BFP} we proposed a way to correct the conjecture, by considering the class of \emph{collapsed Anosov flows}. Roughly speaking, these are partially hyperbolic diffeomorphisms whose dynamics corresponds to that of a self orbit equivalence of an Anosov flow.
 Due to the previous work on other 3-manifolds, the classification problem of partially hyperbolic diffeomorphisms became to show that, in closed 3-manifolds with non virtually solvable fundamental group, every partially hyperbolic diffeomorphism is a collapsed Anosov flow. This was shown to hold, for instance, in hyperbolic 3-manifolds (\cite{BFFP1,BFFP2, FP-dt}), and was recently announced by the last two authors of this appendix to hold for transitive partially hyperbolic diffeomorphisms in \emph{any} 3-manifold~\cite{FP-phtf}.

Such classification results proved that, to partially hyperbolic diffeomorphisms, one can always associate an Anosov flow and a self orbit equivalence of it. The other direction of that correspondence remained a key open question though  (see~\cite[Question 3]{BFP}). More precisely: Can every self orbit equivalence be realized by a collapsed Anosov flow?  In this appendix we give a positive answer to this question under orientability assumptions. 

\begin{thm}\label{teo.mainappendix}
Let $\varphi_t\colon M \to M$ be an Anosov flow in an orientable 3-manifold with orientable foliations. Let $\beta_0\colon M \to M$ be a self orbit equivalence preserving orientations of the bundles. Then, there exists $f\colon M \to M$ a (strong) collapsed Anosov flow associated to $\beta_0$. 
\end{thm}

This in particular completes the classification of partially hyperbolic diffeomorphisms in hyperbolic $3$-manifolds (see~\cite{BFFP1,BFFP2,FP-dt}), as well as for transitive partially hyperbolic diffeomorphisms in any 3-manifolds (by~\cite{FP-phtf}): The main case left open was whether examples of collapsed Anosov flows called \emph{double translations} existed in those manifolds. The existence of such follows from Theorem~\ref{teo.mainappendix} as realization of self orbit equivalences of the one-step up map of some $\RR$-covered Anosov flow in a hyperbolic 3-manifold.

Notice that we do not require the flow to be transitive in Theorem~\ref{teo.mainappendix}, but it is assumed to be a true (i.e., smooth) Anosov flow and not just a topological Anosov flow as was considered in~\cite{BFP}. For transitive Anosov flows, the two notions of smooth and topological coincide up to orbit equivalence, thanks to~\cite{Sh21}, but it is not yet known whether these notions also coincide for non-transitive flows.

We start by quickly recalling the definitions of the objects we are working with here, and refer to~\cite{BFP} for details and more precise statements. 

A \emph{self orbit equivalence} $\beta_0\colon M \to M$ of an Anosov flow $\{\varphi_t\}_t$ is a homeomorphism of $M$ which sends (oriented) orbits to (oriented) orbits of the flow. It can be shown that such a homeomorphism also preserves the weak-stable and weak-unstable foliations $\cF^{ws}$ and $\cF^{wu}$ which intersect in the orbits of the flow. Two self orbit equivalences $\beta_0$ and $\beta$ are said to be equivalent if there is a continuous function $\tau\colon M \to \RR$ so that $\beta(x) = \varphi_{\tau(x)} ( \beta_0(x))$. 

A \emph{collapsed Anosov flow} associated to $(\varphi_t,\beta_0)$ is a partially hyperbolic diffeomorphism $f\colon M \to M$ such that there is a continuous map $h\colon M \to M$ homotopic to the identity and a self orbit equivalence $\beta$ equivalent to $\beta_0$ such that: 

\begin{itemize}
\item $f \circ h = h \circ \beta$, 
\item the map $h$ sends orbits of the flow to curves tangent to the center direction $E^c$ of $f$. 
\end{itemize}

We say that $f$ is a \emph{strong collapsed Anosov flow} if, moreover, the map $h$ sends the leaves of the foliations $\cF^{ws}$ and $\cF^{wu}$ to immersed surfaces tangent respectively to the bundles $E^{cs}$ and $E^{cu}$ of the partially hyperbolic diffeomorphism $f$ and gives $f$-invariant \emph{branching foliations} $\cW^{cs}$ and $\cW^{cu}$.
 
In~\cite[\S 10]{BFP}, we extended the work in~\cite{BGHP} and proved the following fact.

\begin{prop}\label{prop-caf}
Let $\varphi_t\colon M \to M$ be an Anosov flow, $\eta\colon M \to M$ a diffeomorphism and, for all $t$, $f_t\colon M \to M$ defined by $f_t := \varphi_t \circ \eta \circ \varphi_t$.

If $D\eta(T\cF^{ws})$ is transverse to $T\cF^{wu}$ and $D\eta(T\cF^{wu})$ to $T\cF^{ws}$, then there is $t_0$ such that for all $t>t_0$ one has: 
 
 \begin{enumerate}[label=(\arabic*)]
 \item \label{item_strong_caf} $f_t$ is a strong collapsed Anosov flow associated to $(\varphi_t, \beta)$ where $\beta$ is a self orbit equivalence independent of $t$, 
 \item \label{item_bundle_convergence} as $t \to +\infty$ the bundles $E^s_t$, $E^{c}_t$, $E^u_t$ associated to $f_t$ converge uniformly to the bundles associated to the Anosov flow $\{\varphi_t\}_t$. 
 \item \label{item_branching_foliations} the branching foliations $\widetilde{\cW^{cs}_t}$ and $\widetilde{\cW^{cu}_t}$ in the universal cover, converge uniformly to the Anosov foliations $\widetilde{\cF^{ws}}$ and $\widetilde{\cF^{wu}}$.
 \item \label{item_center_curves} center curves converge uniformly to orbits of the Anosov flow in the universal cover. 
 \end{enumerate}
\end{prop} 
  
Item \ref{item_strong_caf} is~\cite[Theorem A]{BFP} (the independence on $t$ follows from~\cite[Theorem C]{BFP}). Item \ref{item_bundle_convergence} follows from~\cite{BGHP} (see~\cite[Proposition 10.1]{BFP}).  Item \ref{item_branching_foliations} follows from~\cite[Proposition 10.1]{BFP}) and the proof of~\cite[Theorem A]{BFP}. Notice that in this item, uniform convergence is meant as a strong uniform convergence, i.e., given $\eps>0$ there is $t_{\eps}$ so that, for $t>t_{\eps}$, leaves of the branching foliations are uniformly $\eps$-$C^1$-close to their corresponding leaves via the map $h_t$ in the definition of strong collapsed Anosov flow which by construction is $C^0$-close to identity. 

Finally, while item \ref{item_center_curves} is not explicitly stated in~\cite[\S 10]{BFP}, it follows directly from the description of center curves in~\cite[Proposition 10.1]{BFP} as well as the uniqueness properties of the branching foliations~\cite[Proposition 10.3]{BFP} and~\cite[Proposition 10.6]{BFP}. 

Let us now restate Corollary~\ref{corintro:smoothing} of the main paper (note that it is standard that Anosov flows in dimension 3 have $C^1$ weak-stable and weak-unstable foliations; see~\cite[Corollary 1.8]{Hasselblatt}):

\begin{thm}\label{teo-mainmain}
Let  $\{\varphi_t\}$ be an Anosov flow in an orientable 3-manifold with orientable foliations, and $\beta_0 \colon M \to M$ be a self orbit equivalence preserving orientations. Then, for every $\eps>0$, there is a diffeomorphism $\eta_{\eps}\colon M \to M$ which is $\eps$-$C^0$-close to $\beta_0$ and such that $D\eta_{\eps}(T\cF^{ws})$ makes angle less than $\eps$ with $T\cF^{ws}$ and $D\eta_{\eps}(T\cF^{wu})$ makes angle less than $\eps$ with $T\cF^{wu}$. 
\end{thm}

Putting together Theorem \ref{teo-mainmain} and Proposition \ref{prop-caf}, we get that, for large $t$ and any fixed $\eps>0$, $f_{t,\eps} = \varphi_t \circ \eta_{\eps} \circ \varphi_t$ are strong collapsed Anosov flows associated to $\varphi_t$ and some self orbit equivalence $\beta'_{\eps}$. Our goal in order to prove Theorem \ref{teo.mainappendix} is then to show that the self orbit equivalence $\beta'_{\eps}$ associated to $f_{t,\eps}$ is equivalent to the original $\beta_0$. Note that there are cases (\cite{BG}) where there are unique (or finitely many) self orbit equivalences in a given mapping class of the manifold. In these cases, it is easy to establish the equivalence class of the self orbit equivalence as $f_{t,\eps}$ is always homotopic to $\eta_{\eps}$ which is homotopic to $\beta_0$. It is therefore the other case (which always corresponds to skewed Anosov flows) that is more challenging and requires more arguments since an homotopy class will contain infinitely many inequivalent self orbit equivalences (but again, thanks to~\cite{BG} we know exactly how they differ from each other). 

We first quote the following result from~\cite{BG} which reduces the problem to the skewed case. For skewed Anosov flows, there is a specific self orbit equivalence, called \emph{one-step up map} constructed by using the skewed structure in the universal cover (see~\cite{BG,BM24}). Note that for some flows (e.g., the geodesic flow) this one step up map can be finite order (or even the identity) but it is always homotopic to the identity and sometimes has infinite order.  

\begin{thm}\label{teoBG} Let $\varphi_t \colon M \to M$ be an Anosov flow. If $\beta_1$ and $\beta_2$ are inequivalent self orbit equivalences of $\varphi_t$ and homotopic to each other, then, $\varphi_t$ is skewed Anosov and $\beta_1 \circ \beta_2^{-1}$ is equivalent to an iterate of a one step up map. 
 \end{thm} 

For more information about \emph{skewed Anosov flows} (sometimes called \emph{skewed} $\RR$-\emph{covered}) and general background on (topological) Anosov flows see~\cite{BM_book}. 

As a consequence we get an easy criterion to check if two homotopic self orbit equivalences are equivalent or not:  

\begin{lem}\label{lem.AF}
Assume that $\beta'$ and $\beta$ are homotopic self orbit equivalences of an Anosov flow $\varphi_t$ and let $\tilde \beta'$ and $\tilde \beta$ be lifts to $\widetilde{M}$ at bounded distance. Let $E$ be a leaf of $\widetilde{\cF}^{wu}$ which is fixed by some nontrivial element $\gamma \in \pi_1(M)$. Then, there is $\delta>0$ (depending only on $\varphi_t$) for which the following is true: if $\tilde \beta(E)$ and $\tilde \beta'(E)$ have points at distance less than $\delta$, then $\beta$ and $\beta'$ are equivalent.  
\end{lem}

\begin{proof}
This follows from the fact that there exists $\delta>0$ (depending only on $\varphi_t$) so that the lift of the one step up map sends any leaf $E \in \widetilde{\cF^{wu}}$ to a leaf $E'$ so that the closest point of $E$ to $E'$ is larger than $\delta$. This is because the foliations are induced by a slithering $\widetilde M \to S^1$ (see~\cite{Thurston}). Thus, applying Theorem \ref{teoBG} we conclude. 
\end{proof}

Now we are ready to prove the main result of the appendix: 

\begin{proof}[Proof of Theorem \ref{teo.mainappendix}]
We consider an Anosov flow $\varphi_t\colon M \to M$ and a self orbit equivalence $\beta_0\colon M \to M$ preserving orientations.

We apply Theorem \ref{teo-mainmain} to obtain $\eta_{\eps}\colon M \to M$ a diffeomorphism $\eps$-$C^0$-close to $\beta_0$ which respects transversalities with angle $\leq \eps$.  In particular, $\eta_{\eps}$ and $\beta_0$ are homotopic, and if we fix $\tilde \beta_0$ a lift of $\beta_0$, we can consider $\tilde \eta_{\eps}$ to be the unique lift of $\eta_{\eps}$ which is $\eps$-close to $\tilde \beta_0$. We also consider $\tilde \varphi_t$ to be the (unique) flow in $\mt$ which lifts $\varphi_t$. 

Note that the diffeomorphisms: 

$$ f_{t,\eps} = \varphi_t \circ \eta_{\eps} \circ \varphi_t, $$

\noindent are all homotopic to $\beta_0$ by construction. By Proposition \ref{prop-caf}, for large enough $t$, these are partially hyperbolic and (strong) collapsed Anosov flows with respect to $\varphi_t$ and some self orbit equivalence $\beta' = \beta'_{t,\eps}$ homotopic to $\eta_{\eps}$ and therefore also homotopic to $\beta_0$. We wish to show that $\beta'$ is equivalent to $\beta_0$, at least for large $t$ and small $\eps$. By Theorem \ref{teoBG} we can assume that if $\beta'$ and $\beta_0$ are not equivalent, then $\varphi_t$ is skewed and they differ by a power of a one-step up map (but we will use this only by applying Lemma \ref{lem.AF}). 

We call $\hat f = f_{t_0,\eps}$ and $\tilde f$ the lift to $\widetilde{M}$ associated to $\tilde \beta_0$, that is, $\tilde f = \tilde \varphi_{t_0} \circ \tilde \eta_{\eps} \circ \tilde \varphi_{t_0}$. Fix a periodic orbit $o$ of $\varphi_t$ and a lift $\tilde o$ to $\widetilde{M}$. Let $c$ be the center curve in $\widetilde{M}$ associated to $\tilde o$, which in particular is $C^0$-close to $\tilde o$ everywhere. Let  $E$ the leaf of $\widetilde{\cF}^{wu}$ containing $\tilde o$. We want to show that if $\eps$ is small and $t_0$ is large, then $\tilde f(c)$ is close to $\tilde \beta_0(E)$, which then by Lemma \ref{lem.AF} implies that $\beta'$ is equivalent to $\beta_0$. In fact, we will show that if $L \in \widetilde{\cW}^{cu}$ is the leaf containing $c$ and  $E$ the leaf of $\widetilde{\cF}^{wu}$ containing $\tilde o$, then the distance of $\tilde f(L)$ (which if $\eps$ is small and $t$ is big is very close to $\tilde \beta'(E)$) and $\tilde \beta_0(E)$ is smaller than $\delta$ for some fixed $\delta$ given by Lemma \ref{lem.AF}. 

Note that if $\epsilon_1 \ll \delta$, then, thanks to item \ref{item_branching_foliations} of Proposition \ref{prop-caf}, we can choose $\eps$ and $t$ so that $E$ and $L$ are uniformly $\eps_1/10$ close and are both invariant under the same deck transformation, say $\gamma \in \pi_1(M)$. Proposition \ref{prop-caf} \ref{item_bundle_convergence} implies that for $t_0$ large we have that the bundles of $\hat f$ are very close to those of $\varphi_t$ which make good angle, and $E$ is tangent to the weak-unstable bundle of $\varphi_t$ and $L$ to the center unstable bundle of $\hat f$, therefore, we know that $E \subset \bigcup_{x \in L} \cW^{ss}_{loc}(x)$ and thus, we know that the distance between $\tilde f(E)$ and $\tilde f(L)$ is uniformly less than $\eps_1/10$. We can also assume that $\tilde f(L)$ (and therefore $\tilde f(E)$ if $\eps_1$ is sufficiently small) is contained in $\bigcup_{x \in E} \cF^{ss}_{loc}(x)$ where $\cF^{ss}_{loc}$ denotes the local strong-stable manifold for the flow $\tilde \varphi_t$. 

Note that $\tilde f(E) = \tilde \varphi_{t_0} \circ \tilde \eta_{\eps} \circ \tilde \varphi_{t_0} (E) =  \tilde \varphi_{t_0} \circ \tilde \eta_{\eps} (E)$ which is very close to $\tilde f(L)$ as was remarked before. If $t_0$ is large, since $\tilde \eta_{\eps}$ is $\eps$ close to $\tilde \beta_0$, flowing by $\tilde \varphi_{t_0}$ this gets even closer to $\tilde \beta_0(E)$. We deduce that $\tilde f(E)$, which is $\eps_1$ close to some leaf of $\widetilde{\cF^{wu}}$ invariant under $(\beta_0)_\ast \gamma$ needs to be close to $\widetilde\beta_0(E)$. But this implies that $\tilde f(L)$ is the leaf close to $\widetilde{\beta_0}(E)$ as we wanted to show. 
\end{proof}

\paragraph{Acknowledgments.}

T.B. is partially supported by the NSERC (ALLRP 598447-24 and RGPIN-2024-04412). S.F. was partially supported by National Science Foundation grant DMS-2054909. R.P. was partially supported by CSIC.

\newpage
\printbibliography[heading=bibintoc, title={References}]

\end{document}

%% file: tikz/path.tex
\begin{tikzpicture}[scale=1.5]
    \clip (-1.6,-1.6) rectangle (2.1,1.6);
    \draw[fill=teal!30, dashed, draw=teal] (0,0) circle (1.5);
    \node[color=teal] at (-0.8,-0.8) {$\mathcal{V}$};
    \node[draw, circle, inner sep=0.7pt, fill=blue] (F) at (0,0) {};
    \node[color=blue] at (0, 0) [yshift=8pt] {$\mathcal{F}$};
    \draw[color=purple, rounded corners=24pt] 
        (0.5, 0.7)
    -- (1,1.5)
    -- (2, 1)
    -- (2,-0.2)
    -- (1,-1)
    -- (0.3,-0.5);
    \node[draw, circle, inner sep=0.7pt, fill=purple] at (0.5,0.7) {};
    \node[color=purple] at (0.5, 0.7) [yshift=10pt] {$\xi$};
    \node[draw, circle, inner sep=0.7pt, fill=purple] at (0.3,-0.5) {};
    \node[color=purple] at (0.3,-0.5) [yshift=-10pt] {$\xi'$};
\end{tikzpicture}

%% file: tikz/h.tex
\begin{tikzpicture}[scale=1.7]
    \foreach \y in {0.2, 0.4,  ..., 1.8} {
        \pgfmathsetmacro{\colorval}{40 + 30 * \y}
        \draw[line width=0.6pt, color=blue!\colorval] (0, \y) -- (2, \y);
    }
    \draw[->] (-0.2, 0) -- (2.2, 0) node[right] {$x$};
    \draw[->] (0, -0.2) -- (0, 2.2) node[above] {$y$};
    \draw[thick] (0, 0) rectangle (2, 2);
    \foreach \y in {0.2, 0.4,  ..., 1.8} {
        \pgfmathsetmacro{\colorval}{40 + 30 * \y}
        \draw[line width=0.6pt, color=blue!\colorval] ({3.5+1/10*sin(200*2*sin(40*\y)) + 1/25*cos(500*2*sin(40*\y))}, {2*sin(40*\y)}) -- ({5.5-1/15*sin(-300*2*sin(40*\y)) + 1/20*cos(400*2*sin(40*\y))}, {2*sin(40*\y)});
    }
    \draw[->] (3.1, -0.2) -- (5.9, -0.2) node[right] {$x$};
    \draw[->] (3.3, -0.4) -- (3.3, 2.4) node[above] {$y$};
    \draw[thick] (3.3, -0.2) rectangle (5.7, 2.2);
    \draw[blue, -{Latex[length=1mm]}] (2.2, 1) -- (3.1, 1) node[midway, above] {$h$};
    \draw[domain=0:2, samples=100, smooth, variable=\t, gray, line cap=round] plot ({3.5+1/10*sin(200*\t) + 1/25*cos(500*\t)}, {\t});
    \draw[domain=0:2, samples=100, smooth, variable=\t, gray, line cap=round] plot ({5.5-1/15*sin(-300*\t) + 1/20*cos(400*\t)}, {\t});
    \draw[gray, line cap=round] ({3.5+1/10*sin(200*2) + 1/25*cos(500*2)},2) -- ({5.5-1/15*sin(-300*2) + 1/20*cos(400*2)}, 2);
    \draw[gray, line cap=round] ({3.5+1/25},0) -- ({5.5+ 1/20}, 0);
\end{tikzpicture}

%% file: tikz/cover.tex
\begin{tikzpicture}[scale=2.5]
    \clip(-1.5,-1.85) rectangle (1.4,1.85);
    \draw[pattern={Lines[distance={7pt}]},pattern color=lightgray, opacity=0.8]
    (-2,-2) rectangle (2,2);
    \draw[black, very thick] (-1, -1.5) -- (0.75,0.25);
    \draw[black, thick, dashed] (-1, -1.5) -- (-1.15,-1.8);
    \draw[black, thick, dashed] (-1, -1.5) -- (-1.25,-1.2);
    \draw[black, thick, dashed] (-1, -1.5) -- (-0.7,-1.75);
    \draw[black, very thick] (0.75,0.25) -- (-1, 1.5);
    \draw[black, thick, dashed] (0.75,0.25) -- (0.85, 0.55);
    \draw[black, thick, dashed] (0.75,0.25) -- (0.95, 0);
    \draw[black, very thick] (-1, 1.5) -- (-1, -1.5);
    \draw[black, thick, dashed] (-1, 1.5) -- (-0.8, 1.75);
    \draw[black, thick, dashed] (-1, 1.5) -- (-1.2, 1.75);
    \draw[black, thick, dashed] (-1, 1.5) -- (-1.2, 1.25);
    \filldraw[black] (-1, -1.5) circle (0.015);
    \filldraw[black] (0.75,0.25) circle (0.015);
    \filldraw[black] (-1, 1.5) circle (0.015);
    \draw[thick, green!70!gray, rounded corners=2] (-1.25, -1.7) rectangle (-0.6, -1.25);
    \fill[green!70!gray, opacity=0.2, rounded corners=2] (-1.25, -1.7) rectangle (-0.6, -1.25);
    \draw[thick, green!70!gray, rounded corners=2] (-1.25, 1.7) rectangle (-0.55, 1.3);
    \fill[green!70!gray, opacity=0.2, rounded corners=2] (-1.25, 1.7) rectangle (-0.55, 1.3);
    \draw[thick, green!70!gray, rounded corners=2] (0.35, 0.05) rectangle (1.05, 0.4);
    \fill[green!70!gray, opacity=0.2, rounded corners=2] (0.35, 0.05) rectangle (1.05, 0.4);
    \node[green!70!gray] at (-0.4, 1.5) {$U_0$};
    \node[green!70!gray] at (-0.45, -1.5) {$U_0$};
    \node[green!70!gray] at (1.2, 0.2) {$U_0$};
    \draw[thick, magenta, rounded corners=2] (-1.1, -1.65) rectangle (-0.9, 1.6);
    \fill[magenta, opacity=0.1, rounded corners=2] (-1.1, -1.65) rectangle (-0.9, 1.6);
    \draw[thick, magenta, rounded corners=2] (-1.2, -1.6) -- (0.75,0.35) -- (0.95,0.35) -- (-1, -1.6) -- cycle;
    \fill[magenta, opacity=0.1, rounded corners=2] (-1.2, -1.6) -- (0.75,0.35) -- (0.95,0.35) -- (-1, -1.6) -- cycle;
    \draw[thick, magenta, rounded corners=2] (0.75, 0.15) -- (1, 0.15) -- (-0.95, 1.55) -- (-1.2, 1.55) -- cycle;
    \fill[magenta, opacity=0.1] (0.75, 0.15) -- (1, 0.15) -- (-0.95, 1.55) -- (-1.2, 1.55) -- cycle;
    \node[magenta] at (-1.22, 0.2) {$U_1$};
    \node[magenta] at (0, 1) {$U_1$};
    \node[magenta] at (0, -0.8) {$U_1$};
    \draw[thick, teal, rounded corners =2] (-0.98, -1.55) -- (-1.06, -1.55) -- (-1.06, 1.53) -- (-0.97, 1.53) -- (0.82, 0.25) -- cycle;
    \fill[teal, opacity=0.2, rounded corners=2] (-0.98, -1.55) -- (-1.06, -1.55) -- (-1.06, 1.53) -- (-0.97, 1.53) -- (0.82, 0.25) -- cycle;
    \node[teal] at (-0.3, 0.2) {$U_2$};
    \node[gray] at (0.8, -1.2) {$\mathcal{F}_0$};
    \end{tikzpicture}

%% file: tikz/h0.tex
\begin{tikzpicture}[scale=1.7]
    \foreach \y in {0.1, 0.2, 0.3, 1.7, 1.8, 1.9} {
        \pgfmathsetmacro{\colorval}{40 + 30 * \y}
        \draw[line width=0.6pt, color=blue!\colorval] (0, \y) -- (2, \y);
    }
    \foreach \y in {0.4, 0.6,...,1.8} {
        \pgfmathsetmacro{\colorval}{40 + 30 * \y}
        \draw[line width=0.3pt, color=blue!\colorval, dotted] (0, \y) -- (2, \y);
    }
    \draw[->] (-0.2, 0) -- (2.2, 0) node[right] {$x$};
    \draw[->] (0, -0.2) -- (0, 2.2) node[above] {$y$};
    \draw[thick] (0, 0) rectangle (2, 2);
    \foreach \y in {0.1, 0.2, 0.3, 1.7, 1.8, 1.9} {
        \pgfmathsetmacro{\colorval}{40 + 30 * \y}
        \draw[line width=0.2pt, color=blue!\colorval] ({3.5+1/10*sin(200*2*sin(40*\y)) + 1/25*cos(500*2*sin(40*\y))}, {2*sin(40*\y)}) -- ({5.5-1/15*sin(-300*2*sin(40*\y)) + 1/20*cos(400*2*sin(40*\y))}, {2*sin(40*\y)});
    }
    \foreach \y in {0.12, 0.18, 0.33, 1.66, 1.84, 1.96} {
        \pgfmathsetmacro{\colorval}{40 + 30 * \y}
        \draw[line width=0.6pt, color=violet!\colorval] ({3.5+1/10*sin(200*2*sin(40*\y)) + 1/25*cos(500*2*sin(40*\y))}, {2*sin(40*\y)}) -- ({5.5-1/15*sin(-300*2*sin(40*\y)) + 1/20*cos(400*2*sin(40*\y))}, {2*sin(40*\y)});
    }
    \foreach \y in {0.4, 0.6,...,1.8} {
        \pgfmathsetmacro{\colorval}{40 + 30 * \y}
        \draw[line width=0.3pt, color=blue!\colorval, dotted] ({3.5+1/10*sin(200*2*sin(40*\y)) + 1/25*cos(500*2*sin(40*\y))}, {2*sin(40*\y)}) -- ({5.5-1/15*sin(-300*2*sin(40*\y)) + 1/20*cos(400*2*sin(40*\y))}, {2*sin(40*\y)});
    }
    \draw[->] (3.1, -0.2) -- (5.9, -0.2) node[right] {$x$};
    \draw[->] (3.3, -0.4) -- (3.3, 2.4) node[above] {$y$};
    \draw[thick] (3.3, -0.2) rectangle (5.7, 2.2);
    \draw[violet, -{Latex[length=1mm]}] (2.2, 1) -- (3.1, 1) node[midway, above] {$\widetilde{h}_{\mathfrak{t},0}$};
    \draw[domain=0:2, samples=100, smooth, variable=\t, gray, line cap=round] plot ({3.5+1/10*sin(200*\t) + 1/25*cos(500*\t)}, {\t});
    \draw[domain=0:2, samples=100, smooth, variable=\t, gray, line cap=round] plot ({5.5-1/15*sin(-300*\t) + 1/20*cos(400*\t)}, {\t});
    \draw[gray, line cap=round] ({3.5+1/10*sin(200*2) + 1/25*cos(500*2)},2) -- ({5.5-1/15*sin(-300*2) + 1/20*cos(400*2)}, 2);
    \draw[gray, line cap=round] ({3.5+1/25},0) -- ({5.5+ 1/20}, 0);
\end{tikzpicture}

%% file: tikz/h1.tex
\begin{tikzpicture}[scale=1.7]
    \foreach \y in {0.1, 0.2, 0.3, 1.7, 1.8, 1.9} {
        \pgfmathsetmacro{\colorval}{40 + 30 * \y}
        \draw[line width=0.6pt, color=blue!\colorval] (0, \y) -- (2, \y);
    }
    \foreach \y in {0.4, 0.6,...,1.8} {
        \pgfmathsetmacro{\colorval}{40 + 30 * \y}
        \draw[line width=0.3pt, color=blue!\colorval, dotted] (0.3, \y) -- (1.7, \y);
    }
    \foreach \y in {0.4, 0.6,...,1.8} {
        \pgfmathsetmacro{\colorval}{40 + 30 * \y}
        \draw[line width=0.6pt, color=blue!\colorval] (0, \y) -- (0.3, \y);
    }
    \foreach \y in {0.4, 0.6,...,1.8} {
        \pgfmathsetmacro{\colorval}{40 + 30 * \y}
        \draw[line width=0.6pt, color=blue!\colorval] (1.7, \y) -- (2, \y);
    }
    \draw[->] (-0.2, 0) -- (2.2, 0) node[right] {$x$};
    \draw[->] (0, -0.2) -- (0, 2.2) node[above] {$y$};
    \draw[thick] (0, 0) rectangle (2, 2);
    \foreach \y in {0.1, 0.2, 0.3, 1.7, 1.8, 1.9} {
        \pgfmathsetmacro{\colorval}{40 + 30 * \y}
        \draw[line width=0.2pt, color=blue!\colorval] ({3.5+1/10*sin(200*2*sin(40*\y)) + 1/25*cos(500*2*sin(40*\y))}, {2*sin(40*\y)}) -- ({5.5-1/15*sin(-300*2*sin(40*\y)) + 1/20*cos(400*2*sin(40*\y))}, {2*sin(40*\y)});
    }
    \foreach \y in {0.4, 0.6,...,1.6} {
        \pgfmathsetmacro{\colorval}{40 + 30 * \y}
        \draw[dotted, line width=0.2pt, color=blue!\colorval] ({3.5+1/10*sin(200*2*sin(40*\y)) + 1/25*cos(500*2*sin(40*\y))}, {2*sin(40*\y)}) -- ({5.5-1/15*sin(-300*2*sin(40*\y)) + 1/20*cos(400*2*sin(40*\y))}, {2*sin(40*\y)});
    }
    \foreach \y in {0.4, 0.6,...,1.6} {
        \pgfmathsetmacro{\colorval}{40 + 30 * \y}
        \draw[line width=0.2pt, color=blue!\colorval] ({3.5+1/10*sin(200*2*sin(40*\y)) + 1/25*cos(500*2*sin(40*\y))}, {2*sin(40*\y)}) -- ({3.5+1/10*sin(200*2*sin(40*\y)) + 1/25*cos(500*2*sin(40*\y))+0.3}, {2*sin(40*\y)});
    }\foreach \y in {0.4, 0.6,...,1.6} {
        \pgfmathsetmacro{\colorval}{40 + 30 * \y}
        \draw[line width=0.2pt, color=blue!\colorval] ({5.5-1/15*sin(-300*2*sin(40*\y)) + 1/20*cos(400*2*sin(40*\y))-0.3}, {2*sin(40*\y)}) -- ({5.5-1/15*sin(-300*2*sin(40*\y)) + 1/20*cos(400*2*sin(40*\y))}, {2*sin(40*\y)});
    }
    \foreach \y in {0.12, 0.18, 0.32, 1.66, 1.84, 1.96} {
        \pgfmathsetmacro{\colorval}{40 + 30 * \y}
        \draw[line width=0.6pt, color=violet!\colorval] ({3.5+1/10*sin(200*2*sin(40*\y)) + 1/25*cos(500*2*sin(40*\y))}, {2*sin(40*\y)}) -- ({5.5-1/15*sin(-300*2*sin(40*\y)) + 1/20*cos(400*2*sin(40*\y))}, {2*sin(40*\y)});
    }
    \foreach \y in {0.38, 0.57, 0.78, 1.03, 1.22, 1.43} {
        \pgfmathsetmacro{\colorval}{40 + 30 * \y}
        \draw[line width=0.6pt, color=violet!\colorval] ({3.5+1/10*sin(200*2*sin(40*\y)) + 1/25*cos(500*2*sin(40*\y))}, {2*sin(40*\y)}) -- ({3.5+1/10*sin(200*2*sin(40*\y)) + 1/25*cos(500*2*sin(40*\y))+0.3}, {2*sin(40*\y)});
    }
    \foreach \y in {0.42, 0.61, 0.82, 1, 1.18, 1.37} {
        \pgfmathsetmacro{\colorval}{40 + 30 * \y}
        \draw[line width=0.6pt, color=violet!\colorval] ({5.5-1/15*sin(-300*2*sin(40*\y)) + 1/20*cos(400*2*sin(40*\y))-0.3}, {2*sin(40*\y)}) -- ({5.5-1/15*sin(-300*2*sin(40*\y)) + 1/20*cos(400*2*sin(40*\y))}, {2*sin(40*\y)});
    }
    \draw[->] (3.1, -0.2) -- (5.9, -0.2) node[right] {$x$};
    \draw[->] (3.3, -0.4) -- (3.3, 2.4) node[above] {$y$};
    \draw[thick] (3.3, -0.2) rectangle (5.7, 2.2);
    \draw[violet,-{Latex[length=1mm]}] (2.2, 1) -- (3.1, 1) node[midway, above] {$\widetilde{h}_{\mathfrak{t},1}$};
    \draw[domain=0:2, samples=100, smooth, variable=\t, gray, line cap=round] plot ({3.5+1/10*sin(200*\t) + 1/25*cos(500*\t)}, {\t});
    \draw[domain=0:2, samples=100, smooth, variable=\t, gray, line cap=round] plot ({5.5-1/15*sin(-300*\t) + 1/20*cos(400*\t)}, {\t});
    \draw[gray, line cap=round] ({3.5+1/10*sin(200*2) + 1/25*cos(500*2)},2) -- ({5.5-1/15*sin(-300*2) + 1/20*cos(400*2)}, 2);
    \draw[gray, line cap=round] ({3.5+1/25},0) -- ({5.5+ 1/20}, 0);
\end{tikzpicture}

%% file: tikz/h2.tex
\begin{tikzpicture}[scale=1.7]
    \foreach \y in {0.1, 0.2, 0.3, 1.7, 1.8, 1.9} {
        \pgfmathsetmacro{\colorval}{40 + 30 * \y}
        \draw[line width=0.6pt, color=blue!\colorval] (0, \y) -- (2, \y);
    }
    \foreach \y in {0.4, 0.6,...,1.8} {
        \pgfmathsetmacro{\colorval}{40 + 30 * \y}
        \draw[line width=0.3pt, color=blue!\colorval, dotted] (0.3, \y) -- (1.7, \y);
    }
    \foreach \y in {0.4, 0.6,...,1.8} {
        \pgfmathsetmacro{\colorval}{40 + 30 * \y}
        \draw[line width=0.6pt, color=blue!\colorval] (0, \y) -- (0.3, \y);
    }
    \foreach \y in {0.4, 0.6,...,1.8} {
        \pgfmathsetmacro{\colorval}{40 + 30 * \y}
        \draw[line width=0.6pt, color=blue!\colorval] (1.7, \y) -- (2, \y);
    }
    \draw[->] (-0.2, 0) -- (2.2, 0) node[right] {$x$};
    \draw[->] (0, -0.2) -- (0, 2.2) node[above] {$y$};
    \draw[thick] (0, 0) rectangle (2, 2);
    \foreach \y in {0.1, 0.2, 0.3, 1.7, 1.8, 1.9} {
        \pgfmathsetmacro{\colorval}{40 + 30 * \y}
        \draw[line width=0.2pt, color=blue!\colorval] ({3.5+1/10*sin(200*2*sin(40*\y)) + 1/25*cos(500*2*sin(40*\y))}, {2*sin(40*\y)}) -- ({5.5-1/15*sin(-300*2*sin(40*\y)) + 1/20*cos(400*2*sin(40*\y))}, {2*sin(40*\y)});
    }
    \foreach \y in {0.4, 0.6,...,1.6} {
        \pgfmathsetmacro{\colorval}{40 + 30 * \y}
        \draw[dotted, line width=0.2pt, color=blue!\colorval] ({3.5+1/10*sin(200*2*sin(40*\y)) + 1/25*cos(500*2*sin(40*\y))}, {2*sin(40*\y)}) -- ({5.5-1/15*sin(-300*2*sin(40*\y)) + 1/20*cos(400*2*sin(40*\y))}, {2*sin(40*\y)});
    }
    \foreach \y in {0.4, 0.6,...,1.6} {
        \pgfmathsetmacro{\colorval}{40 + 30 * \y}
        \draw[line width=0.2pt, color=blue!\colorval] ({3.5+1/10*sin(200*2*sin(40*\y)) + 1/25*cos(500*2*sin(40*\y))}, {2*sin(40*\y)}) -- ({3.5+1/10*sin(200*2*sin(40*\y)) + 1/25*cos(500*2*sin(40*\y))+0.3}, {2*sin(40*\y)});
    }\foreach \y in {0.4, 0.6,...,1.6} {
        \pgfmathsetmacro{\colorval}{40 + 30 * \y}
        \draw[line width=0.2pt, color=blue!\colorval] ({5.5-1/15*sin(-300*2*sin(40*\y)) + 1/20*cos(400*2*sin(40*\y))-0.3}, {2*sin(40*\y)}) -- ({5.5-1/15*sin(-300*2*sin(40*\y)) + 1/20*cos(400*2*sin(40*\y))}, {2*sin(40*\y)});
    }
    \foreach \y in {0.12, 0.18, 0.32, 1.66, 1.84, 1.96} {
        \pgfmathsetmacro{\colorval}{40 + 30 * \y}
        \draw[line width=0.6pt, color=violet!\colorval] ({3.5+1/10*sin(200*2*sin(40*\y)) + 1/25*cos(500*2*sin(40*\y))}, {2*sin(40*\y)}) -- ({5.5-1/15*sin(-300*2*sin(40*\y)) + 1/20*cos(400*2*sin(40*\y))}, {2*sin(40*\y)});
    }
    \foreach \y/\z in {0.38/0.42, 0.57/0.61, 0.78/0.82, 1.03/1, 1.22/1.18, 1.43/1.37} {
        \pgfmathsetmacro{\colorval}{40 + 30 * \y}
        \draw[line width=0.6pt, color=violet!\colorval, rounded corners = 5pt] ({3.5+1/10*sin(200*2*sin(40*\y)) + 1/25*cos(500*2*sin(40*\y))}, {2*sin(40*\y)}) -- ({3.5+1/10*sin(200*2*sin(40*\y)) + 1/25*cos(500*2*sin(40*\y))+0.7}, {2*sin(40*\y)})--({5.5-1/15*sin(-300*2*sin(40*\z)) + 1/20*cos(400*2*sin(40*\z))-0.7}, {2*sin(40*\z)}) -- ({5.5-1/15*sin(-300*2*sin(40*\z)) + 1/20*cos(400*2*sin(40*\z))}, {2*sin(40*\z)});
    }
    \foreach \y in {0.42, 0.61, 0.82, 1, 1.18, 1.37} {
        \pgfmathsetmacro{\colorval}{40 + 30 * \y}
        \draw[line width=0.6pt, color=violet!\colorval] ({5.5-1/15*sin(-300*2*sin(40*\y)) + 1/20*cos(400*2*sin(40*\y))-0.3}, {2*sin(40*\y)}) -- ({5.5-1/15*sin(-300*2*sin(40*\y)) + 1/20*cos(400*2*sin(40*\y))}, {2*sin(40*\y)});
    }
    \draw[->] (3.1, -0.2) -- (5.9, -0.2) node[right] {$x$};
    \draw[->] (3.3, -0.4) -- (3.3, 2.4) node[above] {$y$};
    \draw[thick] (3.3, -0.2) rectangle (5.7, 2.2);
    \draw[violet,-{Latex[length=1mm]}] (2.2, 1) -- (3.1, 1) node[midway, above] {$\widetilde{h}_\mathfrak{t}$};
    \draw[domain=0:2, samples=100, smooth, variable=\t, gray, line cap=round] plot ({3.5+1/10*sin(200*\t) + 1/25*cos(500*\t)}, {\t});
    \draw[domain=0:2, samples=100, smooth, variable=\t, gray, line cap=round] plot ({5.5-1/15*sin(-300*\t) + 1/20*cos(400*\t)}, {\t});
    \draw[gray, line cap=round] ({3.5+1/10*sin(200*2) + 1/25*cos(500*2)},2) -- ({5.5-1/15*sin(-300*2) + 1/20*cos(400*2)}, 2);
    \draw[gray, line cap=round] ({3.5+1/25},0) -- ({5.5+ 1/20}, 0);
    \draw[orange, dashed, very thin] (3.68,-0.2) -- (3.68,2.2);
    \draw[orange, dashed, very thin] (5.32,-0.2) -- (5.32,2.2);
\end{tikzpicture}

%% file: tikz/square.tex
\pgfdeclarepatternformonly{small dots}{\pgfqpoint{-1pt}{-1pt}}{\pgfqpoint{1pt}{1pt}}{\pgfqpoint{2pt}{2pt}}%
{
  \pgfpathcircle{\pgfqpoint{0pt}{0pt}}{0.3pt}
  \pgfusepath{fill}
}

\begin{tikzpicture}[scale=3]
    \clip (-1.2,-1.2) rectangle (1.2,1.1);
    \useasboundingbox (-1.2,-1.2) rectangle (1.2,1.1);
    \fill[yellow, opacity=0.4] (-0.48,1) rectangle (0.5,0.2);
    \fill[pattern=horizontal lines, pattern color=yellow!50!black, opacity=0.2] (-0.48,1) rectangle (0.5,0.2);
    \fill[yellow, opacity=0.4] (-0.48,-0.2) rectangle (0.5,-1);
    \fill[pattern=horizontal lines, pattern color=yellow!50!black, opacity=0.2] (-0.48,-0.2) rectangle (0.5,-1);
    \fill[yellow, opacity=0.4] (0.5,1) rectangle (1,0.5);
    \fill[pattern=horizontal lines, pattern color=yellow!50!black, opacity=0.2] (0.5,1) rectangle (1,0.5);
    \fill[yellow, opacity=0.4] (0.5,-0.5) rectangle (1,-1);
    \fill[pattern=horizontal lines, pattern color=yellow!50!black, opacity=0.2] (0.5,-0.5) rectangle (1,-1);
    \fill[yellow, opacity=0.4] (0.5,-0.5) rectangle (1,0.5);
    \fill[pattern=horizontal lines, pattern color=yellow!50!black, opacity=0.2] (0.5,-0.5) rectangle (1,0.5);
    \fill[green, opacity=0.4] (-1,-1) rectangle (-0.5,1);
    \fill[pattern=small dots, pattern color=green!50!black, opacity=0.2] (-1,-1) rectangle (-0.5,1);
    \fill[purple, opacity=0.6] (-0.48,-0.2) rectangle (0.5,0.2);
    \fill[pattern=north east lines, pattern color=purple!50!black, opacity=0.2] (-0.48,-0.2) rectangle (0.5,0.2);
    \fill[yellow, opacity=0.4] (-0.5,-1) rectangle (-0.48,1);
    \fill[pattern=horizontal lines, pattern color=yellow!50!black, opacity=0.2] (-0.5,-1) rectangle (-0.48,1);
    \draw[orange, line width=3pt] (0.5,-1) -- (0.5,1);
    \fill[orange, opacity=0.2] (0.48,-1) rectangle (0.52,1);
    \fill[pattern=horizontal lines, pattern color=orange!50!black, opacity=0.2] (0.48,-1) rectangle (0.52,1);
    \draw[red, line width=3pt] (0.5,-0.5) -- (0.5,0.5);
    \fill[red, opacity=0.2] (0.48,-0.5) rectangle (0.52,0.5);
    \fill[pattern=north east lines, pattern color=red!50!black, opacity=0.2] (0.48,-0.5) rectangle (0.52,0.5);
    \fill[white] (-0.95,0.9) rectangle (0.95,1);
    \fill[white] (-0.95,-1) rectangle (0.95,-0.9);
    \fill[white] (-1,-1) rectangle (-0.95,1);
    \fill[white] (0.95,-1) rectangle (1,1);
    \fill[blue, opacity=0.4] (-0.95,0.9) rectangle (0.95,1);
    \fill[blue, opacity=0.4] (-0.95,-1) rectangle (0.95,-0.9);
    \fill[blue, opacity=0.4] (-1,-1) rectangle (-0.95,1);
    \fill[blue, opacity=0.4] (0.95,-1) rectangle (1,1);
    \draw[black, thick] (-1,-1) rectangle (1,1);
    \draw[gray, dashed, thin] (-0.5,-1) -- (-0.5,1);
    \draw[gray, dashed, thin] (0,-1) -- (0,1);
    \draw[gray, dashed, thin] (0.5,-1) -- (0.5,1);
    \draw[gray, dashed, thin] (-1,-0.5) -- (1,-0.5);
    \draw[gray, dashed, thin] (-1,0) -- (1,0);
    \draw[gray, dashed, thin] (-1,0.5) -- (1,0.5);
    \draw[gray, thin] (-1,0.2) -- (1,0.2);
    \draw[gray, thin] (-1,-0.2) -- (1,-0.2);
    \node[below] at (-1,-1) {$-1$};
    \node[below] at (-0.5,-1) {$-\frac{1}{2}$};
    \node[below] at (0,-1) {$0$};
    \node[below] at (0.5,-1) {$\frac{1}{2}$};
    \node[below] at (1,-1) {$1$};
    \node[left] at (-1,-1) {$-1$};
    \node[left] at (-1,-0.5) {$-\frac{1}{2}$};
    \node[left] at (-1,0) {$0$};
    \node[left] at (-1,0.5) {$\frac{1}{2}$};
    \node[left] at (-1,1) {$1$};
    \node[left] at (-1,0.2) {$\delta$};
    \node[left] at (-1,-0.2) {$-\delta$};
\end{tikzpicture}

%% file: tikz/vogel.tex
\begin{tikzpicture}[scale=0.8, line cap=round,line join=round,>=triangle 45,x=1.0cm,y=1.0cm]
        \clip(-5.5,-4.5) rectangle (4.8,5.8);
        \draw [dotted, rotate around={90.:(-1.,0.5)},thick, color=magenta, fill=magenta,fill opacity=0.05] (-1.,0.5) ellipse (4.7cm and 4.cm);
        \fill[thick, color=teal,fill=teal,fill opacity=0.1] (-3.5,-2.5) -- (-3.5,3.5) -- (1.5,3.5) -- (1.5,-2.5) -- cycle;
        \fill[thick, color=teal,fill=teal,fill opacity=0.1] (-2.,-1.) -- (-2.,2.) -- (0.,2.) -- (0.,-1.) -- cycle;
        \draw [thick, color=teal] (-2.,-1.)-- (-2.,2.);
        \draw [thick, color=teal] (-2.,2.)-- (0.,2.);
        \draw [thick, color=teal] (0.,2.)-- (0.,-1.);
        \draw [thick, color=teal] (0.,-1.)-- (-2.,-1.);
        \draw [thick, color=teal] (-3.5,-2.5)-- (-3.5,3.5);
        \draw [thick, color=teal] (-3.5,3.5)-- (1.5,3.5);
        \draw [thick, color=teal] (1.5,3.5)-- (1.5,-2.5);
        \draw [thick, color=teal] (1.5,-2.5)-- (-3.5,-2.5);
        \draw [color=gray] (-1.71,-0.54)-- (-1.72,1.72);
        \draw [color=gray] (-1.72,1.72)-- (-0.25,1.65);
        \draw [color=gray] (-0.25,1.65)-- (-0.26,-0.69);
        \draw [color=gray] (-0.26,-0.69)-- (-1.71,-0.54);
        \draw [color=gray] (-1.71,-0.54)-- (-1.15,-2.41);
        \draw [color=gray] (-1.15,-2.41)-- (-0.26,-0.69);
        \draw [color=gray] (-0.26,-0.69)-- (1.18,0.51);
        \draw [color=gray] (1.18,0.51)-- (-0.25,1.65);
        \draw [color=gray] (-1.15,-2.41)-- (1.10,-1.74);
        \draw [color=gray] (1.10,-1.74)-- (-0.26,-0.69);
        \draw [color=gray] (1.10,-1.74)-- (1.18,0.51);
        \draw [color=gray] (-1.72,1.72)-- (-0.92,3.23);
        \draw [color=gray] (-0.92,3.23)-- (-0.25,1.65);
        \draw [color=gray] (-0.25,1.65)-- (1.00,3.00);
        \draw [color=gray] (1.00,3.00)-- (1.18,0.51);
        \draw [color=gray] (1.00,3.00)-- (-0.92,3.23);
        \draw [color=gray] (-1.71,-0.54)-- (-3.27,0.89);
        \draw [color=gray] (-3.27,0.89)-- (-1.72,1.72);
        \draw [color=gray] (-1.72,1.72)-- (-2.78,3.14);
        \draw [color=gray] (-2.78,3.14)-- (-3.27,0.89);
        \draw [color=gray] (-2.78,3.14)-- (-0.92,3.23);
        \draw [color=gray] (-1.71,-0.54)-- (-3.21,-1.66);
        \draw [color=gray] (-3.21,-1.66)-- (-3.27,0.89);
        \draw [color=gray] (-3.21,-1.66)-- (-1.15,-2.41);
        \draw [color=gray] (-2.78,3.14)-- (-1.50,4.72);
        \draw [color=gray] (-1.50,4.72)-- (-0.92,3.23);
        \draw [color=gray] (-0.92,3.23)-- (0.47,5.13);
        \draw [color=gray] (0.47,5.13)-- (-1.50,4.72);
        \draw [color=gray] (1.00,3.00)-- (0.47,5.13);
        \draw [color=gray] (0.47,5.13)-- (2.62,3.93);
        \draw [color=gray] (2.62,3.93)-- (1.00,3.00);
        \draw [color=gray] (1.18,0.51)-- (3.47,1.48);
        \draw [color=gray] (3.47,1.48)-- (1.00,3.00);
        \draw [color=gray] (3.47,1.48)-- (2.64,-0.67);
        \draw [color=gray] (2.64,-0.67)-- (1.18,0.51);
        \draw [color=gray] (1.10,-1.74)-- (2.64,-0.67);
        \draw [color=gray] (1.10,-1.74)-- (2.30,-3.40);
        \draw [color=gray] (2.30,-3.40)-- (2.64,-0.67);
        \draw [color=gray] (2.30,-3.40)-- (-0.01,-3.79);
        \draw [color=gray] (-0.01,-3.79)-- (1.10,-1.74);
        \draw [color=gray] (-0.01,-3.79)-- (-1.15,-2.41);
        \draw [color=gray] (-0.01,-3.79)-- (-2.83,-4.25);
        \draw [color=gray] (-2.83,-4.25)-- (-1.15,-2.41);
        \draw [color=gray] (-2.83,-4.25)-- (-3.21,-1.66);
        \draw [color=gray] (-2.83,-4.25)-- (-4.74,-2.73);
        \draw [color=gray] (-4.74,-2.73)-- (-3.21,-1.66);
        \draw [color=gray] (-3.21,-1.66)-- (-4.65,-0.23);
        \draw [color=gray] (-4.65,-0.23)-- (-3.27,0.89);
        \draw [color=gray] (-3.27,0.89)-- (-5.00,2.18);
        \draw [color=gray] (-5.00,2.18)-- (-2.78,3.14);
        \draw [color=gray] (-2.78,3.14)-- (-3.79,4.68);
        \draw [color=gray] (-3.79,4.68)-- (-5.00,2.18);
        \draw [color=gray] (-3.79,4.68)-- (-1.50,4.72);
        \draw [color=gray] (-5.00,2.18)-- (-4.65,-0.23);
        \draw [color=gray] (-4.74,-2.73)-- (-4.65,-0.23);
        \draw [color=gray] (0.47,5.13)-- (2.80,5.48);
        \draw [color=gray] (2.80,5.48)-- (2.62,3.93);
        \draw [color=gray] (2.80,5.48)-- (4.41,3.84);
        \draw[rounded corners=4pt, gray, double=yellow, double distance=4pt, opacity=0.8, line cap=butt]
           (3.44,3.89) 
            -- (3.4,4.1)
            -- (3.12,4.30)
            -- (2.44,4.56)
            -- (1.87,4.59)
            -- (1.22,4.31)
            -- (0.94,3.99)
            -- (0.53,3.59)
            -- (0.14,3.25)
            -- (-0.24,2.71)
            -- (-0.84,2.19)
            -- (-1.08,1.80)
            -- (-1.1,1.4);
        \draw [very thick] (4.41,3.84)-- (3.47,1.48);
        \draw [very thick] (2.62,3.93)-- (4.41,3.84);
        \draw [very thick] (3.47,1.48)-- (2.62,3.93);
        \draw node at (3.5,3.1) {$P$};
        \draw[teal] node at (-1,0.8) {$N$};
        \draw[teal] node at (-2.6,0.8) {$\widehat{N}$};
        \draw[magenta] node at (-4.2,0.8) {$U$};
    \end{tikzpicture}

%% file: tikz/annulus.tex
\begin{tikzpicture}[scale=1]
    \begin{axis}[
        axis lines=none,
        xlabel={},
        ylabel={},
        ymin=-3.5, ymax=6,
        xmin=-5, xmax=2.5,
        domain=-6:6,
        samples=100,
        clip=true,
        ]
        \fill[blue!10] (-6,-3) rectangle (6, 4);
        \fill[pattern=north east lines, pattern color=gray!40] (-6,-0.28) rectangle (6,1.76); 
        \addplot[domain=-6:6, very thick, blue] {4};
        \addplot[domain=-6:6, very thick, blue] {-3};
        \addplot[domain=-6:6, thick, magenta] {2.5};
        \addplot[domain=-6:6, thick, magenta] {-1};
        \foreach \i in {-7, ..., 6} {
            \addplot[domain=\i-1.5:\i+1.5, samples=100] {tanh(-x+\i)+3.5};
        }
        \foreach \i in {-7, ..., 6} {
            \addplot[domain=\i-0.5:\i+1, samples=100] {-0.5*tanh(-x+\i)+2};
        }
        \foreach \i in {-7, ..., 6} {
            \addplot[domain=\i-1:\i-0.5, samples=100, dotted] {-0.5*tanh(-x+\i)+2};
        }
        \foreach \i in {-7, ..., 6} {
            \addplot[domain=\i+1:\i+2, samples=100] {3*tanh(-x+\i)+2};
        }
        \foreach \i in {-7, ..., 6} {
            \addplot[domain=\i+0.75:\i+1, samples=100, dotted] {3*tanh(-x+\i)+2};
        }
        \foreach \i in {-7, ..., 6} {
            \addplot[domain=\i-1.5:\i+1.5, samples=100] {-1.5*tanh(-x+\i)-2.5};
        }
        \node[text=magenta, fill=blue!10, rounded corners=5pt, inner sep=0.5pt] at (2, 3.1) {$\gamma^+$};
        \node[text=magenta, fill=blue!10, rounded corners=5pt, inner sep=0.5pt] at (2, -1.5) {$\gamma^-$};
        \node[text=blue] at (-1.25,0.8) {$A$};
    \end{axis}
\end{tikzpicture}

%% file: tikz/function.tex
\begin{tikzpicture}[scale=3.5]
        \draw[->] (-1, 0) -- (1, 0);
        \draw[->] (0, -1) -- (0, 1);
        \draw[gray] (-1, -1) rectangle (1, 1);
        \draw[red, dashed] (-1,-1)--(1,1);
        \draw[teal] (-1, -0.8) 
            -- (1, -0.4);
        \draw[blue, thick, rounded corners=13pt]
            (-1, -1) -- (-0.75, -0.75) 
            -- (0.7308, -0.4538) 
            -- (0.9, 0.9) 
            -- (1, 1);
        \draw (-1, -0.07) node[right] {$-1$};
        \draw (1, -0.07) node[left] {$1$};
        \draw (-0.08, -1) node[above] {$-1$};
        \draw (-0.05, 1) node[below] {$1$};
        \draw (0, -0.05) node[right] {$0$};
        \draw (-0.6, 0.03)--(-0.6, -0.03) node[below]{$-1+\sigma$};
        \draw (0.6, 0.03)--(0.6, -0.03) node[below]{$1-\sigma$};
        \draw[dotted] (-0.6, -1)--(-0.6,1);
        \draw[dotted] (0.6, -1)--(0.6,1);
    \end{tikzpicture}